\documentclass[11pt]{amsart}
\usepackage{amsmath, amsthm, amssymb}
\usepackage[margin=2.5cm]{geometry}
\usepackage{graphicx}
\usepackage{fourier}
\usepackage{hyperref}
\usepackage{todonotes}
\usepackage{stmaryrd}
\usepackage{enumerate}
\usepackage{url}

\theoremstyle{plain}
\newtheorem{theorem}{Theorem}[section]

\newtheorem{proposition}[theorem]{Proposition}
\newtheorem{lemma}[theorem]{Lemma}
\newtheorem{conjecture}[theorem]{Conjecture}
\newtheorem{corollary}[theorem]{Corollary}
\theoremstyle{remark}
\newtheorem{remark}[theorem]{Remark}
\newtheorem{example}[theorem]{Example}

\author{Carmelo Cisto}
\address[Cisto]{Universit\'{a} di Messina, Dipartimento di Scienze  Matematiche e Informatiche, Scienze Fisiche e Scienze della Terra, Viale Ferdinando Stagno D'Alcontres 31, 98166 Messina, Italy}
\email{carmelo.cisto@unime.it}

\author{Pedro A. García-Sánchez}
\address[Garc\'ia-S\'anchez]{Departamento de \'Algebra and IMAG, Universidad de Granada, E-18071 Granada, Espa\~na}
\email{pedro@ugr.es}

\author{David Llena}
\address[Llena]{Departamento de Matem\'{a}ticas, Universidad de Almeria, E-04120 Almeria,  Espa\~na}
\email{dllena@ual.es}

\title{Ideal extensions of free commutative monoids}

      \keywords{Gap absorbing monoid, ideal extension of a free monoid, sets of lenghts of factorizations, catenary degree, $\omega$-primality}
		
		\subjclass[2020]{20M14, 20M12, 20M13}

\begin{document}

\begin{abstract}
We introduce a new family of monoids, which we call gap absorbing monoids. Every gap absorbing monoid is an ideal extension of a free commutative monoid. For a gap absorbing monoid $S$ we study its set of atoms and Betti elements, which allows us to show that the catenary degree of $S$ is at most four and that the set of lengths of any element in $S$ is an interval. We also give bounds for the $\omega$-primality of any ideal extension of a free commutative monoid. For ideal extensions $S$ of $\mathbb{N}^d$, with $d$ a positive integer, we show that $\omega(S)$ is finite if and only if $S$ has finitely many gaps.
\end{abstract}

\maketitle

\centerline{\emph{Dedicated to the memory of N. Baeth}}

\section*{Introduction}
  
Let $\mathbb{N}$ be the set of non-negative integers. Let $I$ be a non-empty subset of $\mathbb{N}$. The set $\mathbb{N}^{(I)}$ is the set of all sequences $\mathbf{n}=(n_i)_{i\in I}$ such that $n_i=0$ but for finitely many $i\in I$, that is, if we define the \emph{support} of $\mathbf{n}$ as $\operatorname{Supp}(\mathbf{n})=\{ i\in I : n_i\neq 0\}$, then the cardinality of the support of $\mathbf{n}$ is finite. 

We denote by $\mathbf{e}_i$ the sequence with a one in the $i$th position and having the rest of the entries equal to zero. Notice that $\mathbf{n}=(n_i)_{i\in I}=\sum_{i\in I} n_i\mathbf{e}_i$. The monoid $(\mathbb{N}^{(I)},+)$ (with $+$ the componentwise addition) can be seen as the free monoid on the set $\{\mathbf{e}_i\}_{i\in I}$.
		
For $X,Y\subseteq \mathbb{N}^{(I)}$, set \[X+Y=\{x+y : x\in X, y\in Y\},\] and for $i> 0$, $iX$ is defined inductively by $i X=(i-1)X+X$ ($0X=\{\mathbf{0}\}$). We denote by $\langle X\rangle$ the monoid generated by $X$, that is,
\[\langle X\rangle = \left\{ \sum_{i=1}^n \lambda_i \mathbf{x}_i : n\in \mathbb{N}, \lambda_i\in \mathbb{N}, \mathbf{x}_i\in X, \text{ for all } i \in \{1,\dots,n\}\right\}=\bigcup_{n\in \mathbb{N}} nX.\]
We also say that $X$ generates $\langle X\rangle$ (or that $X$ is a generating set of $\langle X\rangle$).

A subset $E$ of $\mathbb{N}^{(I)}$ is an \emph{ideal} if $E+\mathbb{N}^{(I)}\subseteq E$. Notice that in this case $(E\cup\{\mathbf{0}\},+)$ is a submonoid of $(\mathbb{N}^{(I)},+)$. We say that a submonoid $S$ of $\mathbb{N}^{(I)}$ is an \emph{ideal extension} of $\mathbb{N}^{(I)}$ if $S=E\cup\{\mathbf 0\}$ for some ideal $E$ of $\mathbb{N}^{(I)}$. If $\mathbb{N}^{(I)}\setminus S$ has finitely many elements, then in \cite{Baeth} these monoids are called finite-complement ideals. One of the goals of this manuscript is to extend some of the results presented in that paper concerning non-unique factorization invariants.

Given $\mathbf{m}=(m_i)_{i\in I},\mathbf{n}=(n_i)_{i\in I}\in \mathbb{N}^{(I)}$, we write $\mathbf{m}\le \mathbf{n}$ if $m_i\le n_i$ for all $i\in I$ (the usual partial order on $\mathbb{N}^{(I)}$), that is, $\mathbf{m}\le \mathbf{n}$ if and only if $\mathbf{n}-\mathbf{m}=(n_i-m_i)_{i\in I}\in \mathbb{N}^{(I)}$. For $\mathbf{m},\mathbf{n}\in \mathbb{N}^{(I)}$, we denote \[\llbracket \mathbf{m},\mathbf{n} \rrbracket=\{\mathbf{x}\in \mathbb{N}^{(I)} : \mathbf{m}\leq \mathbf{x}\leq \mathbf{n}\}.\] 
In particular, $\operatorname{B}(\mathbf{n})=\llbracket \mathbf{0},\mathbf{n}\rrbracket$, where $\mathbf{0}$ is the sequence in $\mathbb{N}^{(I)}$ having all its entries equal to zero. Note that the cardinality of $\operatorname{B}(\mathbf{n})$ is $\lvert \operatorname{B}(\mathbf{n})\rvert=\prod_{i\in I}(n_i+1)=\prod_{i\in \operatorname{Supp}(\mathbf{n})}(n_i+1)$, which is finite since the support of $\mathbf{n}$ is finite. 

Let $X\subseteq \mathbb{N}^{(I)}$. We say that $X$ is \emph{closed under intervals}  (or \emph{$\le$-convex}) if for any $\mathbf{u},\mathbf{v}\in X$ with $\mathbf{u}\le \mathbf{v}$, we have that $\llbracket \mathbf{u},\mathbf{v}\rrbracket\subseteq X$. 

Observe that if $S$ is an ideal extension of $\mathbb{N}^{(I)}$, then $S$ is closed under intervals, and so is $\mathcal{H}(S)=\mathbb{N}^{(I)}\setminus S$, the \emph{set of gaps} of $S$. Let $S^*=S\setminus\{\mathbf{0}\}$. It is not difficult to check that the set $\mathcal{A}(S)=S^*\setminus (S^*+S^*)$ is also closed under intervals. The elements in $\mathcal{A}(S)$ are called \emph{atoms}. This definition agrees with the general definition of atom on a monoid, since $S$, apart from being commutative, is also reduced (the only unit is the identity element $\mathbf{0}$) and cancellative ($\mathbf{a}+\mathbf{b}= \mathbf{a}+\mathbf{c}$ in $S$ implies $\mathbf{a}=\mathbf{c}$). As a matter of fact, it easily follows that $S$ is \emph{atomic}, that is, $S=\langle \mathcal{A}(S)\rangle$. 

We conjecture that for every ideal extension $S$ of $\mathbb{N}^{(I)}$, the set $2\mathcal{A}(S)$ is also closed under intervals. This was actually the motivation to the introduction of the concept of gap absorbing monoid. A submonoid $S$ of $\mathbb{N}^{(I)}$ is \emph{gap absorbing} if (1) whenever the sum of two gaps is in $S$, then this sum is either an atom or the sum of two atoms, and (2) the sum of a gap and an atom is either an atom or the sum of two atoms.

Many factorization invariants of an atomic monoid can be determined via a presentation of the mo\-noid, and thus it is central to know what are the Betti elements of the monoid (see for instance \cite{overview}). In Section~\ref{sec:betti-elements}, we show that in a gap absorbing monoid, Betti elements are sums of at most three atoms. With this we are able to prove that the catenary degree of a gap absorbing monoid is at most four (see Section~\ref{sec:catenary-degree}), and that the set of lengths of factorizations of every element in a gap absorbing monoid are intervals (Section~\ref{sec:delta}).

The $\omega$-primality of a monoid does not depend completely on the presentation of the monoid. We provide lower and upper bounds for the $\omega$-primality of an ideal extension of a free commutative monoid. This is done in Section~\ref{sec:omega-primality}, where we first determine what are the $\omega$-primalities of the atoms in this kind of monoids. We also prove that for the case of ideal extensions of $\mathbb{N}^d$, with $d$ a positive integer, the $\omega$-primality of $S$ is finite if and only if the monoid has finitely many gaps.

Section~\ref{sec:back-slash} introduces a family of ideal extensions of free commutative monoids that are gap absorbing. The definition of this class of monoids was partially inspired by the family of $T$-graded generalized numerical semigroups, introduced in \cite{classesGNS}, that is, monoids in $\mathbb{N}^d$ consisting of all elements having ``degree'' in a numerical semigroup $T$. Considering this notion on $\mathbb{N}^{(I)}$ and extending it further, focusing on a subset of coordinates, we introduce the family of \emph{backslash monoids}. For any monoid in this family we show that the Betti elements are always sums of two atoms and thus the catenary degree of the monoid is at most three. We also compute explicitly the elasticity, length density and the $\omega$-primality of these monoids.

We prove in Section~\ref{sec:N2} that every ideal extension of $\mathbb{N}^2$ is gap absorbing, and that every Betti element in such monoids is the sum of two atoms (which forces the catenary degree to be at most three).

Computer experiments suggest that every ideal extension of a free commutative monoid is gap absorbing. Also, we did not find any example of gap absorbing monoid whose Betti elements cannot be expressed as the sum of two atoms. Computing Betti elements is costly, since the number of atoms in this kind of monoids is extremely high, and the calculation of a presentation grows exponentially in the number of atoms. The main motivation that led us to start this research was a series of conversations with N. Baeth, who was interested in proving that the set of lengths of factorizations of elements in complement-finite ideals \cite{Baeth} are intervals. To this end, our approach was to prove that the maximum of the Delta sets should be at most one, and this maximum is reached at a Betti element of the monoid \cite{maxdelta}. We produced many examples with the \texttt{GAP} \cite{gap} package \texttt{numericalsgps} \cite{numericalsgps}, via an implementation of the algorithm proposed in \cite{gsonw}. So for us, it was a priority to see which restrictions we should impose to an element in an ideal extension of a free monoid to become a Betti element. 

It is also known that proving that the catenary degree of a monoid is at most three also forces every set of lengths of factorizations to be an interval, and this is precisely we also focused on finding upper bounds for the catenary degree. For the calculation of the catenary degree we also used the algorithm proposed in \cite{cgspr} which is also implemented in \cite{numericalsgps}, but, as mentioned above, the high number of atoms makes the problem computationally intractable for higher dimensions.

The calculation of the $\omega$-primality is performed via the sets of non-negative integer solutions of some systems of linear Diophantine equations (again we use \cite{numericalsgps} to make these computations). Even with simple examples, the process may take hours and even days. So it was worth trying to find good lower and upper bounds. 

Before we start with the first section, let us fix some notation that we are going to use extensively later. For $\mathbf{a}=(a_i)_{i\in I}$ and $\mathbf{b}=(b_i)_{i\in I}$  in $\mathbb{N}^{(I)}$, set 
\begin{align*}
\mathbf{a} \vee \mathbf{b} &=(\max(a_i,b_i))_{i\in I}\in \mathbb{N}^{(I)},\\
\mathbf{a} \wedge \mathbf{b} & =(\min(a_i,b_i))_{i\in I}\in \mathbb{N}^{(I)}.
\end{align*}
For $\mathbf{a}=(a_i)_{i\in I}\in \mathbb{N}^{(I)}$, set 
\[
\lVert \mathbf{a}\rVert_1=\sum_{i\in I}a_i.
\]

\section{Ideal extensions of free monoids}

Let $S$ be a submonoid of $\mathbb{N}^{(I)}$, with $I$ a set of non-negative integers. Recall that $S$ is an ideal extension of $\mathbb{N}^{(I)}$ if $S^*$ is an ideal of $\mathbb{N}^{(I)}$.
Set 
\[\mathcal{M}(S)=\operatorname{Minimals}_\le(S^*).\]
Then, $S$ is an ideal extension of $\mathbb{N}^{(I)}$ if and only if $S=\{\mathbf{0}\}\cup(\mathcal{M}(S)+\mathbb{N}^{(I)})$. Observe that $\mathcal{M}(S)$ is an antichain with respect to $\le$, that is, all its elements are incomparable with respect to the usual partial order on $\mathbb{N}^{(I)}$. Also,
\[
\mathcal{M}(S)\subseteq \mathcal{A}(S).
\]


Observe that $S$ is an ideal extension of $\mathbb{N}$ if and only if $S$ is an ordinary numerical semigroup, that is, $S=\{0\}\cup(m+\mathbb{N})$ for some positive integer $m$.

\begin{proposition}\label{prop:char-ideal-gaps}
    Let $S$ be a submonoid of $\mathbb{N}^{(I)}$ with $I$ a non-empty set of non-negative integers. Then $S$ is an ideal extension of $\mathbb{N}^{(I)}$ if and only if $\mathcal{M}(S)+\mathcal{H}(S)\subseteq S$.
\end{proposition}
\begin{proof}
    Suppose that $S$ is an ideal extension of $\mathbb{N}^{(I)}$. Then, $\mathcal{M}(S)+\mathcal{H}(S)\subseteq S^*+\mathbb{N}^{(I)}\subseteq S^*$.

    Now suppose that $\mathcal{M}(S)+\mathcal{H}(S)\subseteq S$, and let us prove that $S^*+\mathbb{N}^{(I)}\subseteq S^*$. Take $\mathbf{s}\in S^*$ and $\mathbf{x}\in \mathbb{N}^{(I)}$. There exists $\mathbf{m}\in \mathcal{M}(S)$ with $\mathbf{m}\le \mathbf{s}$, and so $\mathbf{s}=\mathbf{m}+\mathbf{y}$ for some $\mathbf{y}\in \mathbb{N}^{(I)}$. Hence $\mathbf{s}+\mathbf{x}=\mathbf{m}+\mathbf{y}+\mathbf{x}$. If $\mathbf{y}+\mathbf{x}\in S$, then as $S$ is a reduced monoid, $\mathbf{m}+\mathbf{y}+\mathbf{x}\in S^*$. If $\mathbf{y}+\mathbf{x}\not\in S$, then $\mathbf{y}+\mathbf{x}\in \mathcal{H}(S)$, and by hypothesis $\mathbf{m}+\mathbf{y}+\mathbf{x}\in S$.
\end{proof}


Let $I$ be a subset of $\mathbb{N}$ with cardinality at least two, and let $i\in I$. For an element $\mathbf{v}=(v_j)_{j\in I\setminus\{i\}}\in \mathbb{N}^{(I\setminus\{i\})}$, set $(\mathbf{v},t)=t\mathbf{e}_i+\sum_{I\setminus\{i\}}v_j\mathbf{e}_j\in \mathbb{N}^{(I)}$.
Inspired by \cite{Li}, for $S$ an ideal extension of $\mathbb{N}^{(I)}$, $\mathbf{v}\in \mathbb{N}^{(I\setminus\{i\})}$, and a positive integer $k$, define $\pi_{\mathbf{v}}^k\in \mathbb{N}\cup\{\infty\}$ as follows: if $(\mathbf{v},t)\in k\mathcal{M}(S)+\mathbb{N}^{(I)}$ for some non-negative integer $t$, then 
\[
\pi^k_{\mathbf{v}}=\min\{t\in \mathbb{N}: (\mathbf{v},t)\in k\mathcal{M}(S)+\mathbb{N}^{(I)}\},
\]
and $\pi^k_{\mathbf{v}}=\infty$ otherwise. We extent addition from $\mathbb{N}$ to $\mathbb{N}\cup\{\infty\}$ as follows: $a+\infty=\infty$ for all $a\in \mathbb{N}\cup\{\infty\}$. Also, we extend $\le$ on $\mathbb{N}$ to $\mathbb{N}\cup\{\infty\}$ accordingly: $a\le \infty$ for all $a\in \mathbb{N}\cup\{\infty\}$.

Observe that for every positive integer $k$,
\[
k(\mathcal{M}(S)+\mathbb{N}^{(I)})=k\mathcal{M}(S)+\mathbb{N}^{(I)}.
\]

The following two technical lemmas can be viewed as generalizations of \cite[Proposition 3.8]{Li} and \cite[Proposition 3.10]{Li}, respectively.

\begin{lemma}\label{lem:pi-decreasing}
Let $\mathbf{v},\mathbf{w}\in \mathbb{N}^{(I\setminus\{i\})}$. If $\mathbf{v}\le \mathbf{w}$, then $\pi_{\mathbf{w}}^k\le \pi_{\mathbf{v}}^k$ for every positive integer $k$.
\end{lemma}
\begin{proof}
If there is no integer $t$ such that $(\mathbf{v},t)\in k\mathcal{M}(S)+\mathbb{N}^{(I)}$, then $\pi_{\mathbf{v}}^k=\infty$, and the result follows trivially. If there is some $t\in\mathbb{N}$ such that $(\mathbf{v},t)\in k\mathcal{M}(S)+\mathbb{N}^{(I)}$, then as $\mathbf{w}=\mathbf{v}+\mathbf{x}$ for some $\mathbf{x}\in \mathbb{N}^{(I\setminus\{i\})}$, we have that $(\mathbf{w},t)=(\mathbf{v},t)+(\mathbf{x},0)\in k\mathcal{M}(S)+\mathbb{N}^{(I)}$. It easily follows that $\pi_{\mathbf{w}}^k\le \pi_{\mathbf{v}}^k$.
\end{proof}

\begin{lemma}\label{lem:pi1+pi1}
    Let $\mathbf{v}\in \mathbb{N}^{(I\setminus\{i\})}$. Then $\pi_{\mathbf{v}}^2=\min\{ \pi_{\mathbf{a}}^1+\pi_{\mathbf{b}}^1 : \mathbf{a},\mathbf{b}\in \mathbb{N}^{(I\setminus\{i\})}, \mathbf{v}=\mathbf{a}+\mathbf{b}\}$.
\end{lemma}
\begin{proof}
    Let $\mathbf{v},\mathbf{a},\mathbf{b}\in \mathbb{N}^{(I\setminus\{i\})}$ such that $\mathbf{a}+\mathbf{b}=\mathbf{v}$. If either $\pi_{\mathbf{a}}^1$ or $\pi_{\mathbf{b}}^1$ is infinity, then the inequality $\pi_{\mathbf{v}}^2\le \pi_{\mathbf{a}}^1+\pi_{\mathbf{b}}^1$ holds trivially. So, let us suppose that both $s=\pi_{\mathbf{a}}^1$ and $t=\pi_{\mathbf{b}}^1$ are in $\mathbb{N}$. This, in particular, implies that $(\mathbf{a},s)$ and $(\mathbf{b},t)$ are in $\mathcal{M}(S)+\mathbb{N}^{(I)}$, and consequently $(\mathbf{v},s+t)\in 2(\mathcal{M}(S)+\mathbb{N}^{(I)})=2\mathcal{M}(S)+\mathbb{N}^{(I)}$. This means that $\pi_{\mathbf{v}}^2 \le s+t$. Thus, $\pi_{\mathbf{v}}^2\le \min\{ \pi_{\mathbf{a}}^1+\pi_{\mathbf{b}}^1 : \mathbf{a},\mathbf{b}\in \mathbb{N}^{(I\setminus\{i\})}, \mathbf{v}=\mathbf{a}+\mathbf{b}\}$. 

    If $\pi_{\mathbf{v}}^2=\infty$, then by the above argument, $\min\{ \pi_{\mathbf{a}}^1+\pi_{\mathbf{b}}^1 : \mathbf{a},\mathbf{b}\in \mathbb{N}^{(I\setminus\{i\})}, \mathbf{v}=\mathbf{a}+\mathbf{b}\}=\infty$. If $r=\pi_{\mathbf{v}}^2\in \mathbb{N}$, then $(\mathbf{v},r)\in 2\mathcal{M}(S)+\mathbb{N}^{(I)}=2(\mathcal{M}(S)+\mathbb{N}^{(I)})$, and so, there exists $(\mathbf{u},s),(\mathbf{w},t)\in \mathcal{M}(S)+\mathbb{N}^{(I)}$ such that $(\mathbf{v},r)=(\mathbf{u},s)+(\mathbf{w},t)$. Thus, $r=s+t\ge \pi_{\mathbf{u}}^1+\pi_{\mathbf{w}}^1\ge \min\{ \pi_{\mathbf{a}}^1+\pi_{\mathbf{b}}^1 : \mathbf{a},\mathbf{b}\in \mathbb{N}^{(I\setminus\{i\})}, \mathbf{v}=\mathbf{a}+\mathbf{b}\}$.  
\end{proof}

\begin{lemma}\label{lem:pi1-less-pi2}
    Let $\mathbf{v}\in \mathbb{N}^{(I\setminus\{i\})}$. Suppose that $\pi_{\mathbf{v}}^1<\infty$ and that $\pi_{\mathbf{v}}^2\neq 0$. Then, $\pi_{\mathbf{v}}^1<\pi_{\mathbf{v}}^2$.
\end{lemma}
\begin{proof}
    From $2\mathcal{M}(S)+\mathbb{N}^{(I)}\subseteq \mathcal{M}(S)+\mathbb{N}^{(I)}$, we deduce that $\pi_{\mathbf{v}}^1\le \pi_{\mathbf{v}}^2$. If $\pi_{\mathbf{v}}^1= \pi_{\mathbf{v}}^2$, then by Lemma~\ref{lem:pi1+pi1}, we have that $\pi_{\mathbf{v}}^1=\min\{ \pi_{\mathbf{a}}^1+\pi_{\mathbf{b}}^1 : \mathbf{a},\mathbf{b}\in \mathbb{N}^{(I\setminus\{i\})}, \mathbf{v}=\mathbf{a}+\mathbf{b}\}$. Take $\mathbf{a},\mathbf{b}\in \mathbb{N}^{(I\setminus\{i\})}$ such that $\mathbf{v}=\mathbf{a}+\mathbf{b}$ and $\pi_{\mathbf{v}}^1=\pi_{\mathbf{a}}^1+\pi_{\mathbf{b}}^1$. Then, $\mathbf{a}\le \mathbf{v}$ and $\mathbf{b}\le \mathbf{v}$, which, in light of Lemma~\ref{lem:pi-decreasing}, implies $\pi_{\mathbf{a}}^1\ge \pi_{\mathbf{v}}^1$ and $\pi_{\mathbf{b}}^1\ge \pi_{\mathbf{v}}^1$. But then $\pi_{\mathbf{v}}^1=\pi_{\mathbf{a}}^1+\pi_{\mathbf{b}}^1\ge 2\pi_{\mathbf{v}}^1$, forcing $\pi_{\mathbf{v}}^1$ to be zero, in contradiction with $0\neq \pi_{\mathbf{v}}^2=\pi_{\mathbf{v}}^ 1$.
\end{proof}

With the above technical lemmas, we can describe the set of gaps and atoms of a gap absorbing monoid in terms of $\pi^1$ and $\pi^2$. We provide some extra properties that will be used later in Section~\ref{sec:N2}.

\begin{proposition}\label{prop:gaps-atoms-pi}
Let $S$ be an ideal extension of $\mathbb{N}^{(I)}$ and let \[A=\{\mathbf{v}\in \mathbb{N}^{(I\setminus\{i\})}: \pi_{\mathbf{v}}^1<\infty, \pi^2_{\mathbf{v}}\neq 0\}.\] Then,
\begin{enumerate}
    \item the set $A$ is closed under intervals,
    \item for every $\mathbf{z}\in A+A$, the set $\mathbf{z}_A=\{ \mathbf{v}\in A : \mathbf{z}-\mathbf{v}\in A\}$ is closed under intervals,
    \item for every $\mathbf{v}\in A$, if $\mathbf{v}+\mathbf{e}_j\in A$ for some $j\neq i$, then $\pi_{\mathbf{v}+\mathbf{e}_j}^2\ge \pi_{\mathbf{v}}^1$,
    \item $\mathcal{H}(S)=\{(\mathbf{v},x_\mathbf{v})\in \mathbb{N}^{(I)} : x_\mathbf{v}< \pi_{\mathbf{v}}^1\}$, and
    \item $\mathcal{A}(S)=\{(\mathbf{v},x_\mathbf{v})\in \mathbb{N}^{(I)}: \mathbf{v}\in A,  \pi^1_{\mathbf{v}}\le x_\mathbf{v} < \pi^2_{\mathbf{v}}\}$.
\end{enumerate}
\end{proposition}
\begin{proof}
    Notice that if $\mathbf{v}, \mathbf{w}\in A$ with $\mathbf{v}\le \mathbf{w}$, and $\mathbf{u}\in \llbracket \mathbf{v},\mathbf{w}\rrbracket$, then by Lemma~\ref{lem:pi-decreasing} we have that $\pi_{\mathbf{u}}^1\le \pi_{\mathbf{v}}^1<\infty$ and $\pi_{\mathbf{u}}^2\ge \pi_{\mathbf{w}}^2>0$. Thus, $\mathbf{u}\in A$.

    Now, take $\mathbf{v},\mathbf{w}\in \mathbf{z}_A$, with $\mathbf{v}\le \mathbf{w}$, and any $\mathbf{u}\in \llbracket \mathbf{v},\mathbf{w}\rrbracket$. Then $\mathbf{z}-\mathbf{w}\le \mathbf{z}-\mathbf{u}\le \mathbf{z}-\mathbf{v}$. As $\mathbf{z}-\mathbf{w}\in A$ and $\mathbf{z}-\mathbf{v}\in A$, and $A$ is closed under intervals, we deduce that both $\mathbf{u}$ and $ \mathbf{z}-\mathbf{u}$ are in $A$, which means that $\mathbf{u}\in \mathbf{z}_A$.

    Suppose that $\mathbf{v}\in A$ and $\mathbf{w}=\mathbf{v}+\mathbf{e}_j\in A$. Let $\mathbf{a},\mathbf{b}\in \mathbb{N}^{(I\setminus\{i\})}$ be such that $\mathbf{w}=\mathbf{a}+\mathbf{b}$. Then either $\mathbf{a}\le \mathbf{v}$ or $\mathbf{b}\le \mathbf{v}$. Thus, by Lemma~\ref{lem:pi-decreasing}, we deduce that $\pi_{\mathbf{a}}^1+\pi_{\mathbf{b}}^1\ge \pi_{\mathbf{v}}^1$. In light of Lemma~\ref{lem:pi1+pi1}, we deduce that $\pi_{\mathbf{w}}^2\ge \pi_{\mathbf{v}}^1$.
    
    Given $(\mathbf{v},x_\mathbf{v})\in \mathbb{N}^{(I)}$, we have that $(\mathbf{v},x_\mathbf{v})\in S$  if and only if $x_\mathbf{v}\ge \pi_{\mathbf{v}}^1$. From this fact we deduce the equality concerning $\mathcal{H}(S)$.
    
    Notice that $\mathcal{A}(S)=S^*\setminus(S^*+S^*)=(\mathcal{M}(S)+\mathbb{N}^{(I)})\setminus 2(\mathcal{M}(S)+\mathbb{N}^{(I)}) = (\mathcal{M}(S)+\mathbb{N}^{(I)})\setminus (2\mathcal{M}(S)+\mathbb{N}^{(I)})$. Let $(\mathbf{v},t)\in \mathbb{N}^{(I)}$. Then, $(\mathbf{v},t)\in \mathcal{A}(S)$ if and only if $\pi_\mathbf{v}^1\le t<\pi_\mathbf{v}^2$. 
\end{proof}

\section{Gap absorbing monoids}

Let $S$ be a submonoid of $(\mathbb{N}^{(I)},+)$, with $\emptyset \neq I\subseteq \mathbb{N}$. 
We say that $S$ is a \emph{gap absorbing monoid} if 
\begin{enumerate}[({GA}1)]
\item$2\mathcal{H}(S) \subseteq \mathcal{H}(S)\cup \mathcal{A}(S)\cup 2 \mathcal{A}(S)$, and 
\item $\mathcal{H}(S)+\mathcal{A}(S)\subseteq \mathcal{A}(S)\cup 2\mathcal{A}(S)$. 
\end{enumerate}

\begin{remark}
    Notice that if $S$ is a gap absorbing monoid, then for every gap $\mathbf{h}$ of $S$, we have $\mathbf{h}+S^*\subseteq S$ (that is, $\mathbf{h}$ is a pseudo-Frobenius element of $S$).
\end{remark}


\begin{proposition} \label{prop:gap-absorbing-implies-ideal} 
Let $S\subseteq \mathbb{N}^{(I)}$ be a gap absorbing monoid. Then $S$ is an ideal extension of $\mathbb{N}^{(I)}$.
\end{proposition}
\begin{proof}
Notice that $\mathcal{M}(S)+\mathcal{H}(S)\subseteq \mathcal{A}(S)+\mathcal{H}(S)$, which by condition (GA2) is included in $S$. Thus, by Propositions~\ref{prop:char-ideal-gaps} we deduce that $S$ is an ideal extension of $\mathbb{N}^{(I)}$.
\end{proof}


\begin{conjecture}\label{conj:ideal-implies-gap-absorbing}
    Every ideal extension of $\mathbb{N}^{(I)}$ is gap absorbing.
\end{conjecture}

The family of monoids presented in the next result extends the family studied in \cite[Section~4.1]{Baeth}, and they are all gap absorbing.

\begin{proposition} \label{simple-case}
Let $\emptyset\neq J\subseteq I\subseteq \mathbb{N}$. 
Given positive integers $k_j$, $j\in J$, the monoid  $S=\{0\}\cup \{k_j\mathbf{e}_j: j\in J\}+\mathbb{N}^{(I)}$ is gap absorbing.    
\end{proposition} 
\begin{proof}
Set $H=\mathcal{H}(S)$, $A=\mathcal{A}(S)$, and $M=\mathcal{M}(S)=\{k_j\mathbf{e}_j : j\in J\}$. Let us describe how are the sets $H$ and $A$. 

Clearly, $H=\{(x_i)_{i\in I}\in \mathbb{N}^{(I)} : x_j<k_j \text{ for all } j\in J\}$. 

As $A=(M+\mathbb{N}^{(I)})\setminus(2(M+\mathbb{N}^{(I)}))=(M+\mathbb{N}^{(I)})\setminus(2M+\mathbb{N}^{(I)})$, if $\mathbf{x}=(x_i)_{i\in I}\in A$, then $x_i\ge k_i$ for some $i\in J$, because $\mathbf{x}\in M+\mathbb{N}^{(I)}$, and also $x_i<2k_i$, because otherwise $\mathbf{x}\in 2k_i\mathbf{e}_i+\mathbb{N}^{(I)}\subseteq 2M+\mathbb{N}^{(I)}$. There is no $j\in  J\setminus\{i\}$ such that $k_j\le x_j$, since in this setting we would have $\mathbf{x}\in (k_i\mathbf{e}_i+k_j\mathbf{e_j})+\mathbb{N}^{(I)}\subseteq 2M+\mathbb{N}^{(I)}$. We deduce that $A$ is the set of sequences $(x_i)_{i\in I}\in \mathbb{N}^{(I)}$ such that $k_i\le x_i\le 2k_i-1$ for some $i\in  J$ and $x_j<k_j$ for every $j\in  J\setminus\{i\}$.

Let us prove that $2H \subseteq H\cup A\cup 2A$. Take $\mathbf{h}=(h_i)_{i\in I}, \mathbf{h}'=(h_i')_{i\in I}\in H$. If $h_i+h_i'<k_i$ for all $j\in  I$, then $\mathbf{h}+\mathbf{h}'\in H$. If there is a unique $i\in  J$ such that $h_i+h_i'\ge k_i$, then $k_i\le h_i+h_i'<2k_i-1$ and $h_j+h_j'<k_j$ for all $j\in J\setminus\{i\}$, which means that $\mathbf{h}+\mathbf{h}'\in A$. Finally, if there is $i,j\in  J$ with $i\neq j$ and $h_i+h_i'\ge k_i$ and $h_j+h_j'\ge k_j$, we write $\mathbf{h}+\mathbf{h}'=(\mathbf{h}-h_i\mathbf{e}_i+h_j'\mathbf{e}_j)+(\mathbf{h}'-h_j'\mathbf{e}_j+h_i\mathbf{e}_i)$. It easily follows that $\mathbf{h}-h_i\mathbf{e}_i+h_j'\mathbf{e}_j$ and $\mathbf{h}'-h_j'\mathbf{e}_j+h_i\mathbf{e}_i$ are in $A$.

Now, let us show that $H+A\subseteq A\cup 2A$. Consider $\mathbf{h}=(h_i)_{i\in I}\in H$ and $\mathbf{a}=(a_i)_{i\in I}\in A$. Suppose that $k_i\le a_i<2k_i$ for some $i\in J$ and $0\le a_j<k_j$ for all $j\in  J\setminus\{i\}$. If $2k_i\le h_i+a_i<3k_i-1$, then write $\mathbf{h}+\mathbf{a}=(\mathbf{h}-h_i\mathbf{e_i}+k_i\mathbf{e_i})+(\mathbf{a}-a_i\mathbf{e_i}+(h_i+a_i-k_i)\mathbf{e}_i)$, which is in $A+A$, since both $\mathbf{h}-h_i\mathbf{e_i}+k_i\mathbf{e_i}$ and $\mathbf{a}-a_i\mathbf{e_i}+(h_i+a_i-k_i)\mathbf{e}_i$ are in $A$. Thus, in the following, we assume that $k_i\le h_i+a_i<2k_i$. We distinguish two cases.
\begin{itemize}
    \item Suppose that $h_j+a_j<k_j$ for all $j\in  J\setminus\{i\}$. Then, we obtain that $\mathbf{h}+\mathbf{a}\in A$. 
    \item If for some $j\in  J$, $j\neq i$, we have that $k_j\le h_j+a_j$ (notice that $h_j+a_j<2k_j-1$), then $\mathbf{h}+\mathbf{a}=(\mathbf{h}-h_i\mathbf{e}_i+a_j\mathbf{e}_j)+(\mathbf{a}-a_j\mathbf{e}_j+h_i\mathbf{e}_i)\in A+A$, since both $\mathbf{h}-h_i\mathbf{e}_i+a_j\mathbf{e}_j$ and $\mathbf{a}-a_j\mathbf{e}_j+h_i\mathbf{e}_i$ are in $A$. 
\end{itemize}
This proves that $S$ is gap absorbing.
\end{proof}

\begin{remark}
    Observe that gap absorbing monoids are not necessarily finitely generated even when $I$ has finitely many elements. Take for instance $S=\{\mathbf{0}\}\cup (\mathbf{e}_1+\mathbb{N}^2)$. In this case $\mathcal{A}(S)=\{(1,n) : n \in \mathbb{N}\}$ and by Proposition~\ref{simple-case}, $S$ is gap absorbing.
\end{remark}

The following result tells us that on an ideal extension of $\mathbb{N}^{(I)}$ (and therefore on every gap absorbing monoid by Proposition~\ref{prop:gap-absorbing-implies-ideal}) the set of atoms ``seats'' over the set of gaps.

\begin{lemma}\label{atoms-sitting-over-gaps}
Let $S$ be an ideal extension of $\mathbb{N}^{(I)}$. 
\begin{enumerate}[a)]
\item For all $\mathbf{h}\in \mathcal{H}(S)$, if $\mathbf{h}-\mathbf{e}_i \in \mathbb{N}^{(I)}$ for some $i\in  I$, then $\mathbf{h}-\mathbf{e}_i \in \{\mathbf{0}\}\cup \mathcal{H}(S)$.
\item For all $\mathbf{a}\in \mathcal{A}(S)$, if $\mathbf{a}-\mathbf{e}_i \in \mathbb{N}^{(I)}$ for some $i\in  I$, then $\mathbf{a}-\mathbf{e}_i\in \{\mathbf{0}\}\cup  \mathcal{H}(S) \cup \mathcal{A}(S)$. 
\item For every $\mathbf{h}\in \mathcal{H}(S)$, there exists $\mathbf{a}\in \mathcal{A}(S)$ with $\mathbf{h}\le \mathbf{a}$.
\end{enumerate}
\end{lemma}
\begin{proof}
The claim a) easily follows from the structure of $S$. For the claim b), suppose $\mathbf{a}-\mathbf{e}_i=\mathbf{b}+\mathbf{c}$ with $\mathbf{b},\mathbf{c}\in  S^*$. Then $\mathbf{a}=\mathbf{b}+\mathbf{c}+\mathbf{e}_i$ and as $S^*$ is an ideal of $\mathbb{N}^{(I)}$, we obtain $\mathbf{b}+\mathbf{e}_i\in  S^*$, so $\mathbf{a}\in  S^*+ S^*$, which contradicts $\mathbf{a}\in \mathcal{A}(S)$. 

For the last claim, consider $\mathbf{m}\in \operatorname{Minimals}_\le ((\mathbf{h}+\mathbb{N}^{(I)})\cap S)$. Then $\mathbf{m}=\mathbf{h}+\mathbf{x}$ for some $\mathbf{x}\in \mathbb{N}^{(I)}$. As $\mathbf{h}\not\in S$, we deduce that $\mathbf{x}\neq \mathbf{0}$. Suppose that $\mathbf{m}$ is not an atom. Then, $\mathbf{m}=\mathbf{a}+\mathbf{b}$ for some $\mathbf{a},\mathbf{b}\in S^*$. Let $i\in \operatorname{Supp}(\mathbf{x})\subseteq \operatorname{Supp}(\mathbf{m})$. Then $\mathbf{m}-\mathbf{e}_i\in \mathbf{h}+\mathbb{N}^{(I)}$, and by the minimality of $\mathbf{m}$, we deduce that $\mathbf{m}-\mathbf{e}_i\in \mathcal{H}(S)$. But we also have $\mathbf{m}-\mathbf{e}_i=\mathbf{a}+\mathbf{b}-\mathbf{e}_i$, and we can assume without loss of generality that $\mathbf{b}-\mathbf{e}_i\in \mathbb{N}^{(I)}$. So $\mathbf{m}-\mathbf{e}_i\in S$, that is a contradiction.     
\end{proof}

If we increase by one the coordinate of one of the atoms of an ideal extension of a free monoid, then we obtain either an atom or a sum of two atoms. In some cases, we can ensure that we obtain just another atom.

\begin{lemma}\label{lem:a+ei-l-le-3}
    Let $S$ be an ideal extension of $\mathbb{N}^{(I)}$. 
    Then, for every $\mathbf{a}\in \mathcal{A}(S)$ and every $i\in I$, $\mathbf{a}+\mathbf{e}_i\in \mathcal{A}(S)\cup 2\mathcal{A}(S)$. If in addition $\mathbf{e}_i\not \in \mathcal{M}(S)$, then $\mathbf{m}+\mathbf{e}_i\in \mathcal{A}(S)$ for all $\mathbf{m}\in \mathcal{M}(S)$.
\end{lemma}
\begin{proof}
    Clearly, $\mathbf{a}+\mathbf{e}_i\in S$. If $\mathbf{a}+\mathbf{e}_i\not\in \mathcal{A}(S)\cup 2\mathcal{A}(S)$, then $\mathbf{a}+\mathbf{e}_i=\mathbf{s}_1+\mathbf{s}_2+\mathbf{s}_3$ with $\mathbf{s}_1,\mathbf{s}_2,\mathbf{s}_3\in S^*$. As $i\in \operatorname{Supp}(\mathbf{a}+\mathbf{e}_i)$, there exists $j\in \{1,2,3\}$ such that $\mathbf{s}_j-\mathbf{e}_i\in \mathbb{N}^{(I)}$. Suppose without loss of generality that $j=3$. Hence, $\mathbf{a}=\mathbf{s}_1+(\mathbf{s}_2+(\mathbf{s}_3-\mathbf{e}_i))\in S^*+S^*$, contradicting that $\mathbf{a}\in \mathcal{A}(S)=S^*\setminus(S^*+S^*)$.

    For the second part, suppose that $\mathbf{m}+\mathbf{e}_i=\mathbf{t}_1+\mathbf{t}_2$ for some $\mathbf{t}_1,\mathbf{t}_2\in S^*$. Arguing as above, we may assume that $i\in \operatorname{Supp}(\mathbf{t}_2)$. Then, $\mathbf{m}=\mathbf{t}_1+(\mathbf{t}_2-\mathbf{e}_i)\in S$. Let $\mathbf{m}_1\in \mathcal{M}(S)$ be such that $\mathbf{m}_1\le \mathbf{t}_1$, and set $\mathbf{x}=\mathbf{t}_1-\mathbf{m}_1$. Then, $\mathbf{m}=\mathbf{m}_1+\mathbf{x}+(\mathbf{t}_2-\mathbf{e}_i)$, and since both $\mathbf{m}$ and $\mathbf{m}_1$ are minimal, this forces $\mathbf{x}+(\mathbf{t}_2-\mathbf{e}_i)=\mathbf{0}$; in particular $\mathbf{t}_2=\mathbf{e}_i$, which is impossible by the hypothesis $\mathbf{e}_i\not\in \mathcal{M}(S)$.
\end{proof}

For gap absorbing monoids, we have the following consequence.

\begin{corollary}\label{cor:An-sits-previous}
Let $S\subseteq \mathbb{N}^{(I)}$ be a gap absorbing monoid and $n$ be an integer greater than one. Then for all $\mathbf{s}\in n\mathcal{A}(S)$ and for all $i\in  I$ such that $\mathbf{s}-\mathbf{e}_i \in \mathbb{N}^{(I)}$ we have that $\mathbf{s}-\mathbf{e}_i\in (n-1)\mathcal{A}(S) \cup n\mathcal{A}(S)$.
\end{corollary}
\begin{proof}
    If $\mathbf{s}\in n\mathcal{A}(S)$, then $\mathbf{s}-\mathbf{e}_i=\mathbf{a}_1+\cdots +\mathbf{a}_n-\mathbf{e}_i$ with $\mathbf{a}_j\in \mathcal{A}(S)$ for each $j\in \{1,\ldots,n\}$, and we can assume without loss of generality that $\mathbf{a}_n-\mathbf{e}_i \in \mathbb{N}^{(I)}$. In particular, by Lemma~\ref{atoms-sitting-over-gaps}, $\mathbf{a}_n-\mathbf{e}_i \in \{\mathbf{0}\}\cup \mathcal{H}(S) \cup \mathcal{A}(S)$. Therefore, since $S$ is a gap absorbing monoid, we deduce that $\mathbf{s}-\mathbf{e}_i\in (n-1)\mathcal{A}(S) \cup n\mathcal{A}(S)$.
\end{proof}

The next result explains the relation between gap absorbing monoids and ideal extensions of free monoids.

\begin{proposition}\label{prop:2A-interval-eq}
Let $S$ be a submonoid of $\mathbb{N}^{(I)}$. The following conditions are equivalent.
\begin{enumerate}[(1)]
\item  $S$ is a gap absorbing monoid. 

\item 
$S$ is an ideal extension of $\mathbb{N}^{(I)}$ and $n\mathcal{A}(S)\setminus (\bigcup_{i=0}^{n-1} i\mathcal{A}(S))$ is closed under intervals for all $n\in \mathbb{N}\setminus \{0\}$.

\item 
$S$ is an ideal extension of $\mathbb{N}^{(I)}$ and $2\mathcal{A}(S)$ is closed under intervals. 
\end{enumerate}
\end{proposition}
\begin{proof} 
\emph{(1) implies (2)}. The claim $S=\{\mathbf{0}\}\cup (\mathcal{M}(S)+\mathbb{N}^{(I)})$ follows from Proposition~\ref{prop:gap-absorbing-implies-ideal}. We proceed by induction on $n$. For $n=1$, take $\mathbf{a}\le \mathbf{x}\le \mathbf{b}$, with $\mathbf{a},\mathbf{b}\in \mathcal{A}(S)$. As $S=\{\mathbf{0}\}\cup (\mathcal{M}(S)+\mathbb{N}^{(I)})$, we deduce that $\mathbf{x}\in S$. If $\mathbf{x}=\mathbf{s}+\mathbf{t}$ with $\mathbf{s},\mathbf{t}\in S^*$, then $\mathbf{b}=\mathbf{s}+(\mathbf{t}+\mathbf{b}-\mathbf{x})$, and both $\mathbf{s}$ and $\mathbf{t}+\mathbf{b}-\mathbf{x}$ are in $S^*$, contradicting that $\mathbf{b}$ is an atom. Now, consider $n>1$ and let $\mathbf{a}=\mathbf{a}_1+\cdots+\mathbf{a}_n$, $\mathbf{b}=\mathbf{b}_1+\cdots+\mathbf{b}_n$ with $\mathbf{a}_i,\mathbf{b}_i\in \mathcal{A}(S)$ for  $i\in \{1,\dots,n\}$, and $\mathbf{a},\mathbf{b}\in n\mathcal{A}(S)\setminus (\bigcup_{i=0}^{n-1} i\mathcal{A}(S))$. Let $\mathbf{x}\in \mathbb{N}^{(I)}$ such that $\mathbf{a}\leq \mathbf{x}\leq \mathbf{b}$. As $\mathbf{a}\in S^*$ and $S^*$ is an ideal of $\mathbb{N}^{(I)}$, we deduce that $\mathbf{x}\in S^*$. We can see that there exists $\mathbf{m}\in \mathbb{N}^{(I)}$ such that $\mathbf{x}=\mathbf{b}-\mathbf{m}$. It is not difficult to see that it is possible to express $\mathbf{m}=\mathbf{m}_1+\cdots+\mathbf{m}_n$ such that $\mathbf{b}_i-\mathbf{m}_i\in \mathbb{N}^{(I)}$ for $i\in \{1,\dots,n\}$, that is, $\mathbf{x}=(\mathbf{b}_1-\mathbf{m}_1)+\cdots+(\mathbf{b}_n-\mathbf{m}_n)$. In particular, since $\mathbf{b}_i\in \mathcal{A}(S)$ for all $i$, we have $\mathbf{b}_i-\mathbf{m}_i \in \{\mathbf{0}\}\cup \mathcal{H}(S)\cup \mathcal{A}(S)$ (Lemma~\ref{atoms-sitting-over-gaps}). Therefore, since $S$ is a gap absorbing monoid and $\mathbf{x}\in S^*$, we obtain $\mathbf{x}\in q\mathcal{A}(S)$ for some $q\leq n$. Let $l=\min\{n \in \mathbb{N} : \mathbf{x}\in n\mathcal{A}(S)\}$. 
Then, $l\le q\le n$. Suppose that $l<n$. Take $\overline{\mathbf{a}}=\mathbf{a}_1+\cdots+\mathbf{a}_{l}$. Observe that $\overline{\mathbf{a}}\in l\mathcal{A}(S)\setminus (\bigcup_{i=0}^{l-1} i\mathcal{A}(S))$. In fact, if $\overline{\mathbf{a}}\in i\mathcal{A}(S)$ with $i<l$, then $\mathbf{a}=\overline{\mathbf{a}}+\mathbf{a}_{l+1}+\cdots+\mathbf{a}_{n}\in (i+n-l)\mathcal{A}(S)$ with $i+n-l<n$. So we have $\overline{\mathbf{a}}\leq \mathbf{a} \leq \mathbf{x}$ with $\overline{\mathbf{a}},\mathbf{x}\in l\mathcal{A}(S)\setminus (\bigcup_{i=0}^{l-1} i\mathcal{A}(S))$ and by induction hypothesis we obtain $\mathbf{a}\in l\mathcal{A}(S)$ with $l<n$, which is a contradiction. Hence, $q=n$, that is, $\mathbf{x}\in n\mathcal{A}(S)\setminus (\bigcup_{i=0}^{n-1} i\mathcal{A}(S))$.   

\emph{(2) implies (3)}. We obtain this implication by (2) with $n=2$, since every element in $2\mathcal{A}(S)$ is not an atom, that is,  $2\mathcal{A}(S)\setminus \mathcal{A}(S)=2\mathcal{A}(S)$.

\emph{(3) implies (1)}. We prove condition (a). Let $\mathbf{h}_1, \mathbf{h}_2\in \mathcal{H}(S)$, and $\mathbf{m}_1, \mathbf{m}_2\in \mathcal{A}(S)$ such that $\mathbf{h}_1\leq \mathbf{m}_1$ and $\mathbf{h}_2\leq \mathbf{m}_2$ (Lemma~\ref{atoms-sitting-over-gaps}). If $\mathbf{h}_1+\mathbf{h}_2$ belongs to $\mathcal{H}(S) \cup \mathcal{A}(S)$, then we are done. So, suppose $\mathbf{h}_1+\mathbf{h}_2=\mathbf{s}_1+\mathbf{s}_2$ with $\mathbf{s}_1,\mathbf{s}_2\in  S^*$. Then there  exist $\mathbf{n}_1, \mathbf{n}_2\in \mathcal{A}(S)$ such that $\mathbf{n}_1\leq \mathbf{s}_1$ and $\mathbf{n}_2\leq \mathbf{s}_2$. In this setting, we have $\mathbf{n}_1+\mathbf{n}_2\leq \mathbf{h}_1+\mathbf{h}_2\leq \mathbf{m}_1+\mathbf{m}_2$ with, $\mathbf{n}_1+\mathbf{n}_2, \mathbf{m}_1+\mathbf{m}_2\in 2\mathcal{A}(S)$. Hence $\mathbf{h}_1+\mathbf{h}_2\in 2\mathcal{A}(S)$. Condition (b) can be proved in a similar way.  
\end{proof}

According to this last result, Conjecture~\ref{conj:ideal-implies-gap-absorbing} can be rephrased as follows.

\begin{conjecture}\label{conj:ideal-2A-closed-intervals}
    Let $S$ be an ideal extension of $\mathbb{N}^{(I)}$. Then $2\mathcal{A}(S)$ is closed under intervals.
\end{conjecture}

\section{Betti elements and lengths of factorizations}\label{sec:betti-elements}

Let $S$ be a submonoid of $\mathbb{N}^{(I)}$. Let $\mathcal{F}$ be the free monoid on $\mathcal{A}(S)$, and let $\varphi: \mathcal{F} \to S$ be the monoid morphism induced by $\mathbf{a}\mapsto \mathbf{a}$ for all $\mathbf{a}\in \mathcal{A}(S)$ (this morphism is sometimes called the factorization morphism of $S$). Given $\mathbf{s}\in S$, the \emph{set of factorizations of} $\mathbf{s}$ is $\mathsf{Z}(\mathbf{s})=\varphi^{-1}(\mathbf{s})$. Observe that the set of factorizations of $\mathbf{s}$ is finite since all the atoms involved in any factorization of $\mathbf{s}$ are in $\operatorname{B}(\mathbf{s})$, which is a finite set of $\mathbb{N}^{(I)}$. In other words, $S$ is a \emph{finite factorization monoid}, or FF-monoid for short (see for instance \cite{g-hk}).

For a given factorization $\mathbf{z}=\lambda_1 \mathbf{a}_1+\dots+\lambda_n\mathbf{a}_n$ of $\mathbf{s}\in S$  ($\lambda_1,\dots,\lambda_n\in \mathbb{N}$ and $\mathbf{s}= \lambda_1 \mathbf{a}_1+\dots+\lambda_n\mathbf{a}_n$ in $S$), the \emph{length} of $\mathbf{z}$ is $|\mathbf{z}|=\sum_{i=1}^n \lambda_i$. The set of lengths of factorizations of $\mathbf{s}$ is denoted by $\mathsf{L}(\mathbf{s})$. Recall that $S$ is \emph{half-factorial} if $|\mathsf{L}(\mathbf{s})|=1$ for all $\mathbf{s}\in S$. 

Since the number of factorizations of $\mathbf{s}$ is finite, we have that $\mathsf{L}(\mathbf{s})$ has finitely many elements, that is, $S$ is a \emph{bounded factorization monoid}, BF-monoid for short. Set $\ell(\mathbf{s})=\min\mathsf{L}(\mathbf{s})$. Observe that $\ell(\mathbf{s})=\min\{ n\in \mathbb{N} : \mathbf{s}\in n\mathcal{A}(S)\}$.

Next we show that $\ell$ has a non-decreasing behaviour on gap absorbing monoids, and that this property actually characterizes that fact of being gap absorbing.

\begin{proposition}\label{prop:non-decreasing-length}
    Let $S\subseteq \mathbb{N}^{(I)}$ be a gap absorbing monoid. Then, for every $\mathbf{s}\in S$ and every $i\in I$, $\ell(\mathbf{s})\le \ell(\mathbf{s}+\mathbf{e}_i)\le \ell(\mathbf{s})+1$. In particular, for every $\mathbf{t}\in \mathbf{s}+\mathbb{N}^{(I)}$, the inequality $\ell(\mathbf{s})\le \ell(\mathbf{t})$ holds.
\end{proposition}
\begin{proof}
    If $\mathbf{s}+\mathbf{e}_i$ is an atom, by Lemma~\ref{atoms-sitting-over-gaps}, $\mathbf{s}$ is also an atom and we are done.
    
    Let $\mathbf{s}+\mathbf{e}_i=\mathbf{a}_1+\dots+\mathbf{a}_n$, $n=\ell(\mathbf{s}+\mathbf{e}_i)\ge 2$. 
    By Corollary~\ref{cor:An-sits-previous}, 
    $\mathbf{s}=(\mathbf{s}+\mathbf{e}_i)-\mathbf{e}_i\in (n-1)\mathcal{A}(S)\cup n\mathcal{A}(S)$, which means that $\ell(\mathbf{s})\le n=\ell(\mathbf{s}+\mathbf{e}_i)$. 

    Now, let $\mathbf{s}=\mathbf{b}_1+\dots+\mathbf{b}_m$, with $\mathbf{b}_1,\dots,\mathbf{b}_m\in \mathcal{A}(S)$ and $m\ge 1$, be a factorization of $\mathbf{s}$ of minimal length. Then $\mathbf{s}+\mathbf{e}_i=\mathbf{b}_1+\dots+\mathbf{b}_m+\mathbf{e}_i$. By Proposition~\ref{prop:gap-absorbing-implies-ideal} and Lemma~\ref{lem:a+ei-l-le-3}, we have that $\mathbf{b}_m+\mathbf{e}_i\in \mathcal{A}(S)\cup 2\mathcal{A}(S)$, which means that $\mathbf{s}+\mathbf{e}_i\in m\mathcal{A}(S)\cup(m+1)\mathcal{A}(S)$, and so $\ell(\mathbf{s}+\mathbf{e}_i)\le m+1=\ell(\mathbf{s})+1$.
\end{proof}

\begin{proposition}\label{prop:gap-absorbing-eq-non-decreasing-length}
    Let $S$ be an ideal extension of $\mathbb{N}^{(I)}$. Suppose that for every $\mathbf{s},\mathbf{t}\in S$ with $\mathbf{s}\le \mathbf{t}$, we have $\ell(\mathbf{s})\le \ell(\mathbf{t})$. Then, $S$ is a gap absorbing monoid.
\end{proposition}
\begin{proof}
    Let $\mathbf{h}\in\mathcal{H}(S)$ and let $\mathbf{a}\in \mathcal{A}(S)$. By Lemma~\ref{atoms-sitting-over-gaps}, there exists $\mathbf{b}\in \mathcal{A}(S)$ such that $\mathbf{h}\le \mathbf{b}$. Thus $\mathbf{h}+\mathbf{a}\in S$ and $\mathbf{h}+\mathbf{a}\le \mathbf{a}+\mathbf{b}$. By hypothesis, $\ell(\mathbf{h}+\mathbf{a})\le \ell(\mathbf{a}+\mathbf{b})=2$, and consequently $\mathbf{h}+\mathbf{a}\in \mathcal{A}(S)\cup 2\mathcal{A}(S)$. This proves that $\mathcal{H}(S)+\mathcal{A}(S)\subseteq  \mathcal{A}(S)\cup 2\mathcal{A}(S)$.

    Now, let $\mathbf{h},\mathbf{h}'\in \mathcal{H}(S)$. By using Lemma~\ref{atoms-sitting-over-gaps} again, we know that there exist $\mathbf{a},\mathbf{a}'\in\mathcal{A}(S)$ such that $\mathbf{h}\le \mathbf{a}$ and $\mathbf{h}'\le \mathbf{a}'$. Then $\mathbf{h}+\mathbf{h}'\le \mathbf{a}+\mathbf{a}'$. If $\mathbf{h}+\mathbf{h}'\in S$, then by hypothesis $\ell(\mathbf{h}+\mathbf{h}')\le 2$, and thus $\mathbf{h}+\mathbf{h}'\in \mathcal{H}(S)\cup \mathcal{A}(S)\cup 2\mathcal{A}(S)$. Hence, $2\mathcal{H}(S)\subseteq \mathcal{H}(S)\cup \mathcal{A}(S)\cup 2\mathcal{A}(S)$.
\end{proof}

Next, we provide another characterization of the gap absorbing property in terms of minimal lengths of factorizations of certain elements.

\begin{proposition}
    Let $S$ be a submonoid of $\mathbb{N}^{(I)}$. Then $S$ is gap absorbing if and only if $S$ is an ideal extension of $\mathbb{N}^{(I)}$ and for every $\mathbf{a},\mathbf{b}\in \mathcal{A}(S)$ and every $i\in \operatorname{Supp}(\mathbf{a}\vee \mathbf{b})$, the inequality $\ell(\mathbf{a}+\mathbf{b}-\mathbf{e}_i)\le 2$ holds.
\end{proposition}
\begin{proof}
 \emph{Necessity.}  If $i\in \operatorname{Supp}(\mathbf{a}\vee \mathbf{b})$, then we can assume without loss of generality that $\mathbf{b}-\mathbf{e}_i\in \mathbb{N}^{(I)}$, which by Lemma~\ref{atoms-sitting-over-gaps} implies that $\mathbf{b}-\mathbf{e}_i\in \{\mathbf{0}\}\cup \mathcal{H}(S)\cup \mathcal{A}(S)$. If $\mathbf{b}-\mathbf{e}_i\in\mathcal{A}(S)$, then $\ell(\mathbf{a}+\mathbf{b}-\mathbf{e}_i) = 2$. If $\mathbf{b}-\mathbf{e}_i\in \{\mathbf{0}\}\cup \mathcal{H}(S)$, then as $S$ is gap absorbing, $\mathbf{a}+\mathbf{b}-\mathbf{e}_i\in \mathcal{A}(S)\cup 2\mathcal{A}(S)$, that is, $\ell(\mathbf{a}+\mathbf{b}-\mathbf{e}_i)\le 2$.

 \emph{Sufficiency.} We prove that $2\mathcal{A}(S)$ is closed under intervals and then make use of Proposition~\ref{prop:2A-interval-eq}. Take $\mathbf{a},\mathbf{b},\mathbf{c},\mathbf{d}\in \mathcal{A}(S)$ and $\mathbf{x}\in \mathbb{N}^{(I)}$ such that $\mathbf{a}+\mathbf{b}\le \mathbf{x}\le \mathbf{c}+\mathbf{d}$. We have to show that $\mathbf{x}\in 2\mathcal{A}(S)$. Let $\mathbf{y}=\mathbf{c}+\mathbf{d}-\mathbf{x}$. If $\mathbf{y}=\mathbf{0}$, then we are done. Otherwise, choose $i\in \operatorname{Supp}(\mathbf{y})$. We have $\mathbf{a}+\mathbf{b}\le \mathbf{x}\le \mathbf{c}+\mathbf{d}-\mathbf{e}_i$. As $i\in \operatorname{Supp}(\mathbf{y})\subseteq \operatorname{Supp}(\mathbf{c}+\mathbf{d})$, $\mathbf{c}+\mathbf{d}-\mathbf{e}_i\in S$ and by hypothesis we have that $1\le \ell(\mathbf{c}+\mathbf{d}-\mathbf{e}_i)\le 2$. Notice that $\mathbf{c}+\mathbf{d}-\mathbf{e}_i=\mathbf{a}+\mathbf{b}+\mathbf{z}$ for some $\mathbf{z}\in \mathbb{N}^{(I)}$, which means that $\mathbf{c}+\mathbf{d}-\mathbf{e}_i\not\in \mathcal{A}(S)$. Thus, $\mathbf{c}+\mathbf{d}-\mathbf{e}_i=\mathbf{c}'+\mathbf{d}'$ for some $\mathbf{c}',\mathbf{d}'\in \mathcal{A}(S)$, and $\mathbf{a}+\mathbf{b}\le \mathbf{x}\le \mathbf{c}'+\mathbf{d}'$ with $\mathbf{y}'=\mathbf{c}'+\mathbf{d}'-\mathbf{x}<\mathbf{y}$. We continue this process, which after a finite number of steps will yield $\mathbf{x}\in 2\mathcal{A}(S)$.
\end{proof}

Given a submonoid $S$ of $\mathbb{N}^{(I)}$, we define on $\mathbb{N}^{(I)}$ the following relation: \[\mathbf{a}\le_S \mathbf{b} \text{ if } \mathbf{b}-\mathbf{a}\in S.\] It is easy to prove that $\le_S$ is an order relation.

To every $\mathbf{s}\in S$, we associate a graph $\mathbf{G}_\mathbf{s}$ whose vertices are the atoms $\mathbf{a}$ such that $\mathbf{a}\le_S \mathbf{s}$, and two vertices $\mathbf{a}$ and $\mathbf{b}$ are connected by an edge if $\mathbf{a}+\mathbf{b}\le_S \mathbf{s}$. If $\mathbf{G}_\mathbf{s}$ is not connected, then we say that $\mathbf{s}$ is a \emph{Betti element} (or Betti degree) of $S$. The set of Betti elements of $S$ is denoted by $\operatorname{Betti}(S)$.

For an element $\mathbf{s}$ in a submonoid of $\mathbb{N}^{(I)}$ that is not an atom, we have $\ell(\mathbf{s})\ge 2$. When this minimal bound is attained and $\mathbf{s}$ has at least two different factorizations, we can ensure that $\mathbf{s}$ is a Betti element.

\begin{proposition}\label{prop:l2-denumerant2}
    Let $S$ be a submonoid of $\mathbb{N}^{(I)}$ and let $\mathbf{s}\in S$. If $\ell(\mathbf{s})=2$ and $|\mathsf{Z}(\mathbf{s})|\ge 2$, then $\mathbf{s}$ is a Betti element of $S$.
\end{proposition}
\begin{proof}
Since $\ell(\mathbf{s})=2$, there exist $\mathbf{a},\mathbf{b}\in \mathcal{A}(S)$ such that $\mathbf{s}=\mathbf{a}+\mathbf{b}$. Let $\mathbf{c}\in \mathcal{A}(S)\setminus \{\mathbf{a},\mathbf{b}\}$ be a vertex of $\mathbf{G}_\mathbf{s}$. Suppose that $\mathbf{a}\mathbf{c}$ is an edge of $\mathbf{G}_\mathbf{s}$. Then $\mathbf{s}=\mathbf{a}+\mathbf{b}=\mathbf{a}+\mathbf{c}+\mathbf{t}$ for some $\mathbf{t}\in S$. But then $\mathbf{b}=\mathbf{c}+\mathbf{t}$, which forces $\mathbf{t}=\mathbf{0}$ and $\mathbf{b}=\mathbf{c}$, a contradiction. This implies that there is no edge joining $\mathbf{a}$ to any atom other than $\mathbf{b}$. In the same way we show that $\mathbf{b}$ has the same property, and consequently $\{\mathbf{a},\mathbf{b}\}$ are the vertices of a connected component of $\mathbf{G}_\mathbf{s}$. We know that $\mathbf{G}_\mathbf{s}$ has more connected components, since $|\mathsf{Z}(\mathbf{s})|\ge 2$, and so $\mathbf{G}_\mathbf{s}$ is not connected.
\end{proof}

\begin{remark}
    The relation $\le_S$ can be defined on any commutative monoid $S$: for $x,y\in S$, $x\le_S y$ if there exists $z\in S$ such that $x+z=y$. If $S$ is reduced and cancellative, then $\le_S$ is an order relation (if multiplicative notation is used, this is just the divisibility relation).
    In the above result, we can replace $S$ by any atomic monoid that is cancellative, reduced and commutative monoid and the result still holds true.
\end{remark}


Next we prove that any vertex in a graph of an element of an ideal extension is always connected to a minimal non-zero element of the monoid (with respect to the usual partial order).

\begin{lemma} \label{lem:connect-to-M}
Let $S$ be an ideal extension of $\mathbb{N}^{(I)}$. Let $\mathbf{s}\in S^*\setminus (\mathcal{A}(S)\cup 2\mathcal{A}(S))$. Then for every vertex $\mathbf{a}$ of $\mathbf{G}_\mathbf{s}$ there is another vertex $\mathbf{m}$ of $\mathbf{G}_\mathbf{s}$ such that $\mathbf{m}\in\mathcal{M}(S)$ and $\mathbf{a}\mathbf{m}$ is an edge of $\mathbf{G}_\mathbf{s}$.
\label{connected}
\end{lemma}
\begin{proof}
Observe that $\mathbf{s}-\mathbf{a}=\mathbf{t}\in S$. By hypothesis $\mathbf{s}\not\in \mathcal{A}\cup 2\mathcal{A}(S)$, and so $\mathbf{t}\neq\mathbf{0}$ and $\mathbf{t}\notin \mathcal{A}(S)$. Hence, $\mathbf{t}=\mathbf{t}_1+\mathbf{t}_2$ with $\mathbf{t}_1,\mathbf{t}_2\in S^*$. In particular, there exists $\mathbf{m}\in \mathcal{M}(S)$ such that $\mathbf{t}_1=\mathbf{m}+\mathbf{x}$ with $\mathbf{x}\in \mathbb{N}^{(I)}$. Therefore $\mathbf{s}-\mathbf{a}-\mathbf{m}=\mathbf{t}_2+\mathbf{x}\in S$, that is, $\mathbf{a}\mathbf{m}$ is an edge of $\mathbf{G}_\mathbf{s}$.
\end{proof}

The next goal is to determine when two minimal non-zero elements, in a gap absorbing monoid, are in the same connected component.

\begin{lemma}\label{lemma:connected-factrorizatios}

Let $S$ be a gap absorbing monoid. Let $\mathbf{s}\in S\setminus (\mathcal{A}(S)\cup 2\mathcal{A}(S))$, $\mathbf{s}\neq \mathbf{0}$, and $\mathbf{m},\mathbf{n}\in \mathcal{M}(S)$ two vertices of $\mathbf{G}_\mathbf{s}$.
Suppose that one of the following condition holds:
\begin{enumerate}
    \item $\mathbf{m}\vee \mathbf{n}\in \mathcal{A}(S)$,
    \item $\ell(\mathbf{s})\geq 4$.
\end{enumerate}
Then, $\mathbf{m}$ and $\mathbf{n}$ are in the same connected component of  $\mathbf{G}_\mathbf{s}$.
\end{lemma}
\begin{proof}
Let $H=\mathcal{H}(S)$ and $A=\mathcal{A}(S)$. 

By Proposition~\ref{prop:gap-absorbing-implies-ideal}, we know that $S=\{\mathbf{0}\}\cup (\mathcal{M}(S)+\mathbb{N}^{(I)})$. We can write $\mathbf{s}= \mathbf{m}+\mathbf{a}\ge \mathbf{m}$ and $\mathbf{s}=\mathbf{n}+\mathbf{b}\ge \mathbf{n}$ with $\mathbf{a},\mathbf{b}\in S$. Thus, $\mathbf{s}=\mathbf{m}\vee\mathbf{n}+\mathbf{c}$ for some $\mathbf{c}\in \mathbb{N}^{(I)}$. 

Suppose that condition (1) holds. This in particular means that $\mathbf{c}\neq \mathbf{0}$. If $\mathbf{c}\in H$, then as $S$ is gap absorbing, we deduce that $\mathbf{s}\in A\cup 2A$, which contradicts the choice of $\mathbf{s}$. Thus $\mathbf{c}\in S$, and $\mathbf{s}=\mathbf{m}+\mathbf{a}=\mathbf{n}+\mathbf{b}=\mathbf{m}\vee\mathbf{n}+\mathbf{c}$. Notice that $\mathbf{c}$ cannot be an atom since this would imply that $\mathbf{s}=\mathbf{m}\vee\mathbf{n}+\mathbf{c}\in 2A$, a contradiction. Thus, there exists $\mathbf{d}$ an atom of $S$ and $\mathbf{e}\in S^*$ such that $\mathbf{c}=\mathbf{d}+\mathbf{e}$. Take $\mathbf{x}_1=(\mathbf{m}\vee\mathbf{n})-\mathbf{m}\in \mathbb{N}^{(I)}$ and $\mathbf{x}_2=(\mathbf{m}\vee\mathbf{n})-\mathbf{n}\in \mathbb{N}^{(I)}$. Then, $\mathbf{s}=\mathbf{m}+\mathbf{a}=\mathbf{n}+\mathbf{b}=\mathbf{m}+\mathbf{d}+\mathbf{e}+\mathbf{x}_1=\mathbf{n}+\mathbf{d}+\mathbf{e}+\mathbf{x}_2$. As $\mathbf{e}+\mathbf{x}_1, \mathbf{e}+\mathbf{x}_2\in S$, we conclude that $\mathbf{m}\mathbf{d}$ and $\mathbf{n}\mathbf{d}$ are edges of $\mathbf{G}_\mathbf{s}$, and thus $\mathbf{m}$ and $\mathbf{n}$ are in the same connected component.

Now suppose that (2) holds, and that $\mathbf{m}\vee\mathbf{n}$ is not an atom. Notice that $\mathbf{m}\vee\mathbf{n}\le \mathbf{m}+\mathbf{n}\in A+A$. As $S$ is gap absorbing and $\mathbf{m}\vee\mathbf{n}$ is not an atom, we deduce that $\mathbf{m}\vee\mathbf{n}$ is in $2A$; write $\mathbf{m}\vee\mathbf{n}=\mathbf{d}+\mathbf{e}$, with $\mathbf{d},\mathbf{e}\in A$. If $\mathbf{c}$ is a gap of $S$, then $\mathbf{c}+\mathbf{d}\in A\cup 2A$, and thus $\mathbf{s}=\mathbf{d}+\mathbf{e}+\mathbf{c}\in 2A\cup 3A$, contradicting that $\ell(\mathbf{s})\ge 4$. Thus $\mathbf{c}\in S$, and as $\ell(\mathbf{s})\ge 4$, we have that $\mathbf{c}$ is not an atom and it is not zero. We can write $\mathbf{c}=\mathbf{f}+\mathbf{t}$ with $\mathbf{f}$ an atom of $S$ and $\mathbf{t}\in S^*$. In particular, $\mathbf{s}=\mathbf{m}+\mathbf{f}+\mathbf{t}+\mathbf{x}_1=\mathbf{n}+\mathbf{f}+\mathbf{t}+\mathbf{x}_2$ with $\mathbf{x}_1$ and $\mathbf{x}_2$ as in the previous paragraph, concluding again that $\mathbf{m}$ and $\mathbf{n}$ are in the same connected component of $\mathbf{G}_\mathbf{s}$.
\end{proof}

Thus every vertex of a graph associated to an element in $S^*$ is connected to a vertex that is minimal in $S^*$, and in some particular cases two of these minimal vertices are in the same connected component. Putting this together we obtain some relevant consequences.

\begin{theorem}\label{thm:lge4-not-betti}
Let $S\subseteq \mathbb{N}^{(I)}$ be a gap absorbing monoid and let $\mathbf{s}\in S$. If $\mathbf{s}$ is a Betti element, then $\mathbf{s}\in 2\mathcal{A}(S)\cup 3\mathcal{A}(S)$.
\end{theorem}
\begin{proof}
    Let $\mathbf{s}$ be a Betti element of $S$. Suppose that $\mathbf{s}\not \in 2\mathcal{A}(S)\cup 3\mathcal{A}(S)$. Then as $\mathbf{s}$ is not an atom, $\ell(\mathbf{s})\ge 4$. Then by Lemma~\ref{lem:connect-to-M} every vertex of $\mathbf{G}_\mathbf{s}$ is connected to an atom in $\operatorname{Minimals}_\le (S^*)$, and by Lemma~\ref{lemma:connected-factrorizatios}, any two atoms in $\operatorname{Minimals}_\le (S^*)$ are connected, concluding that $\mathbf{G}_\mathbf{s}$ is connected, a contradiction.
\end{proof}

\begin{corollary}\label{cor:antichain-of-supatoms}
Let $S$ be a gap absorbing monoid. 
Let $\mathbf{a}\in S$ with $\ell(\mathbf{a})=3$ and suppose that for every $\mathbf{m},\mathbf{n}\in \mathcal{M}(S)\cap \operatorname{B}(\mathbf{a})$, there exists $\mathbf{m}_1,\ldots,\mathbf{m}_{h} \in \mathcal{M}(S)\cap \operatorname{B}(\mathbf{a})$ such that $\mathbf{m}_1=\mathbf{m}$, $\mathbf{m}_h=\mathbf{n}$ and $\mathbf{m}_i \vee \mathbf{m}_{i+1}\in \mathcal{A}(S)$ for all $i\in \{1,\ldots,h-1\}$. Then $\mathbf{a}$ is not a Betti element.

\end{corollary}

\begin{remark} \label{rem:betti-base-case} 
    Let $S$ be the gap absorbing monoid defined in Proposition~\ref{simple-case}, with $\mathcal{M}(S)=\{k_j\mathbf{e}_j : j\in J\}$. Let $\mathbf{s}\in S\setminus (\mathcal{A}(S)\cup 2\mathcal{A}(S))$ and assume $\mathbf{s}=\mathbf{m}+\mathbf{a}$ and $\mathbf{s}=\mathbf{n}+\mathbf{b}$ with $\mathbf{m},\mathbf{n}\in \mathcal{M}(S)$. Using the same argument as in the proof of Lemma~\ref{lemma:connected-factrorizatios}, we can write $\mathbf{s}=\mathbf{m}\vee \mathbf{n}+\mathbf{c}$ with $\mathbf{c}\in \mathbb{N}^{(I)}$. In this particular case, $\mathbf{m}\vee \mathbf{n}=\mathbf{m}+ \mathbf{n}\in 2\mathcal{A}(S)$  and consequently $\mathbf{c}\neq \mathbf{0}$. By assuming $\mathbf{c}\in \mathcal{H}(S)$, from the description of $\mathcal{H}(S)$ and $\mathcal{A}(S)$, we obtain $\mathbf{n}+\mathbf{c}\in \mathcal{A}(S)$. So $\mathbf{s}=\mathbf{m}+(\mathbf{n}+\mathbf{c})\in 2\mathcal{A}(S)$, which contradicts the choice of $\mathbf{s}$. Thus, $\mathbf{c}\in S$ and $\mathbf{s}=\mathbf{m}+\mathbf{a}=\mathbf{n}+\mathbf{b}=\mathbf{m}+\mathbf{n}+\mathbf{c}$, which implies, using Lemma~\ref{lem:connect-to-M}, that $\mathbf{G}_\mathbf{s}$ is connected. In particular, $\mathrm{Betti}(S)\subseteq 2\mathcal{A}(S)$. 
\end{remark}



Let $S$ be an ideal extension of $\mathbb{N}^{(I)}$. Next, we examine the case $\mathbf{e}_i\in \mathcal{M}(S)$ for some $i\in I$, showing how this case can be reduced to the study of an ideal extension in $\mathbb{N}^{(J)}$ with $J\subsetneq I$.

Suppose that $\mathcal{M}(S)\cap \{\mathbf{e}_i: i \in I\}=\{\mathbf{e}_{l} : l\in L\}$ with $L\subsetneq I$  (in the case $L=I$ we have $S=\mathbb{N}^{(I)}$). Define $J=I\setminus L$ and 
\[\mathsf{C}(S)=\{\mathbf{s}\in S : \operatorname{Supp}(\mathbf{s})\subseteq J\},\] which is called the \emph{core} of $S$ (using the notation in \cite{Baeth}).

We now show how the problem of obtaining the Betti elements of $S$ can be reduced to finding the Betti elements in $\mathsf{C}(S)$. First, we need a technical lemma which, among other things, relates the atoms, minimal non-zero elements and gaps of $S$ and $\mathsf{C}(S)$.

\begin{lemma}\label{lem:basic-facts-core}
Let $S$ be an ideal extension of $\mathbb{N}^{(I)}$ and let $L=\{i\in I : \mathbf{e}_i\in S\}$.
\begin{enumerate}[(a)]
\item $S=\mathsf{C}(S)\cup (\bigcup_{l \in L} (\mathbf{e}_{l}+\mathbb{N}^{(I)}))$ (disjoint union).
\item $\mathcal{H}(S)=\mathcal{H}(\mathsf{C}(S))$.
\item $\mathcal{M}(S)=\mathcal{M}(\mathsf{C}(S))\cup \{\mathbf{e}_l:l\in L\}$ (disjoint union).
\item $\mathsf{C}(S)$ is a divisor-closed submonoid of $S$.
\item $\mathcal{A}(S)=\mathcal{A}(\mathsf{C}(S))\cup \{\mathbf{e}_{l} :l \in L\}\cup (\bigcup_{l \in L} (\mathbf{e}_{l}+\mathcal{H}(S)))$.
\end{enumerate}
\end{lemma}
\begin{proof}
The first three assertions follow directly from the definition. 

If $\mathbf{s}\in \mathsf{C}(S)$ and $\mathbf{t}\in S$ is such that $\mathbf{t}\le_S \mathbf{s}$, then $\mathbf{t}\in \mathsf{C}(S)$. This proves that $\mathsf{C}(S)$ is a divisor-closed submonoid of $S$, and in particular, $\mathcal{A}(S)\cap \mathsf{C}(S)=\mathcal{A}(\mathsf{C}(S))$. If $\mathbf{a}\in \mathcal{A}(S)$, then either $\mathbf{a}\in \mathsf{C}(S)$ or $\mathbf{a}\in \bigcup_{l\in L}\mathbf{e}_l+\mathbb{N}^{(I)}$. 
\begin{itemize}
    \item If $\mathbf{a}\in \mathsf{C}(S)$, then  $\mathbf{a}\in \mathcal{A}(S)\cap \mathsf{C}(S)=\mathcal{A}(\mathsf{C}(S))$.
    \item If $\mathbf{a}=\mathbf{e}_l+\mathbf{x}$ for some $l\in L$ and $\mathbf{x}\in \mathbb{N}^{(I)}$, then, according to Lemma~\ref{atoms-sitting-over-gaps}, $\mathbf{x}=\mathbf{a}-\mathbf{e}_l\in \{\mathbf{0}\}\cup  \mathcal{H}(S)\cup \mathcal{A}(S)$. Since $\mathbf{a}$ is an atom, $\mathbf{x}$ cannot be in $\mathcal{A}(S)$, which forces $\mathbf{x}=\mathbf{0}$ or $\mathbf{x}\in \mathcal{H}(S)$.
\end{itemize}
For the other inclusion, we already know that $\mathcal{A}(\mathsf{C}(S))\subseteq \mathcal{A}(S)$. Also $\mathbf{e}_l\in \mathcal{M}(S)\subseteq \mathcal{A}(S)$ for all $l\in L$. Let $l\in L$ and $\mathbf{h}\in \mathcal{H}(S)$. If $\mathbf{e}_l+\mathbf{h}=\mathbf{s}_1+\mathbf{s}_2$, with $\mathbf{s}_1,\mathbf{s}_2\in S^*$, then $l$ is either in the support of $\mathbf{s}_1$ or in the support of $\mathbf{s}_2$. Thus, $\mathbf{h}=\mathbf{s}_1+\mathbf{s}_2-\mathbf{e}_l\in S$, which is a contradiction.
\end{proof}

Since $\mathsf{C}(S)$ is divisor-closed, the factorizations of an element $\mathbf{s}\in \mathsf{C}(S)$ involve only atoms in $\mathcal{A}(\mathsf{C}(S))$. Thus, $\mathbf{s}$ has a unique factorization in $S$ if and only if $\mathbf{s}$ has a unique factorization in $\mathsf{C}(S)$. 

\begin{lemma} \label{lem:unique_fctr-2A-CS}
Let $S$ be an ideal extension of $\mathbb{N}^{(I)}$ and let $L=\{i\in I : \mathbf{e}_i\in S\}$.
Let $\mathbf{s}\in 2\mathcal{A}(\mathsf{C}(S))$ and suppose it has a unique factorization. Then one of the following occurs:
\begin{enumerate}
\item $\mathbf{s}=\mathbf{m}_1+\mathbf{m}_2$ with $\mathbf{m}_1,\mathbf{m}_2\in \mathcal{M}(\mathsf{C}(S))$;
\item $\mathbf{s}=(\mathbf{m}+\mathbf{e}_i)+\mathbf{m}=2\mathbf{m}+\mathbf{e}_i$ for some $\mathbf{m}\in \mathcal{M}(\mathsf{C}(S))$ and  some $i\in I\setminus L$.
\end{enumerate}
\end{lemma}
\begin{proof}
Suppose $\mathbf{s}=\mathbf{s}_1+\mathbf{s}_2$ with $\mathbf{s}_1,\mathbf{s}_2\in \mathcal{A}(\mathsf{C}(S))$. Then $\mathbf{s}_1=\mathbf{m}_1+\mathbf{n}_1$ and $\mathbf{s}_2=\mathbf{m}_2+\mathbf{n}_2$ for some $\mathbf{m}_1,\mathbf{m}_2\in \mathcal{M}(S)$ and $\mathbf{n}_1,\mathbf{n}_2\in \mathbb{N}^{(I)}$. From $\mathbf{s}=(\mathbf{m}_1+\mathbf{n}_1)+(\mathbf{m}_2+\mathbf{n}_2)=\mathbf{m}_1+(\mathbf{m}_2+\mathbf{n}_1+\mathbf{n}_2)=\mathbf{m}_2+(\mathbf{m}_1+\mathbf{n}_1+\mathbf{n}_2)$ we obtain two different factorizations of $\mathbf{s}$ in the case $\mathbf{n}_1$ and $\mathbf{n_2}$ are both non-zero. Then, suppose without loss of generality that $\mathbf{n}_2=\mathbf{0}$.
Hence, $\mathbf{s}=(\mathbf{m_1}+\mathbf{n_1})+\mathbf{m}_2$. 

If $\mathbf{n}_1=\mathbf{0}$, then we fall into case (1).

If $\mathbf{n}_1$ is not zero, then $\mathbf{n}_1\geq \mathbf{e}_i$ for some $i\in I$ (as $\mathbf{s}\in \mathsf{C}(S)$, this forces $i\not\in L$). By Lemma~\ref{lem:a+ei-l-le-3}, we have $\mathbf{m}_1+\mathbf{e}_i, \mathbf{m}_2+\mathbf{e}_i\in \mathcal{A}(S)$, so $\mathbf{s}=(\mathbf{m_1}+\mathbf{n_1})+\mathbf{m}_2=(\mathbf{m}_1+\mathbf{e}_i)+(\mathbf{m}_2+\mathbf{n}_1-\mathbf{e}_i)$. Since $\mathbf{s}$ admits a unique factorization, we deduce that $\mathbf{m}_1+\mathbf{e}_i\in \{\mathbf{m}_1+\mathbf{n_1},\mathbf{m}_2\}$. As $\mathbf{m}_1$ and $\mathbf{m}_2$ are incomparable, we deduce $\mathbf{m}_1+\mathbf{n}_1=\mathbf{m}_1+\mathbf{e}_i$. Therefore, $\mathbf{n}_1=\mathbf{e}_i$. As the two factorizations $\mathbf{s}=(\mathbf{m}_1+\mathbf{e}_i)+\mathbf{m_2}=\mathbf{m}_1+(\mathbf{m}_2+\mathbf{e}_i)$ must be the same, we conclude that $\mathbf{m}_1=\mathbf{m}_2$, that is, $\mathbf{s}=2\mathbf{m}_1+\mathbf{e}_i$.
\end{proof}

Not every element that is the sum of two minimal non-zero elements will have a unique factorization as the next example shows.

\begin{example}
    Let $S=\{(0,0)\}\cup(\{(2,0),(1,2),(0,3)\}+\mathbb{N}^2$. Then $|\mathsf{Z}(2(1,2))|=3$.
\begin{verbatim}
gap> M:=[[2,0],[1,2],[0,3]];;
gap> s:=FiniteComplementIdealExtension(M);
<Affine semigroup>
gap> a:=MinimalGenerators(s);
[ [ 0, 3 ], [ 1, 2 ], [ 0, 5 ], [ 0, 4 ], [ 3, 0 ], [ 2, 0 ], [ 1, 3 ], 
  [ 2, 1 ], [ 1, 4 ], [ 2, 2 ], [ 3, 1 ] ]
\end{verbatim}
Now, we run over all the elements in $2\mathcal{A}(S)$ and look for the first having more than one factorization.
\begin{verbatim}
gap> Filtered(Set(Cartesian(M,M),Sum), 
    mm->Length(FactorizationsVectorWRTList(mm,a))>1);
[ [ 2, 4 ] ]
gap> FactorizationsVectorWRTList([2,4],a);
[ [ 0, 0, 0, 1, 0, 1, 0, 0, 0, 0, 0 ], [ 0, 2, 0, 0, 0, 0, 0, 0, 0, 0, 0 ], 
  [ 1, 0, 0, 0, 0, 0, 0, 1, 0, 0, 0 ] ]
\end{verbatim}
\end{example}

Next, we study what elements with minimum length equal to two that are not in the core of $S$ have unique factorizations.

\begin{lemma} \label{lem:unique_fctr-2A-NotCS}
Let $S$ an ideal extension of $\mathbb{N}^{(I)}$, $L=\{i\in I : \mathbf{e}_i\in S\}$ and $J=I\setminus L$. If $\mathbf{s}\in 2\mathcal{A}(S)\setminus\mathsf{C}(S)$, then $\mathbf{s}$ has a unique factorization if and only if one of the following occurs:
\begin{enumerate}[(1)]
\item $\mathbf{s}=\mathbf{e}_{l}+\mathbf{m}$ with $\mathbf{m}\in \mathcal{M}(S)$, and $l\in L$;
\item $\mathbf{s}=(\mathbf{e}_{l}+\mathbf{e}_{j})+\mathbf{e}_{l}=2\mathbf{e}_{l}+\mathbf{e}_j$ with $l\in L$ and $j\in J$.
\end{enumerate}
\label{unique_fact}
\end{lemma} 
\begin{proof}
\emph{Necessity}. Let $\mathbf{s}=\mathbf{s}_1+\mathbf{s}_2$, with $\mathbf{s}_1,\mathbf{s}_2\in \mathcal{A}(S)$. Since $\mathbf{s}\not\in \mathsf{C}(S)$ and $\mathsf{C}(S)$ is a submonoid of $S$, we can suppose that $\mathbf{s}_2\not\in \mathsf{C}(S)$. Thus, $\mathbf{s}_2=\mathbf{e}_l+\mathbf{x}$ for some $l\in L$ and $\mathbf{x}\in \{\mathbf{0}\}\cup\mathcal{H}(S)$ (Lemma~\ref{lem:basic-facts-core}). 

If $\mathbf{x}$ is not zero, then $\mathbf{s}=\mathbf{s}_1+(\mathbf{e}_l+\mathbf{x})=(\mathbf{s}_1+\mathbf{x})+\mathbf{e}_l$ and the fact that $\mathbf{s}$ admits a unique factorization forces $\mathbf{e}_l=\mathbf{s}_1$; whence $\mathbf{s}=2\mathbf{e}_l+\mathbf{x}$. Let $i\in I$ be such that $i$ is in the support of $\mathbf{x}$. The fact that $\mathbf{x}=\mathbf{e}_i$ is a gap yields $i\in J$. Then $\mathbf{s}=(\mathbf{e}_l+\mathbf{e}_i)+(\mathbf{e}_l+\mathbf{x}-\mathbf{e}_i)$. If $\mathbf{x}\neq \mathbf{e}_i$, Lemma~\ref{atoms-sitting-over-gaps} is telling us that $\mathbf{x}-\mathbf{e}_i\in \mathcal{H}(S)$, and by Lemma~\ref{lem:basic-facts-core}, $\mathbf{e}_l+\mathbf{x}-\mathbf{e}_i\in \mathcal{A}(S)$. The uniqueness of the factorizations of $\mathbf{s}$ implies that $\mathbf{e}_l+\mathbf{x}-\mathbf{e}_i\in \{\mathbf{e}_l,\mathbf{e}_l+\mathbf{x}\}$, which is impossible. Thus $\mathbf{x}=\mathbf{e}_i$, and $\mathbf{s}=2\mathbf{e}_l+\mathbf{e}_i$. 

If $\mathbf{x}=0$, then $\mathbf{s}=\mathbf{s}_1+\mathbf{e}_l$. If $\mathbf{s}_1$ is not a minimal element of $S$, we can pick $\mathbf{m}\in \mathcal{M}(S)$ and $\mathbf{y}\in \mathbb{N}^{(I)}$ such that $\mathbf{s}_1=\mathbf{m}+\mathbf{y}$ (and thus $\mathbf{y}$ is a gap of $S$, since otherwise $\mathbf{s}_1$ would not be an atom). But then Lemma~\ref{lem:basic-facts-core} asserts that $\mathbf{e}_l+\mathbf{y}$ is an atom of $S$, which means that $\mathbf{s}=(\mathbf{m}+\mathbf{y})+\mathbf{e}_l=\mathbf{m}+(\mathbf{e}_l+\mathbf{y})$ has two factorizations; a contradiction. Thus $\mathbf{y}$ is zero and $\mathbf{s}=\mathbf{e}_l+\mathbf{m}$ with $\mathbf{m}\in \mathcal{M}(S)$.

{\em Sufficiency}. If $\mathbf{s}=\mathbf{e}_l+\mathbf{m}=\mathbf{a}_1+\dots+\mathbf{a}_n$ with $l\in L$, $\mathbf{m}\in \mathcal{M}(S)$, and $\mathbf{a}_1,\dots,\mathbf{a}_n\in\mathcal{A}(S)$, then $n\ge 2$ since $\mathbf{s}$ is not an atom. We can assume without loss of generality that $l$ is in the support of $\mathbf{a}_n$, and consequently $\mathbf{m}=\mathbf{a}_1+\dots+\mathbf{a}_{n-1}+(\mathbf{a}_n-\mathbf{e}_l)$. 
The minimality of $\mathbf{m}$ forces $n=2$ and $\mathbf{a}_n-\mathbf{e}_l=\mathbf{0}$.

If $\mathbf{s}=2\mathbf{e}_l+\mathbf{e}_j=\mathbf{a}_1+\dots+\mathbf{a}_n$ with $l\in L$ and $j\in J$ and $\mathbf{a}_1,\dots,\mathbf{a}_n\in \mathcal{A}(S)$, then $2\le n\le 3$, since $\lVert \mathbf{s}\rVert_1 =3$. If $n=3$, then $\mathbf{a}_i=\mathbf{e}_{j_i}$ for each $i\in\{1,2,3\}$, but this leads to $\{\mathbf{a}_1,\mathbf{a}_2,\mathbf{a}_3\}=\{\mathbf{e}_l,\mathbf{e}_j\}$, and $\mathbf{e}_j$ is a gap of $S$. Thus, $n=2$, and we can assume that $\mathbf{a}_1=\mathbf{e}_{i_1}+\mathbf{e}_{i_2}$ and $\mathbf{a}_2=\mathbf{e}_{i_3}$ (one of them must have 1-norm equal to one and the other equal to two). From the equality $2\mathbf{e}_l+\mathbf{e}_j=\mathbf{e}_{i_1}+\mathbf{e}_{i_2}+\mathbf{e}_{i_3}$ we deduce that $\{i_1,i_2,i_3\}=\{l,j\}$. As $\mathbf{e}_{i_3}$ is an atom, by Lemma~\ref{lem:basic-facts-core}, $i_3\in L$, and so $i_3= l$. Therefore,  $2\mathbf{e}_l+\mathbf{e}_j=\mathbf{a}_1+\mathbf{e}_{l}$, which forces $\mathbf{a}_1=\mathbf{e}_l+\mathbf{e}_j$.
\end{proof}

\begin{proposition}\label{prop:betti-core}
Let $S$ an ideal extension of  $\mathbb{N}^{(I)}$, $L=\{i\in I : \mathbf{e}_i\in S\}$ and $J=I\setminus L$. Denote
\begin{itemize}
\item $B_1= \{\mathbf{e}_{l}+\mathbf{a}: l \in L, \mathbf{a}\in \mathcal{A}(\mathsf{C}(S))\setminus \mathcal{M}(S)\}$;
\item $B_2= \{2\mathbf{e}_{l}+\mathbf{h} : l\in L, \mathbf{h}\in \mathcal{H}(S)\setminus \{\mathbf{e}_j : j\in J \}\}$;
\item $B_3= \{\mathbf{e}_{l}+\mathbf{h}+\mathbf{a}: l\in L, \mathbf{h}\in \mathcal{H}(S), \mathbf{a}\in \mathcal{A}(\mathsf{C}(S))\}$; 
\item $B_4=\bigcup_{l,k \in L} (\mathbf{e}_{l}+\mathcal{H}(S))+(\mathbf{e}_{k}+\mathcal{H}(S))$.
\end{itemize} 
Then $\mathrm{Betti}(S)=B_1 \cup B_2 \cup B_3 \cup B_4 \cup \mathrm{Betti}(\mathsf{C}(S))$.
\end{proposition}
\begin{proof}
Observe that if $\mathbf{s}\in \mathsf{C}(S)$, as $\mathsf{C}(S)$ is divisor-closed, all the vertices of $\mathbf{G}_{\mathbf{s}}$ are in $\mathcal{A}(\mathsf{C}(S))$. Hence, $\operatorname{Betti}(\mathsf{C}(S))=\mathsf{C}(S)\cap \operatorname{Betti}(S)$. Thus, we can concentrate on the Betti elements of $S$ that are not in $\mathsf{C}(S)$. 

Let $\mathbf{s}\in \operatorname{Betti}(S)\setminus \mathsf{C}(S)$. If we prove that $\mathbf{s}\in 2\mathcal{A}(S)$, then our result follows directly by Lemma~\ref{lem:unique_fctr-2A-NotCS} and Proposition~\ref{prop:l2-denumerant2}. We show that if $\ell(\mathbf{s})\ge 3$ and $\mathbf{s}\not\in\mathsf{C}(S)$, then $\mathbf{G}_{\mathbf{s}}$ is connected.

We are assuming that $\mathbf{s}-\mathbf{e}_k\in \mathbb{N}^{(I)}$ for some $k\in L$ ($\mathbf{s}$ not in $\mathsf{C}(S)$), and that $\mathbf{s}=\mathbf{a}_1+\dots+\mathbf{a}_r$ for some $\mathbf{a}_1,\dots,\mathbf{a}_r\in\mathcal{A}(S)$, with $r=\ell(\mathbf{s})\ge 3$. Since $k$ is in the support of $\mathbf{s}$, we may suppose without loss of generality that $k$ is also in the support of $\mathbf{a}_r$. Then $\mathbf{s}-\mathbf{e}_k=\mathbf{a}_1+\dots+\mathbf{a}_{r-1}+(\mathbf{a}_r-\mathbf{e}_k)\in S$, which in particular means that $\mathbf{e}_k$ is a vertex of $\mathbf{G}_{\mathbf{s}}$. We prove that any other vertex of this graph is connected to $\mathbf{e}_k$.

Let $\mathbf{a}\in \mathcal{A}(S)$ such that $\mathbf{s}-\mathbf{a}\in S$ and $\mathbf{a}\neq \mathbf{e}_k$. Recall that by Lemma~\ref{lem:basic-facts-core}, either $\mathbf{a}\in \mathcal{A}(\mathsf{C}(S))$, or $\mathbf{a}=\mathbf{e}_l$ for some $l\in L\setminus\{k\}$, or $\mathbf{a}=\mathbf{e}_l+\mathbf{h}$ for some $l\in L$ and $\mathbf{h}\in \mathcal{H}(S)$. 
\begin{itemize}
    \item If either $\mathbf{a}\in \mathsf{C}(S)$ or $\mathbf{a}=\mathbf{e}_l$ or $\mathbf{a}=\mathbf{e}_l+\mathbf{h}$, with $l \in L\setminus\{k\}$ and $\mathbf{h}\in \mathcal{H}(S)$, then $\mathbf{s}-\mathbf{a}=\mathbf{a}_1'+\dots+\mathbf{a}_n'$ for some $\mathbf{a}_1',\dots,\mathbf{a}_n'\in \mathcal{A}(S)$, and as $\ell(\mathbf{s})\ge 3$, we have $n\ge 2$. Notice that $k$ is in the support of $\mathbf{s}$ but it is not in the support of $\mathbf{a}$, and so $k$ is in the support of $\mathbf{s}-\mathbf{a}$. This means that we can suppose without loss of generality that $k$ is in the support of $\mathbf{a}_n'$. Hence, $\mathbf{s}-\mathbf{a}-\mathbf{e}_k=\mathbf{a}_1'+\dots+\mathbf{a}_{n-1}'+(\mathbf{a}_n'-\mathbf{e}_k)\in S$. In particular, $\mathbf{a}$ is connected to $\mathbf{e}_k$ in $\mathbf{G}_{\mathbf{s}}$.
    \item If $\mathbf{a}=\mathbf{e}_k+\mathbf{h}$ with $\mathbf{h}\in \mathcal{H}(S)$, then $\mathbf{s}-\mathbf{e}_k-\mathbf{h}=\mathbf{a}_1'+\dots+\mathbf{a}_n'$ with $n\ge 2$ (recall that $\ell(\mathbf{s})\ge 3$). Hence, $\mathbf{s}=\mathbf{e}_k+\mathbf{h}+\mathbf{a}_1'+\dots+\mathbf{a}_n'$. We easily deduce that $\mathbf{a}\mathbf{a}_1'$ and $\mathbf{e}_k\mathbf{a}_1'$ are edges of $\mathbf{G}_{\mathbf{s}}$, which means that $\mathbf{a}$ and $\mathbf{e}_k$ are connected. 
\end{itemize}
This proves that $\mathbf{G}_{\mathbf{s}}$ is connected, which concludes the proof.
\end{proof}

\begin{remark}\rm
Observe that with the notation of the previous result we have $B_1 \cup B_2 \cup B_3 \cup B_4 \subseteq 2\mathcal{A}(S)$. In particular, the maximum minimal length of a Betti element in $S$ is attained by a Betti element in $\mathsf{C}(S)$.

\end{remark}

\section{Catenary degree}\label{sec:catenary-degree}

Let $S$ be a submonoid of $\mathbb{N}^{(I)}$. We can associate to any factorization $\sum_{\mathbf{a}\in \mathcal{A}(S)} \lambda_{\mathbf{a}} \mathbf{a}$ of an element $\mathbf{s}$ in $S$ the tuple $(\lambda_{\mathbf{a}})_{\mathbf{a} \in \mathcal{A}(S)}$. We are thus identifying $\mathcal{F}$, the free monoid on $\mathcal{A}(S)$, with $\mathbb{N}^{(\mathcal{A}(S))}$. In particular, we can write $\mathsf{Z}(\mathbf{s})=\{ (\lambda_{\mathbf{a}})_{\mathbf{a} \in \mathcal{A}(S)} : \sum_{\mathbf{a}\in \mathcal{A}(S)} \lambda_{\mathbf{a}} \mathbf{a}=\mathbf{s}\}$ and $|(\lambda_{\mathbf{a}})_{\mathbf{a} \in \mathcal{A}(S)}|=\lVert(\lambda_{\mathbf{a}})_{\mathbf{a} \in \mathcal{A}(S)}\rVert_1=\sum_{\mathbf{a}\in \mathcal{A}(S)} \lambda_{\mathbf{a}}$. Given $\mathbf{u}=(\lambda_{\mathbf{a}})_{\mathbf{a} \in \mathcal{A}(S)}$ and $\mathbf{v}=(\mu_{\mathbf{a}})_{\mathbf{a} \in \mathcal{A}(S)}$, define the \emph{distance} between $\mathbf{u}$ and $\mathbf{v}$ as 
\[
\operatorname{d}(\mathbf{u},\mathbf{v})=\max\{|\mathbf{u}-(\mathbf{u}\wedge \mathbf{v})|,|\mathbf{v}-(\mathbf{u}\wedge \mathbf{v})|\}. 
\]

If $\mathbf{u}$ and $\mathbf{v}$ are factorizations of $\mathbf{s}$, then an $N$-\emph{chain} joining $\mathbf{u}$ and $\mathbf{v}$ is a sequence $\mathbf{u}_1,\dots,\mathbf{u_n}\in \mathsf{Z}(\mathbf{s})$ such that $\mathbf{u}_1=\mathbf{u}$, $\mathbf{u}_n=\mathbf{v}$, and $\operatorname{d}(\mathbf{u}_i,\mathbf{u}_{i+1})\le N$ for all $i\in \{1,\dots,n-1\}$.

The \emph{catenary degree} of $\mathbf{s}$, $\mathsf{c}(\mathbf{s})$, is the minimum positive integer $N$ such that for any two factorizations of $\mathbf{s}$ there exists an $N$-chain joining them. The catenary degree of $S$ is defined as 
\[
\mathsf{c}(S)=\sup\{ \mathsf{c}(\mathbf{s}) : \mathbf{s}\in S\}.
\]

Let $\mathbf{z},\mathbf{z}'$ be two factorizations of $\mathbf{s}\in S$. We say that $\mathbf{z}$ and $\mathbf{z}'$ are $\mathcal{R}$-related if there exists $\mathbf{z}_1,\dots,\mathbf{z}_n\in \mathsf{Z}(\mathbf{s})$ such that $\mathbf{z}_1=\mathbf{z}$, $\mathbf{z}_n=\mathbf{z}'$, and $\mathbf{z}_i\wedge \mathbf{z}_{i+1}$ is not zero. The $\mathcal{R}$ relation is an equivalence relation; its equivalence classes are known as the $\mathcal{R}$-\emph{classes of factorizations} of $\mathbf{s}$. A factorization $\mathbf{z}$ of $\mathbf{s}$ is called \emph{isolated} if $\{\mathbf{z}\}$ is an $\mathcal{R}$-class of $\mathsf{Z}(\mathbf{s})$, that is, $\mathbf{z}$ is not $\mathcal{R}$-related to any other factorization of $\mathbf{s}$.

The number of $\mathcal{R}$-classes of factorizations of $\mathbf{s}$ is precisely the number of connected components of $\mathbf{G}_{\mathbf{s}}$ (see for instance \cite[Proposition~9.7]{fg}; here we are using that the number of atoms in $\operatorname{B}(\mathbf{s})$ are those that can be vertices of $\mathbf{G}_\mathbf{s}$ and are the only ones that can appear in the factorizations of $\mathbf{s}$; and $\mathcal{A}(S)\cap\operatorname{B}(\mathbf{s})$ has finite cardinality).

It is well known (see for instance \cite[Section~3]{ph}; the number of atoms below an element is finite and we could also apply \cite[Theorem~3.1]{cgspr}) that if $\mathcal{R}_1,\dots,\mathcal{R}_n$ are the different $\mathcal{R}$-classes of factorizations of $\mathbf{s}$, 
\begin{equation}\label{eq:catenary-r-classes}
\mathsf{c}(\mathbf{s})=\max\{ \min\{ |\mathbf{z}| : \mathbf{z}\in \mathcal{R}_i\}: i\in \{1,\dots,n\}\},
\end{equation}
and that (see for instance \cite{bgs})
\[
\mathsf{c}(S)=\sup\{\mathsf{c}(\mathbf{b}) : \mathbf{b}\in \operatorname{Betti}(S)\}.
\]





We are going to use \eqref{eq:catenary-r-classes} to compute an upper bound for $\mathsf{c}(S)$. To this end, we first need to bound the maximal length of the isolated factorizations of the Betti elements of $S$.

\begin{lemma}\label{lem:unique-fact-l-ge-3}
Let $S$ be an ideal extension of  $\mathbb{N}^{(I)}$, $\emptyset\neq I\subseteq\mathbb{N}$, and let $L=\{ i \in I : \mathbf{e}_i \in S\}$. 
Let $\mathbf{s}\in \mathsf{C}(S)$ such that $\ell(\mathbf{s})\geq 3$. Then $\mathbf{s}$ has a unique factorization if and only if  $\mathbf{s}=7\mathbf{e}_i$ for some $i\in I\setminus L$ such that $2\mathbf{e}_i\in \mathsf{C}(S)$.
\end{lemma}
\begin{proof}
We start by proving the sufficiency. It is clear that if $2\mathbf{e}_i\in \mathsf{C}(S)$, then $2\mathbf{e}_i\in \mathcal{M}(S)$ and by Lemma~\ref{lem:a+ei-l-le-3}, $3\mathbf{e}_i$ is also an atom of $S$. If follows that $7\mathbf{e}_i=2\mathbf{e}_i+2\mathbf{e}_i+3\mathbf{e}_i$ is the unique factorization of $7\mathbf{e}_i$. Observe that in this setting $\ell(\mathbf{s})=3$.

For the converse, 
express $\mathbf{s}$ as $\mathbf{s}=\mathbf{a}_1+\dots+\mathbf{a}_n$, with $\mathbf{a}_1,\dots,\mathbf{a}_n\in \mathcal{A}(S)$ and $n=\ell(\mathbf{s})\ge 3$. 
If $\mathbf{s}$ has a unique factorization, then the same happens with $\mathbf{t}=\mathbf{a}_1+\mathbf{a}_2+\mathbf{a}_3$, and as $\ell(\mathbf{s})=n$, we have that $\ell(\mathbf{t})=3$. Moreover, $\mathbf{a}_i+\mathbf{a}_j$ has unique factorization for every $i,j\in \{1,2,3\}$ and $i\neq j$.

If $\mathbf{a}_1$ is not minimal, then for $\mathbf{a}_1+\mathbf{a}_2$ to have a unique minimal factorization, according to Lemma~\ref{lem:unique_fctr-2A-CS}, $\mathbf{a}_1=\mathbf{m}+\mathbf{e}_i$ for some $\mathbf{m}\in \mathcal{M}(S)$, $i\in I\setminus L$, and $\mathbf{a}_2=\mathbf{m}$, and then $\mathbf{t}=(\mathbf{m}+\mathbf{e}_i)+\mathbf{m}+\mathbf{a}_3$. If $\mathbf{a}_3$ is not minimal, then by applying Lemma~\ref{lem:unique_fctr-2A-CS} to $\mathbf{m}+\mathbf{a}_3$, we deduce that there exists $j\in I\setminus L$ such that $\mathbf{a}_3=\mathbf{m}+\mathbf{e}_j$. Thus, $\mathbf{t}=(\mathbf{m}+\mathbf{e}_i)+(\mathbf{m}+\mathbf{e}_j)+\mathbf{m}$. But then, according to Lemma~\ref{lem:unique_fctr-2A-CS}, $\mathbf{a}_1+\mathbf{a}_3$ will not have a unique factorization, which is impossible. This forces $\mathbf{a}_3$ to be minimal, and as $\mathbf{a}_1+\mathbf{a}_3$ has a unique factorization we in addition have that $\mathbf{a}_3=\mathbf{m}$. Therefore $\mathbf{t}=(\mathbf{m}+\mathbf{e}_i)+\mathbf{m}+\mathbf{m}$. If there exists $k\in I\setminus \{i\}$ such that $\mathbf{m}-\mathbf{e}_k\in \mathbb{N}^{(I)}$, then $k\not\in L$ and from $\mathbf{t}=(\mathbf{m}+\mathbf{e}_k)+(2\mathbf{m}+\mathbf{e}_i-\mathbf{e}_k)$ we obtain a different factorization of $\mathbf{t}$, since $\mathbf{m}+\mathbf{e}_k\in \mathcal{A}(S)$ (Lemma~\ref{lem:a+ei-l-le-3}) and $2\mathbf{m}+\mathbf{e}_i-\mathbf{e}_k\in S$. If for all $k\in I\setminus \{i\}$ we have $\mathbf{m}-\mathbf{e}_k\notin \mathbb{N}^{(I)}$, then $\mathbf{m}=m\mathbf{e}_i$ for some positive integer $m$, and as $i\not\in L$, $m\geq 2$. If $m\ge 3$, $(m+2)\mathbf{e}_i\in \mathcal{A}(S)$ and $\mathbf{s}=(2m-1)\mathbf{e}_i+(m+2)\mathbf{e}_i$ is a another factorization of $\mathbf{t}$, a contradiction. Thus $m=2$, and $\mathbf{t}=7\mathbf{e}_i$. 

We can argue analogously if $\mathbf{a}_2$ or $\mathbf{a}_3$ are not minimal elements of $S^*$. So, the only remaining case is $\mathbf{t}=\mathbf{m}_1+\mathbf{m}_2+\mathbf{m}_3$ with $\mathbf{m}_1,\mathbf{m}_2,\mathbf{m}_3\in \mathcal{M}(S)$. We can assume that $\mathbf{m_3}-\mathbf{e}_i\in \mathbb{N}^{(I)}$ for some $i\in I$, and $i\not\in L$ ($\mathbf{a}_i+\mathbf{a}_j\in \mathsf{C}(S)$). Then, from 
$\mathbf{t}=(\mathbf{m}_1+\mathbf{e}_i)+(\mathbf{m}_2+\mathbf{m}_3-\mathbf{e}_i)$ we obtain a different factorization of 
$\mathbf{t}$, since $\mathbf{m}_1+\mathbf{e}_i\in \mathcal{A}(S)\setminus \mathcal{M}(S)$ (Lemma~\ref{lem:a+ei-l-le-3}) and $\mathbf{m}_2+\mathbf{m}_3-\mathbf{e}_i\in S$ (and in the first factorization only elements in $\mathcal{M}(S)$ are involved).

This proves that if $\mathbf{s}$ has unique factorization, then $\mathbf{a}_1+\mathbf{a}_2+\mathbf{a}_3=7\mathbf{e}_i$, and we can choose $\mathbf{a}_1=3\mathbf{e}_i$ and $\mathbf{a}_2=\mathbf{a}_3=2\mathbf{e}_i$. If $n\ge 4$, the same argument shows that $\mathbf{a}_1+\mathbf{a}_2+\mathbf{a}_4=7\mathbf{e}_i$, proving in this way that $\mathbf{a}_3=\mathbf{a}_4=2\mathbf{e}_i$. However, $\mathbf{a}_1+\mathbf{a}_2+\mathbf{a}_3+\mathbf{a}_4=9\mathbf{e}_i$ has two different factorizations $3\mathbf{a}_1=3(3\mathbf{e}_i)=3\mathbf{e}_i+3(2\mathbf{e}_i)=\mathbf{a}_1+3\mathbf{a}_2$, a contradiction. We conclude that $n=3$ and $\mathbf{s}=7\mathbf{e}_i$.
\end{proof}

\begin{lemma}\label{lem:isolated-more-4-GA}
Let $S\subseteq \mathbb{N}^{(I)}$ be a gap absorbing monoid. Let $\mathbf{s}\in \operatorname{Betti}(S)$ and $\mathbf{z}\in \mathsf{Z}(\mathbf{s})$. If $\lvert \mathbf{z} \rvert\geq 4$, then $\mathbf{z}$ is not isolated.
\end{lemma}
\begin{proof}
By Proposition~\ref{prop:gap-absorbing-implies-ideal}, we know that $S$ is an ideal extension of $\mathbb{N}^{(I)}$. We distinguish two cases depending on if $\mathbf{s}$ is in $\mathsf{C}(S)$ or not. Let $L=\{i\in I : \mathbf{e}_i \in S\}$.

Suppose we have a factorization $\mathbf{z}$ that corresponds to the expression $\mathbf{s}=\mathbf{a}_1+\cdots+\mathbf{a}_n$, with $\mathbf{a}_i\in \mathcal{A}(S)$ for all $i\in \{1,\ldots,n\}$ and $n\geq 4$, and assume that $\mathbf{z}$ is isolated.

If $\mathbf{s}\in \mathsf{C}(S)$, and $\mathbf{z}$ is isolated, then $\mathbf{a}_1+\mathbf{a}_2+\mathbf{a}_3$ must have a unique expression, and so by Lemma~\ref{lem:unique-fact-l-ge-3} $2\mathbf{e}_i\in S$, $i\not\in L$, and $\mathbf{a}_1+\mathbf{a}_2+\mathbf{a}_3=7\mathbf{e}_i$. Moreover, from the proof of Lemma~\ref{lem:unique-fact-l-ge-3} we can assume that $\mathbf{a}_1=3\mathbf{e}_i$ and $\mathbf{a}_2=2\mathbf{e}_i=\mathbf{a}_3$; $\mathbf{a}_1$ is the only atom in the expression $\mathbf{a}_1+\mathbf{a}_2+\mathbf{a}_3$ that is not minimal. Also, from the proof of Lemma~\ref{lem:unique-fact-l-ge-3}, we deduce that $\mathbf{a}_1+\mathbf{a}_2+\mathbf{a}_4$ is also of the form $7\mathbf{e}_j$ for some $j$, but as $\operatorname{Supp}(\mathbf{a}_1)=\{i\}$, we derive that $i=j$. In particular $\mathbf{a}_4=2\mathbf{e}_i$. From $3\mathbf{a}_1=\mathbf{a}_1+3\mathbf{a}_2=\mathbf{a}_1+\mathbf{a}_2+\mathbf{a}_3+\mathbf{a}_4$ we deduce that $\mathbf{z}$ cannot be an isolated factorization.

Now, suppose that $\mathbf{s}\not \in \mathsf{C}(S)$, that is, there exists $l\in L$ such that $\mathbf{s}\in \mathbf{e}_l+\mathbb{N}^{(I)}$. As $l\in \operatorname{Supp}(\mathbf{s})$, there exists $i\in \{1,\dots,n\}$ such that $l\in \operatorname{Supp}(\mathbf{a}_i)$. We can assume without loss of generality that $i=n$, and so $\mathbf{s}-\mathbf{e}_l=\mathbf{a}_1+\dots+\mathbf{a}_{n-1}+(\mathbf{a}_n-\mathbf{e}_l)\in S$. Write $\mathbf{a}_{n-1}+(\mathbf{a}_n-\mathbf{e}_l)=\mathbf{b}_1+\dots+\mathbf{b}_k$ for some $\mathbf{b}_1,\dots,\mathbf{b}_k\in \mathcal{A}(S)$. Then $\mathbf{s}=\mathbf{e}_l+\mathbf{a}_1+\dots+\mathbf{a}_{n-2}+\mathbf{b}_1+\dots+\mathbf{b}_k$. Since $\mathbf{z}$ is an isolated factorization, $k=1$, $\mathbf{b}_1=\mathbf{a}_{n-1}$, and $\mathbf{a}_n=\mathbf{e}_l$, so $\mathbf{s}=\mathbf{a}_1+\dots+\mathbf{a}_{n-1}+\mathbf{e}_l$. As $\mathbf{z}$ is isolated and $\mathsf{Z}(\mathbf{s})$ has more than one $\mathcal{R}$-class, there exists an expression of $\mathbf{s}$ of the form $\mathbf{s}=\mathbf{a}_1'+\dots+\mathbf{a}_m'$ with $\{\mathbf{a}_1',\dots,\mathbf{a}_m'\}\subseteq \mathcal{A}(S)\setminus\{\mathbf{e}_l,\mathbf{a}_1,\dots,\mathbf{a}_{n-1}\}$ and $m\ge 2$. As $l\in \operatorname{Supp}(\mathbf{s})$ we can suppose without loss of generality that $l$ is in the support of $\mathbf{a}_m'$, and then by Lemma~\ref{lem:basic-facts-core}, $\mathbf{a}_m'=\mathbf{e}_l+\mathbf{h}$ with $\mathbf{h}\in \{\mathbf{0}\}\cup \mathcal{H}(S)$. Hence,  $\mathbf{s}=\mathbf{e}_l+\mathbf{a}_1'+\dots+(\mathbf{a}_{m-1}'+\mathbf{h})$. 
Notice that as $\mathbf{z}$ is isolated, we must have $m=2$ and $\mathbf{h}\neq \mathbf{0}$; in fact, it is trivial to see that $\mathbf{h}$ must be nonzero, and if $m>2$, then, as $\mathbf{a}_{m-1}'+\mathbf{h}\in S$, we can write $\mathbf{s}=\mathbf{e}_l+\mathbf{a}_1'+\cdots+\mathbf{a}_{m-2}'+\mathbf{b}_1'+\cdots+\mathbf{b}_r'$, with $\mathbf{b}_1',\ldots,\mathbf{b}_r'\in \mathcal{A}(S)$, and as $\mathbf{a}_i'\notin\{\mathbf{e}_l,\mathbf{a}_1,\ldots,\mathbf{a}_{n-1}\}$ we obtain a different factorization of $\mathbf{s}$ in the same $\mathcal{R}$-class of $\mathbf{z}$. Therefore,  $\mathbf{s}=\mathbf{e}_l+\mathbf{a}_1'+\mathbf{h}$. Since $S$ is gap absorbing, $\mathbf{a}_1'+\mathbf{h}\in \mathcal{A}(S)\cup 2\mathcal{A}(S)$. So either $\mathbf{s}=\mathbf{e}_l+\mathbf{a}_1+\dots+\mathbf{a}_{n-1}=\mathbf{e}_l+\mathbf{c}_1$ or  $\mathbf{s}=\mathbf{e}_l+\mathbf{a}_1+\dots+\mathbf{a}_{n-1}=\mathbf{e}_l+\mathbf{c}_1+\mathbf{c}_2$, with $\mathbf{c}_1,\mathbf{c}_2\in\mathcal{A}(S)$. By using once more the fact that the factorization $\mathbf{z}$ is isolated, $n-1\in\{1,2\}$, which contradicts the hypothesis $n\ge 4$.
\end{proof}

\begin{remark}
    Let $S$ be an ideal extension of $\mathbb{N}^{(I)}$. Observe that Lemma~\ref{lem:isolated-more-4-GA} holds even if $S$ is not gap absorbing in the case $S=\mathsf{C}(S)$, since in the first part of the proof we did not use the assumption of gap absorbing.
\end{remark}

\begin{theorem}\label{th:bound-catenary}
    Let $S$ be a gap absorbing monoid. Then $\mathsf{c}(S)\le 4$.
\end{theorem} 
\begin{proof}
    Let $\mathbf{b}\in \operatorname{Betti}(S)$. According to Theorem~\ref{thm:lge4-not-betti}, $\ell(\mathbf{b})\le 3$. Let $\mathbf{z}=(z_\mathbf{a})_{\mathbf{a}\in \mathcal{A}(S)}$ be a factorization of $\mathbf{b}$. If $\mathbf{z}$ is isolated, by Lemma~\ref{lem:isolated-more-4-GA} and Proposition~\ref{prop:gap-absorbing-implies-ideal}, $|\mathbf{z}|\le 3$. If $\mathbf{z}$ is not isolated, then there exists $\mathbf{z}'=(z_\mathbf{a}')_{\mathbf{a}\in \mathcal{A}(S)}\in \mathsf{Z}(\mathbf{b})$ and $\mathbf{a}\in \mathcal{A}(S)$ such that $z_\mathbf{a}\neq 0 \neq z_\mathbf{a}'$. In particular $\mathbf{b}-\mathbf{a}\in S$, and by Proposition~\ref{prop:gap-absorbing-eq-non-decreasing-length}, $\ell(\mathbf{b}-\mathbf{a})\le 3$. Thus, there exists $\mathbf{z}''\in \mathsf{Z}(\mathbf{b}-\mathbf{a})$ such that $|\mathbf{z}''|\le 3$. Notice that $\mathbf{z}''+\mathbf{e}_\mathbf{a}$ is a factorization of $\mathbf{b}$ that is in the same $\mathcal{R}$-class of $\mathbf{z}$ with length at most four. From \eqref{eq:catenary-r-classes} we derive that $\mathsf{c}(\mathbf{b})\le 4$, and consequently $\mathsf{c}(S)\le 4$.
\end{proof}

As a consequence of the proof of the previous theorem we obtain the following.

\begin{corollary}\label{cor:catenary-three}
    Let $S$ be a gap absorbing monoid. If $\operatorname{Betti}(S)\subseteq 2\mathcal{A}(S)$, then $\mathsf{c}(S)\le 3$. 
\end{corollary}




\section{Delta sets}\label{sec:delta}

Let $I$ be a set of non-negative integers and let $S$ be a submonoid of $\mathbb{N}^{(I)}$. Recall that for $\mathbf{s}\in S$, with $\mathsf{L}(\mathbf{s})=\{l_1<l_2< \dots <l_n\}$ (we know that $S$ is a BF-monoid), the \emph{Delta set} of $\mathbf{s}$ is defined as $\Delta(\mathbf{s})=\{l_{i+1}-l_i : i\in \{1,\dots,n-1\}\}$. In particular, the set of lengths of $\mathbf{s}$ is an interval of integers if and only if $\Delta(\mathbf{s})\subseteq \{1\}$. The Delta set of $S$ is defined as the union of all the Delta sets of its elements. 

It is well known (see for instance \cite[Theorem~1.6.3]{g-hk}) that $2+\sup\Delta(S)\le \mathsf{c}(S)$. By Theorem~\ref{th:bound-catenary}, we have that if $S$ is gap absorbing, then $\max\Delta(S)\le 2$. Next we prove that as a matter of fact $\max\Delta(S)\le 1$, which means that every set $\mathsf{L}(\mathbf{s})$ with $\mathbf{s}\in S$ is an interval.

\begin{theorem}\label{th:Ls-interval}
Let $S$ be a gap absorbing monoid. Then, $\mathsf{L}(\mathbf{s})$ is an interval for every $\mathbf{s}\in S$.
\end{theorem}
\begin{proof}
    We start by proving that if $\ell(\mathbf{s})=2$, then $\mathsf{L}(\mathbf{s})$ is an interval. Suppose, on the contrary, that there is some $\mathbf{s}\in S$ with $\ell(\mathbf{s})=2$ and such that $\mathsf{L}(\mathbf{s})$ is not an interval; we may also assume that $\mathbf{s}$ is minimal with respect to $\le_S$ having this condition. Let $L>2$ be an element of $\mathsf{L}(\mathbf{s})$ such that $L-1\not\in \mathsf{L}(\mathbf{s})$ (whence $L-1>2$). Then, there exist $\mathbf{a}_1,\mathbf{a}_2\in \mathcal{A}(S)$ and $\mathbf{b}_1,\dots,\mathbf{b}_L\in \mathcal{A}(S)$ such that $\mathbf{s}=\mathbf{a}_1+\mathbf{a}_2=\mathbf{b}_1+\dots+\mathbf{b}_L$. 
    Hence,  $\mathbf{b}_1+\mathbf{b}_2\le \mathbf{s}-\mathbf{b}_L\le \mathbf{s}$. Both $\mathbf{b}_1+\mathbf{b}_2$ and $\mathbf{s}$ are in $2\mathcal{A}(S)\setminus \mathcal{A}(S)$. From Proposition~\ref{prop:2A-interval-eq}, we obtain that $\mathbf{s}-\mathbf{b}_L\in 2\mathcal{A}(S)\setminus \mathcal{A}(S)$. In particular, $\ell(\mathbf{s}-\mathbf{b}_L)=2$ and $\mathbf{s}-\mathbf{b}_L\le_S \mathbf{s}$. By the minimality of $\mathbf{s}$, we deduce that $\mathsf{L}(\mathbf{s}-\mathbf{b}_L)$ is an interval, and as $2,L-1\in \mathsf{L}(\mathbf{s}-\mathbf{b}_L)$, we get $L-2\in \mathsf{L}(\mathbf{s}-\mathbf{b}_L)$, which translates to $L-1\in \mathsf{L}(\mathbf{s})$, a contradiction. Thus, for every $\mathbf{s} \in S$ with $\ell(\mathbf{s})=2$ we have that $\mathsf{L}(\mathbf{s})$ is an interval.

    Now we show that if $\ell(\mathbf{s})=3$, then $\mathsf{L}(\mathbf{s})$ is an interval. We use a similar argument to the one employed above. Suppose that $\mathbf{s}$ is minimal with respect to $\le_S$ such that $\ell(\mathbf{s})=3$ and $\mathsf{L}(\mathbf{s})$ is not an interval. Take $L>3$ to be an element of $\mathsf{L}(\mathbf{s})$ such that $L-1\not\in \mathsf{L}(\mathbf{s})$. There exists $\mathbf{a}_1,\mathbf{a}_2,\mathbf{a}_3\in \mathcal{A}(S)$ and $\mathbf{b}_1,\dots,\mathbf{b}_L\in \mathcal{A}(S)$ such that $\mathbf{s}=\mathbf{a}_1+\mathbf{a}_2+\mathbf{a}_3=\mathbf{b}_1+\dots+\mathbf{b}_L$. Then $\mathbf{b}_1+\mathbf{b}_2+\mathbf{b}_3\le \mathbf{s}-\mathbf{b}_L\le \mathbf{s}$. Notice that $\mathbf{s} \in 3\mathcal{A}(S)\setminus(\mathcal{A}(S)\cup 2\mathcal{A}(S))$. Clearly, $\mathbf{b}_1+\mathbf{b}_2+\mathbf{b}_3\not\in\mathcal{A}(S)$. If $\mathbf{b}_1+\mathbf{b}_2+\mathbf{b}_3\in 2\mathcal{A}(S)$, say $\mathbf{b}_1+\mathbf{b}_2+\mathbf{b}_3=\mathbf{c}_1+\mathbf{c}_2$, with $\mathbf{c}_1,\mathbf{c}_2\in\mathcal{A}(S)$, we have that $\mathbf{s}=\mathbf{c}_1+\mathbf{c}_2+\mathbf{b}_4+\dots+\mathbf{b}_L$, which is a factorization of length $L-1$, and thus a contradiction. Hence, $\mathbf{b}_1+\mathbf{b}_2+\mathbf{b}_3\in 3\mathcal{A}(S)\setminus(\mathcal{A}(S)\cup 2\mathcal{A}(S))$. By Proposition~\ref{prop:2A-interval-eq}, we infer that $\mathbf{s}-\mathbf{b}_L\in 3\mathcal{A}(S)\setminus(\mathcal{A}(S)\cup 2\mathcal{A}(S))$, and in particular $\ell(\mathbf{s}-\mathbf{b}_L)=3$ and $\mathbf{s}-\mathbf{b}_L\le_S \mathbf{s}$. The minimality of $\mathbf{s}$ forces $\mathsf{L}(\mathbf{s}-\mathbf{b}_L)$ to be an interval. From $3,L-1\in \mathsf{L}(\mathbf{s}-\mathbf{b}_L)$ we derive that $L-2\in \mathsf{L}(\mathbf{s}-\mathbf{b}_L)$, and consequently $L-1\in\mathsf{L}(\mathbf{s})$, yielding once more a contradiction.

    As a consequence of Theorem~\ref{thm:lge4-not-betti}, if $\mathbf{s} \in S$ is a Betti element, then $\ell(\mathbf{s})\le 3$, and thus $\mathsf{L}(\mathbf{s})$ is an interval. In particular, this means that $\Delta(\mathbf{s})\subseteq \{1\}$, and as $S$ is a BF-monoid, by \cite[Theorem~2.5]{maxdelta}, $\max\Delta(S)\le 1$. This forces $\mathsf{L}(\mathbf{s})$ to be an interval for all $\mathbf{s} \in S$. 
\end{proof}

Observe that from Theorem~\ref{th:Ls-interval}, we easily derive that Conjecture~\ref{conj:ideal-implies-gap-absorbing} implies \cite[Conjecture~4.16]{Baeth}.



There are other submonoids of $\mathbb{N}^{(I)}$ with the property that the sets of lengths of its elements are intervals. Next, we provide another family that arises when weakening the conditions in Proposition~\ref{prop:2A-interval-eq}.

\begin{proposition}\label{prop:nA-closed-intervals-delta-1}
    Let $S$ be a submonoid of $\mathbb{N}^{(I)}$ such that for every positive $n$, the set $n\mathcal{A}(S)\setminus(n-1)\mathcal{A}(S)$ is closed under intervals. Then, $\max\Delta(S)\le 1$.
\end{proposition}
\begin{proof}
    Let us argue by contradiction. Let $\mathbf{s}\in S$ be minimal with respect to $\le_S$ such that $\mathsf{L}(\mathbf{s})$ is not an interval. So there exists $L\in \mathsf{L}(\mathbf{s})$ such that $L-1\not\in \mathsf{L}(\mathbf{s})$ and $L>\ell(\mathbf{s})=l$. We can find expressions of $\mathbf{s}$ of the form $\mathbf{s}=\mathbf{a}_1+\dots+\mathbf{a}_l=\mathbf{b}_1+\dots+\mathbf{b}_L$, with $\mathbf{a}_1,\dots,\mathbf{a}_l,\mathbf{b}_1,\dots,\mathbf{b}_L\in \mathcal{A}(S)$. As $l<L$, we have that $\mathbf{b}_1+\dots+\mathbf{b}_{l}\le_S \mathbf{s}-\mathbf{b}_L\le_S \mathbf{s}=\mathbf{a}_1+\dots+\mathbf{a}_{l}\in l\mathcal{A}(S)$. 
    If $\mathbf{b}_1+\dots+\mathbf{b}_l\in (l-1)\mathcal{A}(S)$, then $\mathbf{b}_1+\dots+\mathbf{b}_{l}=\mathbf{a}_1'+\dots+\mathbf{a}_{l-1}'$ with $\mathbf{a}_1',\dots,\mathbf{a}_{l-1}'\in \mathcal{A}(S)$ and therefore $\mathbf{s}=\mathbf{a}_1'+\dots+\mathbf{a}_{l-1}'+\mathbf{b}_{l+1}+\dots+\mathbf{b}_L\in (L-1)\mathcal{A}(S)$, a contradiction. Thus, $\mathbf{b}_1+\dots+\mathbf{b}_l\in l\mathcal{A}(S)\setminus (l-1)\mathcal{A}(S)$ and clearly $\mathbf{s}\in l\mathcal{A}(S)\setminus(l-1)\mathcal{A}(S)$. By hypothesis $\mathbf{s}-\mathbf{b}_L\in l\mathcal{A}(S)\setminus(l-1)\mathcal{A}(S)$, and as $\mathbf{s}$ was minimal not having an interval as set of lengths, we deduce that $\mathsf{L}(\mathbf{s}-\mathbf{b}_L)$ is an interval. Also $l$ and $L-1$ are in $\mathsf{L}(\mathbf{s}-\mathbf{b}_L)$, and so $L-2\in \mathsf{L}(\mathbf{s}-\mathbf{b}_L)$, which means that $L-1\in\mathsf{L}(\mathbf{s})$, a contradiction.
\end{proof}

\begin{example}
    Let $S=\langle \llbracket (20,3),(23,4)\rrbracket\rangle$. Let us double check with \texttt{numericalsgps} that $\max(\Delta(S))=1$.
 \begin{verbatim}
gap> g:=Union(List([20..23],i->[i,3]),List([20..23],i->[i,4]));
[ [ 20, 3 ], [ 20, 4 ], [ 21, 3 ], [ 21, 4 ], 
  [ 22, 3 ], [ 22, 4 ], [ 23, 3 ], [ 23, 4 ] ]
gap> a:=AffineSemigroup(g);;
gap> DeltaSet(a);
[ 1 ]
\end{verbatim}   

Notice that this example also shows that  the condition on $2\mathcal{A}(S)$ being closed under intervals does not imply that $\operatorname{Betti(S)}\subseteq 2\mathcal{A}(S)$. Observe that $(161,28)\in \operatorname{Betti}(S)$, and $\ell((161,28))=7$.
 \begin{verbatim}
gap> g:=Union(List([20..23],i->[i,3]),List([20..23],i->[i,4]));
[ [ 20, 3 ], [ 20, 4 ], [ 21, 3 ], [ 21, 4 ], 
  [ 22, 3 ], [ 22, 4 ], [ 23, 3 ], [ 23, 4 ] ]
gap> a:=AffineSemigroup(g);;
gap> BettiElements(a);
[ [ 41, 7 ], [ 42, 6 ], [ 42, 7 ], [ 42, 8 ], [ 43, 6 ], [ 43, 7 ],
  [ 43, 8 ], [ 44, 6 ], [ 44, 7 ], [ 44, 8 ], [ 45, 7 ], [ 160, 24 ],
  [ 160, 25 ], [ 160, 26 ], [ 160, 27 ], [ 160, 28 ], [ 161, 24 ],
  [ 161, 25 ], [ 161, 26 ], [ 161, 27 ], [ 161, 28 ] ]
gap> Factorizations([161,28],a);
[ [ 0, 0, 0, 0, 0, 0, 0, 7 ], [ 3, 4, 1, 0, 0, 0, 0, 0 ], 
  [ 4, 3, 0, 1, 0, 0, 0, 0 ] ]     
\end{verbatim}   
\end{example}

There are not too many results in the literature dealing with elements in atomic monoids whose sets of lengths are intervals. For numerical semigroups, one can find \cite[Corollary~2.3]{a-c-al} with $k=1$, which is actually a particular case of Proposition~\ref{prop:nA-closed-intervals-delta-1}. For Krull monoids see \cite[Theorem~7.6.9]{g-hk}.

\section{Omega primality}\label{sec:omega-primality}

Let $S$ be a submonoid of $\mathbb{N}^{(I)}$, for some $I\subseteq \mathbb{N}$, and let $\mathbf{s}\in S$. The $\omega$\emph{-primality} of $\mathbf{s}$, denoted $\omega(\mathbf{s})$, is the least positive integer $N$ such that whenever $\mathbf{s} \le_S \mathbf{s}_1+\cdots+\mathbf{s}_n$ for some $\mathbf{s}_1,\ldots, \mathbf{s}_n\in S$, then $\mathbf{s} \le_S \sum_{j\in J} \mathbf{s}_j$ for some $J\subseteq \{1,\ldots, n\}$ with $|J|\le N$. 

The $\omega$-primality of $S$ is defined as 
\[
\omega(S)=\sup\{ \omega(\mathbf{a}) : \mathbf{a}\in \mathcal{A}(S)\}.
\]

 Observe that for any submonoid of $S$ of $\mathbb{N}^{(I)}$, and $\mathbf{a},\mathbf{b}\in \mathbb{N}^{(I)}$, if $\mathbf{a}\le_S \mathbf{b}$, then $\mathbf{a}\le \mathbf{b}$. We will take advantage of this fact to find upper bounds for the $\omega$-primality of an atom, and thus of the semigroup itself.

\begin{lemma}\label{lem:number-le-sum}
    Let $\mathbf{a},\mathbf{a}_1,\dots,\mathbf{a}_n\in \mathbb{N}^{(I)}$. Suppose that $\mathbf{a} \le \mathbf{a}_1+\dots+\mathbf{a}_n$. Then there exists $J\subseteq \{1,\dots,n\}$, with $|J|\le \lVert \mathbf{a}\rVert_1$, such that $\mathbf{a} \le \sum_{j\in J}\mathbf{a}_j$.
\end{lemma}
\begin{proof}
    If $n\le \lVert \mathbf{a}\rVert_1$, then take $J=\{1,\dots,n\}$. So, let us suppose that $n>\lVert \mathbf{a}\rVert_1$. Start with $J=\emptyset$. The support of $\mathbf{a}$ is included in $\cup_{i=1}^n \operatorname{Supp}(\mathbf{a}_i)$. Let $k\in \operatorname{Supp}(\mathbf{a})$, and let $j\in \{1,\dots,n\}$ be such that $k\in \operatorname{Supp}(\mathbf{a}_j)$. Then $\mathbf{a}-\mathbf{e}_k \le \mathbf{a}_1+\dots+\mathbf{a}_{j-1}+(\mathbf{a}_j-\mathbf{e}_k)+\mathbf{a}_{j+1}+\dots+\mathbf{a}_n$. Add $j$ to the set $J$. Now $\operatorname{Supp}(\mathbf{a}-\mathbf{e}_k) \subseteq \operatorname{Supp}(\mathbf{a}_1) \cup \dots \cup \operatorname{Supp}(\mathbf{a}_{j-1}) \cup \operatorname{Supp}(\mathbf{a}_j-\mathbf{e}_k)\cup \operatorname{Supp}(\mathbf{a}_{j+1})\cup \dots \cup \operatorname{Supp}(\mathbf{a}_n)$. For $j\in J$ and $k\in  \operatorname{Supp}(\mathbf{a}_j)$, denote by $\lambda_k\in \mathbb{N}$ the number of times $\mathbf{e}_k$ has been removed by $\mathbf{a}_j$ during the procedure described above. Set $J_j=\{k \in \operatorname{Supp}(\mathbf{a}_j): \lambda_k>0\}$. After repeating this procedure we will arrive to $\mathbf{0}\le \sum_{j\in J} (\mathbf{a}_j-\sum_{k\in J_j} \lambda_k \mathbf{e}_k) + \sum_{i\in \{1,\dots,n\}\setminus J} \mathbf{a}_i$, $\mathbf{a}=\sum_{j\in J}\sum_{k\in J_j}\lambda_k\mathbf{e}_k$, and $\mathbf{a} \le \sum_{j\in J}\mathbf{a}_j$.
\end{proof}

For the proof of the following lemma we are going to use \cite[Lemma~3.2]{bgsg}, which asserts that for the computation of $\omega(\mathbf{s})$ we can restrict our search to  sums of the form $\mathbf{a}_1+\cdots+\mathbf{a}_n$, with $\mathbf{a}_1,\ldots,\mathbf{a}_n\in \mathcal{A}(S)$ and $\mathsf{s}\le_S \mathbf{a}_1+\cdots+\mathbf{a}_n$.

\begin{theorem}\label{thm:upper-bound-omega-norm1}
    Let $S$ be an ideal extension of $\mathbb{N}^{(I)}$.
    For every $\mathbf{a}\in \mathcal{A}(S)$, $\omega(\mathbf{a})\le \lVert \mathbf{a}\rVert_1 +1$. In particular,
    \[
    \omega(S)\le 1+\sup_{\mathbf{a}\in \mathcal{A}(S)}\lVert \mathbf{a}\rVert_1.
    \]
\end{theorem}
\begin{proof}
    Let $\mathbf{a}_1,\dots,\mathbf{a}_n\in \mathcal{A}(S)$ be such that $\mathbf{a}\le_S \mathbf{a}_1+\dots+\mathbf{a}_n$ and $n>\lVert \mathbf{a} \rVert_1$. Then, $\mathbf{a}\le \mathbf{a}_1+\dots+\mathbf{a}_n$. By Lemma~\ref{lem:number-le-sum}, there exists $J\subseteq \{1,\dots,n\}$ such that $\mathbf{a}+\mathbf{x} = \sum_{j\in J}\mathbf{a}_j$ and $|J|\le \lVert \mathbf{a} \rVert_1$. As $n>\lVert \mathbf{a} \rVert_1\ge |J|$, there exists $i\in \{1,\dots,n\}\setminus J$. Hence $\mathbf{a}+\mathbf{a}_i+\mathbf{x} = \sum_{j\in J\cup\{i\}} \mathbf{a}_j$, and from $\mathbf{a}_i+\mathbf{x}\in S$ (Proposition~\ref{prop:gap-absorbing-implies-ideal}), we deduce that $\mathbf{a}\le_S \sum_{j\in J\cup\{i\}} \mathbf{a}_j$. By \cite[Lemma~3.2]{bgsg}, $\omega(\mathbf{a})\le \lVert \mathbf{a} \rVert_1+1$.
\end{proof}


\begin{example}
    Let $m$ be an integer greater than one, and let $S=\langle m,\dots,2m-1\rangle=\{0\}+(m+\mathbb{N})$. By \cite[Proposition 3.1]{algorithm_omega}, $\omega(S)=3<2m=1+\sup\{ x : x\in \{m,\dots,2m-1\}\}$. Thus, the upper bound given in Theorem~\ref{thm:upper-bound-omega-norm1} is far from being sharp. 
\end{example}

This example also highlights that \cite[Proposition~4.9]{Baeth} is wrong, since according to that result, the $\omega$-primality of $\langle m,\dots,2m-1\rangle$ should belong to $\{m,m+1\}$. 

\begin{proposition}\label{prop:omega-non-decreasing}
    Let $S$ be an ideal extension of $\mathbb{N}^{(I)}$.
    Let $\mathbf{a},\mathbf{b}\in \mathcal{A}(S)$ with $\mathbf{a}\le \mathbf{b}$. Then, $\omega(\mathbf{a})\le \omega(\mathbf{b})$.
\end{proposition}
\begin{proof}
    Suppose that $\mathbf{a}\le_S \mathbf{s}_1+\dots+\mathbf{s}_n$ with $\mathbf{s}_1,\dots,\mathbf{s}_n\in S$. Write $\mathbf{a}+\mathbf{x}=\mathbf{b}$ for some $\mathbf{x}\in \mathbb{N}^{(I)}$. There exists $\mathbf{s}\in S$ such that $\mathbf{a}+\mathbf{s}= \mathbf{s}_1+\dots+\mathbf{s}_n$, and consequently $\mathbf{a}+\mathbf{s}+\mathbf{x}=\mathbf{b}+\mathbf{s}=\mathbf{s}_1+\dots+\mathbf{s}_n+\mathbf{x}$. As $\mathbf{s}\in S$, we have that $\mathbf{b}\le_S \mathbf{s}_1'+\dots+\mathbf{s}_n'$, with $\mathbf{s}_i'=\mathbf{s}_i$ for $i<n$ and $\mathbf{s}_n'=\mathbf{s}_n+\mathbf{x}\in S$. Hence, there exists $J\subseteq \{1,\dots,n\}$ with $|J|\le \omega(\mathbf{b})$ such that $\mathbf{b}\le_S \sum_{j\in J}\mathbf{s}_j'$. Thus, $\mathbf{b}+\mathbf{s}'= \sum_{j\in J}\mathbf{s}_j'$ for some $\mathbf{s}'\in S$, and so $\mathbf{a}+\mathbf{x}+\mathbf{s}'=\sum_{j\in J}\mathbf{s}_j'$. If $n\in J$, then $\mathbf{x}$ cancels out from both sides of the equality and we get  $\mathbf{a}+\mathbf{s}'=\sum_{j\in J}\mathbf{s}_j$, which implies $\mathbf{a}\le_S \sum_{j\in J}\mathbf{s}_j$. If $n$ is not in $J$, then $\mathbf{a}+\mathbf{s}'+\mathbf{x}=\sum_{j\in J}\mathbf{s}_j$, and as $\mathbf{s}'+\mathbf{x}\in S$, we also derive that $\mathbf{a}\le_S \sum_{j\in J}\mathbf{s}_j$, concluding that $\omega(\mathbf{a})\le \omega(\mathbf{b})$.
\end{proof}

\begin{example}
Let $S=\{(0,0)\}\cup (\{(2,0),(0,3)\}+\mathbb{N}^2)$. The set of atoms of $S$ can be computed as follows:
\begin{verbatim}
gap> s:=FiniteComplementIdealExtension([2,0],[0,3]);
<Affine semigroup>
gap> MinimalGenerators(s);
[ [ 0, 3 ], [ 0, 5 ], [ 0, 4 ], [ 3, 0 ], [ 2, 0 ], [ 1, 3 ], [ 2, 1 ],
  [ 1, 4 ], [ 2, 2 ], [ 3, 1 ], [ 1, 5 ], [ 3, 2 ] ]
\end{verbatim}
By Proposition~\ref{prop:omega-non-decreasing}, in order to compute $\omega(S)$, it suffices to compute $\max\{\omega((3,2)), \omega((1,5))\}$.
\begin{verbatim}
gap> OmegaPrimality([3,2],s);
4
gap> OmegaPrimality([1,5],s);
5
\end{verbatim}
Thus $\omega(S)=5$, and the bound given by Theorem~\ref{thm:upper-bound-omega-norm1} is $7=1+\lVert (5,1)\rVert_1$. Observe that according to \cite[Proposition~4.9]{Baeth}, $\omega(S)$ should be in $\{3,4\}$.
\end{example}



It is well known, see for instance \cite[Section~3]{gk}, that $\mathsf{c}(S)\le \omega(S)$. In the following example we provide a monoid having an ``extreme'' behaviour with respect to $\omega$-primality and catenary degree (compare with \cite[Example~4.7]{bgsg}).




\begin{example} \label{ex:omega-infty}

    Let $S=\{(0,0)\}\cup ((0,1)+\mathbb{N}^2)$. 
    The set of atoms of $S$ is $\mathcal{A}(S)=\{(n,1) : n\in \mathbb{N}\}$. Take $n$ a positive integer. Define $s_n=\frac{n(n+1)}{2}=\sum_{i=1}^n i$. Then $(s_n,1)+(n-1)(0,1)=\sum_{i=1}^n(i,1)$, and so $(s_n,1)\le_S \sum_{i=1}^n(i,1)$. If $J$ is a proper subset of $\{1,\dots,n\}$, then $s_n>\sum_{j\in J} j$, and so $(s_n,1)\not\le \sum_{j\in J}(j,1)$, which in particular implies $(s_n,1)\not\le_S \sum_{j\in J}(j,1)$. Thus $\omega((s_n,1))\ge n$, and consequently $\omega(S)=\infty$.

     According to Remark~\ref{rem:betti-base-case}, $\operatorname{Betti}(S)\subseteq 2\mathcal{A}(S)$, and so $\operatorname{Betti}(S)\subseteq \{(n,2) : n\in \mathbb{N}\}$. In particular, this implies that the maximal length of a factorization of a Betti element of $S$ is two, which by \cite[Section~3]{ph} forces $\mathsf{c}(S)\le 2$. The factorizations of $(4,2)$ are $(2,1)+(2,1)$ and $(1,1)+(3,1)$, and so $\mathsf{c}((4,2))=2=\mathsf{c}(S)$. By \cite[Theorem~1.6.3]{g-hk}, $S$ is half-factorial; this also follows from \cite[Proposition~1]{kl}, since the set of atoms is contained in a hyperplane.
\end{example}

Notice that $\{(0,0)\}+((0,1)+\mathbb{N}^2)$ is half-factorial but not factorial. If $S=\{\mathbf{0}\}\cup (X+\mathbb{N}^k)$ with $k$ a positive integer and $|\mathbb{N}^k\setminus S|<\infty$, then \cite[Theorem~4.4]{Baeth} asserts that $S$ is half-factorial if an only if it is factorial.

\begin{proposition}\label{prop:lower-omega-disjoint-supp}
Let $S$ be an extension ideal of $\mathbb{N}^{(I)}$. Let $\mathbf{a}\in \mathcal{A}(S)$ and suppose there exists $\mathbf{b}\in S^*$ such that $\mathbf{a}\wedge \mathbf{b}=\mathbf{0}$. Then $\omega(\mathbf{a})\geq \lVert \mathbf{a} \rVert_1$.
\end{proposition}
\begin{proof}
 Let $J=\operatorname{Supp}(\mathbf{a})$ and write $\mathbf{a}=(a_i)_{i\in I}$. If $\mathbf{a}=\mathbf{e}_j$ for some $j\in J$, then $\omega(\mathbf{a})\ge 1=\lVert \mathbf{a}\rVert_1$. So, suppose that $\sum_{j\in J} a_j\ge 2$. Observe that $\mathbf{a}+\sum_{j\in J} a_j \mathbf{b}=\sum_{j\in J}\sum_{i=1}^{a_j}(\mathbf{e}_j+\mathbf{b})$, $\sum_{j\in J} a_j \mathbf{b}\in S$, and $\mathbf{e}_j+\mathbf{b} \in S$ for all $j\in J$. Hence, $\mathbf{a}\leq_S\sum_{j\in J}\sum_{i=1}^{a_j}(\mathbf{e}_j+\mathbf{b})$. Moreover, if $k\in J$, then from $\mathbf{a}\wedge \mathbf{b}=\mathbf{0}$, we deduce that $\mathbf{a}\nleq \sum_{j\in J}\sum_{i=1}^{a_j}(\mathbf{e}_j+\mathbf{b})-(\mathbf{e}_k+\mathbf{b})$. Therefore, $\omega(\mathbf{a})\geq \sum_{j\in J}a_j=\lVert \mathbf{a}\rVert_1$.
\end{proof}

\begin{example}
    Notice that even if $\mathbf{e}_1$ is in $S$, this does not imply that $\omega(\mathbf{e}_1)=1$. Take for example $S=\{(0,0)\}\cup ((1,0)+\mathbb{N}^2)$. Then, $(1,0)\le_S (1,1)+(1,1)$ but $(1,0)\not\le_S(1,1)$. In general, $1\le \omega(\mathbf{e}_1)\le 2=1+\lVert \mathbf{e}_1\rVert_1$ (Theorem~\ref{thm:upper-bound-omega-norm1}). 
\end{example}


Let $S$ be an ideal extension of $\mathbb{N}^{(I)}$. Define the set of \emph{extreme rays} of $S$ as
\[
\mathcal{E}(S)=\{\mathbf{m}\in \mathcal{M}(S) : |\operatorname{Supp}(\mathbf{m})|=1\}.
\]

\begin{corollary}\label{cor:lower-bound-omega-extreme-rays}
    Let $S$ be an ideal extension of $\mathbb{N}^{(I)}$ with more than one extreme ray. Then 
    \[
    \omega(S)\ge \sup\{ 2\lVert \mathbf{r}\rVert_1-1 : \mathbf{r}\in \mathcal{E}(S)\}.
    \]
\end{corollary}
\begin{proof}
    The claim follows easily from Proposition~\ref{prop:lower-omega-disjoint-supp}, since if $\mathbf{r}$ and $\mathbf{r}'$ are two different extreme rays, and $\mathbf{r}=r\mathbf{e}_j$ for some $j\in I$, then $2\mathbf{r}-\mathbf{e}_j\in \mathcal{A}(S)$ and $(2\mathbf{r}-\mathbf{e}_j)\wedge \mathbf{r}'=\mathbf{0}$.
\end{proof}


\begin{example} \label{ex:omega-ubound}\rm 
Let $\emptyset\neq J\subseteq I\subseteq \mathbb{N}$ and consider the monoid $S=\{0\}\cup \{k_j\mathbf{e}_j: j\in J\}+\mathbb{N}^{(I)}$ for
given positive integers $k_j$, $j\in J$ (see Proposition~\ref{simple-case}). By Corollary~\ref{cor:lower-bound-omega-extreme-rays}, we have that $\omega(S)\geq \sup\{2k_j-1:j\in J\}$. So, for all $n\in \mathbb{N}$ there exists an extension ideal $S$ of $\mathbb{N}^{(I)}$ such that $\omega(S)\geq n$. We can also obtain an example with infinite omega primality. For instance, suppose $J$ is an infinite set and define $k_j=j+1$ for all $j\in J$. Then $\omega(S)\geq j+1$ for all $j\in J$ and since $J$ is infinite we have $\omega(S)=\infty$.
\end{example}

If $|I|=1$, then $S$ is isomorphic to an ordinary numerical semigroup, and by \cite[Proposition 3.1]{algorithm_omega} then $\omega(S)=3$ (or $1$ if $S$ is isomorphic to $\mathbb{N}$).

\begin{proposition}\label{prop:infinite-omega-primality}
Let $S$ be an ideal extension of $\mathbb{N}^k$ for some positive integer $k$. Then, $\omega(S)<\infty$ if and only if $|\mathcal{H}(S)|<\infty$. 
\end{proposition}
\begin{proof}
    Observe that $\mathcal{H}(S)$ has finitely many elements if and only if for all $i\in \{1,\dots,k\}$ there exists a positive integer $k_i$ such that $k_i\mathbf{e}_i\in S$. 

    \emph{Sufficiency}. In this setting, $S$ is a generalized numerical semigroup and thus it is finitely generated (see \cite[Proposition 2.3]{GNS}). Thus, $\omega(S)$ is finite in virtue of Theorem~\ref{thm:upper-bound-omega-norm1}.

    \emph{Necessity}. Let us suppose that $\mathcal{H}(S)$ has infinitely many elements, that is, there exists $i\in\{1,\dots,k\}$ such that $t\mathbf{e}_i\not\in S$ for every positive integer $t$. Let $\mathbf{m}=(m_1,\dots,m_k) \in \mathcal{M}(S)$ be such that its $i$th coordinate is maximal among the elements in $\mathcal{M}(S)$.
    
    We show that $\mathbf{m}+n\mathbf{e}_i\in\mathcal{A}(S)$ for all $n\in \mathbb{N}$. Let $n\in \mathbb{N}$ and assume that $\mathbf{m}+n\mathbf{e}_i=\mathbf{a}+\mathbf{b}$ with $\mathbf{a}=(a_1,\dots,a_k), \mathbf{b}=(b_1,\dots,b_k)\in S^*$. By hypothesis, there exist $j\in \{1,\dots,k\}\setminus \{i\}$ such that $m_j\neq 0$. Since $a_j+b_j=m_j$, we can assume without loss of generality that $a_j<m_j$. Moreover, we have $a_l\leq m_l$ for all $l\in \{1,\dots,k\}\setminus \{i,j\}$. Let $\mathbf{n}=(n_1,\dots,n_k)\in \mathcal{M}(S)$ such that $\mathbf{n}\leq \mathbf{a}$; in particular, $n_j<m_j$ and thus $\mathbf{n}\neq \mathbf{m}$. If $n_i\leq m_i$, we obtain $\mathbf{n}< \mathbf{m}$, which is in contradiction with $\mathbf{m}\in \mathcal{M}(S)$, and  the inequality $n_i>m_i$ cannot hold by the maximality of $m_i$. So, $\mathbf{m}+n\mathbf{e}_i\in \mathcal{A}(S)$ for all $n\in \mathbb{N}$. 
    
    Let $n\in \mathbb{N}$ with $n>2$ and denote $s_n=\sum_{i=1}^n i=\frac{n(n+1)}{2}$. Observe that $\mathbf{m}+((n-2)m_i+s_n)\mathbf{e}_i\leq_S \sum_{k=1}^n(\mathbf{m}+k\mathbf{e}_i)$, since $\sum_{k=1}^n(\mathbf{m}+k\mathbf{e}_i)-(\mathbf{m}+((n-2)m_i+s_n)\mathbf{e}_i)=n\mathbf{m}+s_n\mathbf{e}_i-\mathbf{m}-(n-2)m_i\mathbf{e}_i-s_n\mathbf{e}_i=(n-1)\mathbf{m}-(n-2)m_i\mathbf{e}_i =\mathbf{m}+(\sum_{l\in I, l\neq i}(n-2) m_l\mathbf{e}_l)\in S$. But for all $j\in \{1,\dots, n\}$ we have $\mathbf{m}+((n-2)m_i+s_n)\mathbf{e}_i\nleq_S \sum_{k=1,k\neq j}^n(\mathbf{m}+k\mathbf{e}_i)$, since $\sum_{k=1,k\neq j}^n(\mathbf{m}+k\mathbf{e}_i)-(\mathbf{m}+((n-2)m_i+s_n)\mathbf{e}_i)= (n-2)\mathbf{m}-(n-2)m_i\mathbf{e}_i-j\mathbf{e}_i$, whose $i$th coordinate is negative. So, $\omega(\mathbf{m}+((n-2)m_i+s_n)\mathbf{e}_i)\ge n$.
\end{proof}


\begin{remark}\label{rem:omega-infinite}
    Let $S$ be an ideal extension of $\mathbb{N}^{(I)}$ for some non-empty set $I$ of $\mathbb{N}$. Suppose that there is $i\in I$ such that there is no $\mathbf{s}\in S$ whose support is $\{i\}$. Let $\mathbf{a}$ be an atom of $S$, and define  $J=\operatorname{Supp}(\mathbf{a})\cup \{i\}$. Let $M=\{\mathbf{s}\in S : \operatorname{Supp}(\mathbf{s})\subseteq J\}$. Then, $M$ is a divisor-closed submnonoid of $S$ isomorphic to the submonoid of $\mathbb{N}^J$ defined as $M'=\{ (s_j)_{j\in J} : (s_i)_{i\in I}\in M\}$. Observe that $M'$ is an ideal extension of $\mathbb{N}^J$ and that $\mathcal{A}(M)\subseteq \mathcal{A}(S)$. Also, since $i\in J$, we have that $\mathbb{N}^J\setminus M'$ has infinite cardinality, and so by Proposition~\ref{prop:infinite-omega-primality}, we deduce that $\infty=\omega(M')=\omega(M)$, and this forces $\omega(S)=\infty$.

\noindent  If we assume that $I$ is an infinite set and for all $i \in I$ there exists a positive integer $k_i$ such that $k_i\mathbf{e}_i\in S$, then $\omega(S)$ may be finite or not. For an example with $\omega(S)=\infty$, consider $S=\{0\}\cup \{k_i\mathbf{e}_i: i\in I\}+\mathbb{N}^{(I)}$ with $k_i=i+1$ for all $i\in I$ (see also Example~\ref{ex:omega-ubound}). For an example with $\omega(S)<\infty$, see  Proposition~\ref{prop:omega-backslash} with $I=J$.
\end{remark}


\section{Backslash monoids}\label{sec:back-slash}

In this section, we consider a special class of ideal extensions of  $\mathbb{N}^{(I)}$ and study some of its factorization invariants. 


Let $\lambda=(\lambda_i)_{i\in I}$ be a sequence of positive integers. We define  $|\mathbf{x}|_\lambda=\sum_{i\in I}\lambda_i x_i$. Let $T\subseteq \mathbb{N}$ be a numerical semigroup, define the following set:
$$S_{\lambda}^I(T)=\{\mathbf{x}\in \mathbb{N}^{(I)}  : |\mathbf{x}|_\lambda \in T\setminus \{0\}\}\cup \{\mathbf{0}\}.$$
It is not difficult to see that $S_{\lambda}^I(T)$ is a monoid, because of the linearity of $|\cdot|_\lambda$. 

\begin{remark}
If the numerical semigroup $T$ is ordinary, that is, of the form $T= \{0\}\cup (m+\mathbb{N})$ with $m$ a positive integer, then $S_{\lambda}^I(T)$ is an ideal extension of $\mathbb{N}^{(I)}$. To prove this, take $\mathbf{s}\in S_{\lambda}^I(T)^*$ (that is, $|\mathbf{s}|_\lambda \ge m$) and $\mathbf{x}\in\mathbb{N}^{(I)}$. Then $|\mathbf{s}+\mathbf{x}|_\lambda =|\mathbf{s}|_\lambda +|\mathbf{x}|_\lambda\ge m+|\mathbf{x}|_\lambda \in T\setminus\{0\}$. 


The converse is not true in general, that is, if $T$ has not the form $\{0\}\cup (m+\mathbb{N}$), then it is possible that $S_\lambda^I(T)$ is an ideal extension of $\mathbb{N}^{(I)}$ for some $I\subseteq \mathbb{N}$ and $\lambda \in \mathbb{N}^{(I)}$. In fact, consider $T=\mathbb{N}\setminus \{1,2,3,5\}$, $\lambda=(2,3)$ and $S=S_\lambda^{\{1,2\}}=\{(x,y)\in \mathbb{N}^2 : 2x+3y\in T\}$. Take $(x,y)\in S^*$, and so $2x+3y\in T^*$. Then $(x,y)+(1,0)\in S$, since $2(x+1)+3y=(2x+3y)+2\in T$. Also, $(x,y)+(0,1)\in S$, because $2x+3(y+1)=(2x+3y)+3\in T$. Thus, $S$ is an ideal extension of $\mathbb{N}^2$.
\end{remark}

We study the monoid $S_{\lambda}^I(T)$ in the special case $\lambda_j=1$ for all $j\in J$ and $\lambda_i=0$ for every $i\in I\setminus J$, with $J$ a subset of $I$. In this special case, we denote the monoid $S_{\lambda}^I(T)$ by $S_{J}^I(T)$, and we write $|\cdot|_J$ instead of $|\cdot|_\lambda$. If $J=I$, then $|\cdot |_J$ equals $\lVert\cdot\rVert_1$. Due to their shape, we call these monoids \emph{backslash monoids}.

\begin{proposition}\label{atoms-S_J^I(T)}
Let $T$ be a numerical semigroup, $I$ a non-empty set of non-negative integers, and $\emptyset\neq J\subseteq I$. Then \[\mathcal{A}(S_J^I(T))=\{\mathbf{s}\in S_J^I(T) : |\mathbf{s}|_J\in \mathcal{A}(T)\}.\]
\end{proposition}
\begin{proof}
Let $\mathbf{s}\in S_J^I(T)$ such that $|\mathbf{s}|_J\in \mathcal{A}(T)$. If $\mathbf{s}=\mathbf{a}+\mathbf{b}$ with $\mathbf{a},\mathbf{b}\in S_J^I(T)^*$, then $|\mathbf{s}|_J=|\mathbf{a}|_J+|\mathbf{b}|_J$ and $|\mathbf{a}|_J, |\mathbf{b}|_J\in T^*$, which contradicts the fact that $|\mathbf{s}|_J$ is an atom of $T$.

For the other inclusion, let $\mathbf{s}\in \mathcal{A}(S_J^I(T))$ and suppose $|\mathbf{s}|_J\notin \mathcal{A}(T)$, that is, $|\mathbf{s}|_J=a+b$ for some $a,b \in T\setminus \{0\}$. In particular $\sum_{j\in J}s_j=a+b$ and also $\sum_{j\in J}s_j>a$. There exists $k\in J$ such that $s_k$ is not zero, and consequently $(\sum_{j\in J\setminus \{k\}}s_j)+(s_k-1)>a-1$. Repeating this process until the right hand side is zero, it is not difficult to see that it is possible to write $a=\sum_{j\in J}a_j$ for some $a_j\in \mathbb{N}$ such that $s_j\geq a_j$ for all $j\in J$ and there exists $k\in J$ such that $s_k>a_k$. Set $a_i=0$ for all $i\in I\setminus J$, and denote $\mathbf{a}=(a_i)_{i\in I}\in \mathbb{N}^{(I)}$. In particular $|\mathbf{a}|_J=a\in T^*$, and consequently $\mathbf{a}\in S_J^I(T)$. By considering the element $\mathbf{x}=\mathbf{s}-\mathbf{a}\in \mathbb{N}^{(I)}$, we have that $|\mathbf{x}|_J=|\mathbf{s}|_J-|\mathbf{a}|_J=b\in T^*$. But this implies that $\mathbf{s}-\mathbf{a}\in S_J^I(T)^*$, contradicting the fact that $\mathbf{s}\in \mathcal{A}(S_J^I(T))$.
\end{proof}

For the next result we recall that for an atomic monoid $M$ and $m\in M$, the \emph{elasticity} of $m$ is defined as 
\[\rho(m)=\frac{\sup \mathsf{L}(m)}{\min \mathsf{L}(m)}.\] 
The elasticity of $M$ is defined as 
\[\rho(M)=\sup \{\rho(m) : m \in M\}.\] 
Moreover, for an element $m$ in a BF-monoid $M$, the \emph{length density} of $m$ can be defined as in \cite{bckm}: 
\[\operatorname{LD}(m)=\frac{|\mathsf{L}(\mathbf{s})|-1}{\max\mathsf{L}(m)-\ell(m)}.\] The length density of the monoid $M$ is defined as
\[\operatorname{LD}(M)=\inf \{\operatorname{LD}(m): m\in M, |\mathsf{L}(m)|\ge 2\}\]

\begin{proposition}\label{prop:lengths-S_J^I(T)}
Let $T$ be a numerical semigroup, $I$ a non-empty set of non-negative integers, and $\emptyset\neq J\subseteq I$. Then, for all $\mathbf{s}\in S_J^I(T)$ we have $\mathsf{L}(\mathbf{s})=\mathsf{L}(|\mathbf{s}|_J)$. In particular, $\rho(S_J^I(T))=\rho(T)$ and $\operatorname{LD}(S_J^I(T))=\operatorname{LD}(T)$.
\end{proposition}
\begin{proof}
Let $\mathbf{s}\in S_J^I(T)$ and write $|\mathbf{s}|_J=t_1+\cdots+t_n$ with $t_i\in \mathcal{A}(T)$ for every $i\in \{1,\ldots,n\}$. Since $|\mathbf{s}|_J>t_n$, arguing as in the proof of Proposition~\ref{atoms-S_J^I(T)}, we can find $\mathbf{s}_n\in \mathbb{N}^{(I)}$ such that $\mathbf{s}-\mathbf{s}_n\in \mathbb{N}^{(I)}$ and $|\mathbf{s}_n|_J=t_n$. In particular, $\mathbf{s}_n\in S_J^I(T)$. Moreover, $|\mathbf{s}-\mathbf{s}_n|_J=t_1+\dots+t_{n-1}$, that is, $\mathbf{s}-\mathbf{s}_n\in S_J^I(T)$. Repeating the procedure several times, we obtain $\mathbf{s}=\mathbf{s}_1+\mathbf{s}_2+\cdots+\mathbf{s}_n$ with $|\mathbf{s}_i|_J=t_i$ for every $i\in \{1,\dots,n\}$. In particular, by Proposition~\ref{atoms-S_J^I(T)}, $\mathbf{s}_i\in \mathcal{A}(S_J^I(T))$ for every $i\in \{1,\ldots,n\}$. So $\mathsf{L}(|\mathbf{s}|_J)\subseteq \mathsf{L}(\mathbf{s})$. For the converse, if $\mathbf{s}=\mathbf{s}_1+\dots+\mathbf{s}_n$ with $\mathbf{s}_i\in \mathcal{A}(S_J^I(T))$, we have $|\mathbf{s}|_J=|\mathbf{s}_1|_J+\dots+|\mathbf{s}_n|_J$, and by Proposition~\ref{atoms-S_J^I(T)} we have $|\mathbf{s}_i|_J\in \mathcal{A}(T)$ for every $i\in \{1,\ldots,n\}$, so the first claim follows. 

The second claim is a direct consequence of the first claim and the definitions of elasticity and length density.
\end{proof}


\begin{remark}
Let $S$ be an ideal extension of $\mathbb{N}^d$ with $d$ a positive integer. If for all $i\in \{1,\ldots,d\}$ there exists $m_i=\min\{ \alpha \in \mathbb{N} : \alpha \mathbf{e}_i \in S\}$, then $S$ is a \emph{complement-finite ideal} \cite{Baeth}. For this kind of monoids, \cite[Lemma 4.11]{Baeth} provides an upper bound for $\rho(S)$, and it is proved that for the class of monoids introduced in Proposition~\ref{simple-case} (in the case $I=J=\{1,\ldots,d\}$) the elasticity attains the upper bound (see \cite[Proposition 4.15]{Baeth}). Observe that for all $i\in \{1,\ldots,d\}$ the set $S_i=S\cap \{x\mathbf{e}_i : x\in \mathbb{N}\}$ is isomorphic to the numerical semigroup $\mathbb{N}\setminus \{1,\ldots,m_i-1\}$, so $\sup\{\rho(\mathbf{s}) : \mathbf{s}\in S_i\}=\frac{2m_i-1}{m_i}$ (see \cite[Theorem 12]{Num-sem-book}). Therefore, if we define $m=\max \{m_1,\dots,m_d\}$, we have the lower bound $\rho(S)\geq \frac{2m-1}{m}$ for the elasticity of a complement-finite ideal $S$. In the case $T$ is the ordinary semigroup $\{0\}\cup (n+\mathbb{N})$, the monoid $S_J^I(T)\subseteq \mathbb{N}^d$ is a complement-finite ideal when $J=\{1,\ldots,d\}$, and $m_i=n$ for all $i$. By Proposition~\ref{prop:lengths-S_J^I(T)}, the elasticity of $S_J^I(T)$ attains the lower bound $(2n-1)/n$. Thus ideal extensions with finitely many gaps of the form $\{\mathbf{0}\}+(\{m_1\mathbf{e}_1,\dots,m_e\mathbf{e}_d\}+\mathbb{N})^d)$ have maximal elasticity, while those of the form $S_I^I(T)$ with $T$ an ordinary numerical semigroup and $I$ a finite set have minimal elasticity.
\end{remark}

We now want to see under what conditions $S_J^I(T)$ is gap absorbing. To this end, we first give a description of $2\mathcal{A}(S_J^I(T))$, with $T$ an ordinary numerical semigroup.

\begin{lemma}\label{lem:2A-S_J^I(T)}
Let $m$ be a positive integer and $T=\{0\}\cup(m+\mathbb{N})$. Then, $2\mathcal{A}(S_J^I(T))=\{\mathbf{x}\in \mathbb{N}^{(I)} : 2m\leq |\mathbf{x}|_J\leq 4m-2\}$.
\end{lemma}
\begin{proof}
As $\mathcal{A}(T)=\{m,m+1,\dots,2m-1\}$, by Proposition~\ref{atoms-S_J^I(T)}, $2m\le |\mathbf{a}+\mathbf{b}|_J \le 4m-2$, for $\mathbf{a},\mathbf{b}\in \mathcal{A}(S_J^I(T))$.

For the other inclusion, let $\mathbf{x}\in \mathbb{N}^{(I)}$ such that $|\mathbf{x}|_J=t\in [2m,4m-2]$. It is possible to write $t=t_1+t_2$ with $t_1,t_2\in \{m,\ldots,2m-1\}$. Since $t>t_1$, arguing as in the proof of Proposition~\ref{atoms-S_J^I(T)}, we can obtain $\mathbf{x}_1\in \mathbb{N}^{(I)}$ such that $|\mathbf{x}_1|_J=t_1$ and $\mathbf{x}-\mathbf{x}_1\in \mathbb{N}^{(I)}$. Let $\mathbf{x}_2=\mathbf{x}-\mathbf{x}_1$. Then $|\mathbf{x}_2|_J=t_2$. Hence, by Proposition~\ref{atoms-S_J^I(T)}, $\mathbf{x}_1,\mathbf{x}_2\in \mathcal{A}(S_J^I(T))$, that is, $\mathbf{x}\in 2\mathcal{A}(S_J^I(T))$.
\end{proof}

\begin{proposition}\label{prop:backslash-gap-absorbing-and-minimals}
Let $T$ be a numerical semigroup. Then, $S_J^I(T)$ is a gap absorbing semigroup if and only if $T$ is an ordinary numerical semigroup. In this case, \[\mathcal{M}(S_J^I(T))=\{\mathbf{m}=(m_i)_{i\in I}\in \mathbb{N}^{(I)} : |\mathbf{m}|_J=\min(T^*) \text{ and } m_i=0 \text{ for all } i\in I\setminus J\}.\] 
\end{proposition}
\begin{proof}
Denote $S=S_J^I(T)$ and $m=\min(T^*)$ (the multiplicity of $T$). 

\emph{Sufficiency}. Let $\mathbf{x},\mathbf{y}\in \mathcal{H}(S)$ and $\mathbf{a}\in \mathcal{A}(S)$, so $|\mathbf{x}|_J,|\mathbf{y}|_J\in \{1,\ldots,m-1\}$ and $|\mathbf{a}|_J\in \{m,\ldots,2m-1\}$ by Proposition~\ref{atoms-S_J^I(T)}. Then $|\mathbf{x}+\mathbf{y}|_J\in \{2,\ldots,2m-2\}$ and $|\mathbf{x}+\mathbf{a}|_J\in \{m+1,\ldots,3m-2\}$, that is, $\mathbf{x}+\mathbf{y}\in \mathcal{H}(S)\cup \mathcal{A}(S)\cup 2\mathcal{A}(S)$ and $\mathbf{x}+\mathbf{a}\in \mathcal{A}(S)\cup 2\mathcal{A}(S)$ by Lemma~\ref{lem:2A-S_J^I(T)}. Therefore, $S$ is gap absorbing. 

\emph{Necessity}. Observe that if $S$ is gap absorbing, then $S$ is an ideal extension of $\mathbb{N}^{(I)}$ (Proposition~\ref{prop:gap-absorbing-implies-ideal}). Suppose there exists $t\in T$ such that $t+1\notin T$. Let $\mathbf{x}\in \mathbb{N}^{(I)}$ such that $|\mathbf{x}|_J=t$ and consider $\mathbf{x}_k=\mathbf{x}+\mathbf{e}_k$ for some $k\in J$. Then $|\mathbf{x}_k|_J=t+1$, that is, $\mathbf{x}_k\notin S$, which contradicts that $S$ is an ideal extension of $\mathbb{N}^{(I)}$.

Finally, let us prove that $\mathcal{M}(S)=\{\mathbf{m}=(m_i)_{i\in I}\in \mathbb{N}^{(I)} : |\mathbf{m}|_J=m \text{ and } m_i=0 \text{ for all } i\in I\setminus J\}$.
Let $\mathbf{m}=(m_i)_{i\in I}\in \mathcal{M}(S)$. If $m_i\neq 0$ for some $i\in I\setminus J$, then $|\mathbf{m}-m_i\mathbf{e}_i|_J=|\mathbf{m}|_J\in T$, and so $\mathbf{m}-m_i\mathbf{e}_i$ is in $S_J^I(T)$ and $\mathbf{m}-m_i\mathbf{e}_i<\mathbf{m}$, a contradiction. If $|\mathbf{m}|_J>m$, take $j\in J$ such that $m_j\neq 0$. Then, $|\mathbf{m}-\mathbf{e}_j|_J=|\mathbf{m}|_j-1\ge m$, and consequently $\mathbf{m}-\mathbf{e}_j\in S$, contradicting the minimality of $\mathbf{m}$. For the other inclusion, let $\mathbf{m}=(m_i)_{i\in I}\in \mathbb{N}^{(I)}$ be such that $m_i=0$ and $|\mathbf{m}|_J=m$. By definition, $\mathbf{m}\in S$. If there exists $\mathbf{n}=(n_i)_{i\in I}\in S^*$ with $\mathbf{n}<\mathbf{m}$, then $n_i=0$ for all $i\in I\setminus J$ and $m\le |\mathbf{n}|_J<|\mathbf{m}|_J=m$, a contradiction. Thus, $\mathbf{m}\in \mathcal{M}(S)$.
\end{proof}

Recall that on a gap absorbing monoid the Betti elements can be expressed as the sum of at most three atoms (Theorem~\ref{thm:lge4-not-betti}). We prove next that for backslash monoids we can get a sharper description.

\begin{proposition} \label{prop:betti-S_J^I(T)}
Let $T$ be an ordinary semigroup. Then, $\operatorname{Betti}(S_J^I(T))\subseteq 2\mathcal{A}(S_J^I(T))$.
\end{proposition}
\begin{proof}
Denote $S=S_J^I(T)$. We know that $T=\{0\}\cup(m+\mathbb{N})$ for some positive integer $m$. By Proposition~\ref{prop:betti-core}, we can assume $S=\mathsf{C}(S)$, so $m\geq 2$. 

In order to prove our result, we are going to make use of Proposition~\ref{prop:l2-denumerant2} and Corollary~\ref{cor:antichain-of-supatoms}. 
For $\mathbf{s}\in S$, 
we show that for every $\mathbf{m},\mathbf{n}\in \mathcal{M}(S)\cap \operatorname{B}(\mathbf{s})$, there exists a sequence $\mathbf{u}_1,\ldots,\mathbf{u}_n\in \mathcal{M}(S)\cap \operatorname{B}(\mathbf{s})$ such that $\mathbf{u}_1=\mathbf{m}$, $\mathbf{u}_n=\mathbf{n}$ and $\mathbf{u}_{i} \vee \mathbf{u}_{i+1} \in \mathcal{A}(S)$ for all $i\in\{1,\ldots,n-1\}$. 

Since $\mathbf{m}$ and $\mathbf{n}$ are not comparable via $\le$, we can express them as $\mathbf{m}=\sum_{k=1}^r m_{a_k}\mathbf{e}_{a_k}+\sum_{k=1}^s m_{b_k}\mathbf{e}_{b_k}+\sum_{k=1}^t m_{c_k}\mathbf{e}_{c_k}$ and $\mathbf{n}=\sum_{k=1}^r n_{a_k}\mathbf{e}_{a_k}+\sum_{k=1}^s n_{b_k}\mathbf{e}_{b_k}+\sum_{k=1}^t n_{c_k}\mathbf{e}_{c_k}$,
such that $m_{a_k}>n_{a_k}$ for all $k\in \{1,\ldots,r\}$, $m_{b_k}<n_{b_k}$ for all $k\in \{1,\ldots,s\}$ and $m_{c_k}=n_{c_k}$ for all $k\in \{1,\ldots,t\}$; notice that $m_i=n_i=0$ for all $i\in I\setminus J$. Since $|\mathbf{m}|_J=|\mathbf{n}|_J=m$, it is easy to check that $\sum_{k=1}^r m_{a_k}-\sum_{k=1}^r n_{a_k}=\sum_{k=1}^s n_{b_k}-\sum_{k=1}^s m_{b_k}$. At this point, we can define $\mathbf{u}_2= (m_{a_{1}}-1)\mathbf{e}_{a_1}+ \sum_{k=2}^r m_{a_k}\mathbf{e}_{a_k}+(m_{b_{1}}+1)\mathbf{e}_{b_1}+\sum_{k=2}^s m_{b_k}\mathbf{e}_{b_k}+\sum_{k=1}^t m_{c_k}\mathbf{e}_{c_k}$. By looking at the coordinates of $\mathbf{u}_2$, as $\mathbf{m},\mathbf{n}\in \operatorname{B}(\mathbf{s})$, we deduce that $\mathbf{u}_2\in \operatorname{B}(\mathbf{s})$ as well. Moreover $|\mathbf{u}_2|_J=m$ and for every $i\in I\setminus J$, the $i$th coordinate of $\mathbf{u}_2$ is zero, so in light of Proposition~\ref{prop:backslash-gap-absorbing-and-minimals}, $\mathbf{u}_2\in \mathcal{M}(S)$. Furthermore, we have that $|\mathbf{m} \vee \mathbf{u}_{2}|_J=m+1$, which by Proposition~\ref{atoms-S_J^I(T)} implies that $\mathbf{m} \vee \mathbf{u}_{2}\in \mathcal{A}(S)$. Therefore we have obtained the second element of the sequence. By repeating the process now with $\mathbf{u}_2$ and $\mathbf{n}$ and, after a finite number of steps, we obtain $\mathbf{u}_n=\mathbf{n}$.
\end{proof}


With the help of this result, we can explicitly compute the catenary degree of a backslash monoid associated to an ordinary numerical semigroup.

\begin{proposition}\label{catenary_bans}
    Let $T$ be an ordinary numerical semigroup, with $T\neq \mathbb{N}$, $I$ a non-empty set of non-negative integers, and $\emptyset\neq J\subseteq I$. Then $\mathsf{c}(S_J^I(T))=3$.
\end{proposition}
\begin{proof}
    By Proposition~\ref{prop:betti-S_J^I(T)} and Corollary~\ref{cor:catenary-three}, we know that $\mathsf{c}(S_J^I(T))\leq 3$. Moreover, if $j\in J$, then $\mathsf{L}((4m-2)\mathbf{e}_j)\in\{2,3\}$ and each factorization of length two has no common atoms with any other factorization of length three, obtaining in this way $\mathsf{c}((4m-2)\mathbf{e}_j)=3$ and $\mathsf{c}(S_J^I(T))= 3$.
\end{proof}

If $T$ is an ordinary numerical semigroup, we know by \cite[Proposition 3.1]{algorithm_omega} that $\omega(T)=3$. Next, we explicitly compute $\omega(S_J^I(T))$ in the case $|I|>1$.

\begin{proposition}\label{prop:omega-backslash}
    Let $T$ be an ordinary numerical semigroup, with $T\neq \mathbb{N}$, $I$ a non-empty set of non-negative integers with $|I|>1$, and $J$ a nonempty subset of $I$. Then, 
    \begin{itemize}
        \item $\omega(S_I^I(T))=2\min(T^*)-1$, 
        and 
        \item     $\omega(S_J^I(T))=\infty$ for any proper subset $J$ of $I$.
    \end{itemize}
\end{proposition}

\begin{proof}
    Set $m=\min(T^*)$ and denote $S=S_J^I(T)$. If $J\neq I$, take $i\in I\setminus J$ and $j\in J$. We argue as in Remark~\ref{rem:omega-infinite}. Define $M=\{ \mathbf{s} \in S : \operatorname{Supp}(\mathbf{s})\subseteq\{i,j\}\}$. Then, $M$ is a divisor-closed monoid that is isomorphic to $\{(0,0)\}+((0,m)+\mathbb{N}^2)$. By Proposition~\ref{prop:infinite-omega-primality}, we know that $\omega(M)=\infty$. As $\mathcal{A}(M)\subseteq \mathcal{A}(S)$, we derive that $\omega(S)=\infty$. 

    Now, suppose that $I=J$. Arguing as in Proposition~\ref{prop:backslash-gap-absorbing-and-minimals}, it is not difficult to show that $\mathbf{a}$ is a maximal atom with respect to the usual partial order on $\mathbb{N}^{(I)}$ if and only if $|\mathbf{a}|_I = \lVert \mathbf{a}\rVert_1 = 2m-1$. We prove that $\omega(\mathbf{a})\le \lVert \mathbf{a}\rVert_1 = 2m-1$, which in virtue of Proposition~\ref{prop:omega-non-decreasing}, will provide us the upper bound for $\omega(S)$. 
    
    Let $\mathbf{a}\leq_S \mathbf{a}_1+\cdots+\mathbf{a}_n$ with $\mathbf{a}_1,\ldots,\mathbf{a}_n\in \mathcal{A}(S)$. If $n\le \lVert \mathbf{a}\rVert_1$, we are done. So, let us suppose that $n>\lVert \mathbf{a}\rVert_1$. By Lemma~\ref{lem:number-le-sum}, there exists $L\subseteq \{1,\dots,n\}$ such that $\mathbf{a}+\mathbf{x} = \sum_{l \in L}\mathbf{a}_l$ for some $\mathbf{x}\in \mathbb{N}^{(I)}$ and $|L|\le \lVert \mathbf{a} \rVert_1$. By adding atoms in both sides, we can assume that $|L|=\lVert \mathbf{a} \rVert_1$. Since $m\le |\mathbf{a}_l|_I\le 2m-1$ for all $l\in L$ (Proposition~\ref{atoms-S_J^I(T)}), we deduce that  $|\mathbf{x}|_I\geq |L|m-(2m-1)$. Notice that $|L|=\lVert \mathbf{a}\rVert_1= |\mathbf{a}|_J=2m-1$. Hence, $|\mathbf{x}|_J\ge (2m-1)m-(2m-1)\ge m$, and thus $\mathbf{x}\in S_J^I(T)$. Therefore $\mathbf{a}\leq_S \sum_{l \in L}\mathbf{a}_l$, and in particular $\omega(\mathbf{a})\leq |L|\leq \lVert \mathbf{a}\rVert_1=2m-1$. Observe that $m\mathbf{e}_1\in \mathcal{E}(S)$, and so the other inequality follows by Corollary~\ref{cor:lower-bound-omega-extreme-rays}.  
\end{proof}

\begin{remark} 
    Observe that also Proposition~\ref{prop:omega-backslash} highlights that \cite[Proposition~4.9]{Baeth} is wrong, since according to that result, the $\omega$-primality of $S_I^I(T)$ with $T=\{0\}+(m+\mathbb{N})$, should belong to $\{m,m+1\}$. Moreover, from the proof of Proposition~\ref{prop:omega-backslash}, we obtain that $\omega(S_I(T))=\sup_{\mathbf{a}\in \mathcal{A}(S)}\lVert \mathbf{a}\rVert_1$, which is near to upper bound provided by Theorem~\ref{thm:upper-bound-omega-norm1}. We wonder if there exists an ideal extension $S$ of $\mathbb{N}^{(I)}$ such that $\omega(S)= 1+\sup_{\mathbf{a}\in \mathcal{A}(S)}\lVert \mathbf{a}\rVert_1$.     
\end{remark}

\section{\texorpdfstring{Gap absorbing submonoids of $\mathbb{N}^2$}{Gap absorbing submonoids on dimension two}}\label{sec:N2}

We study gap absorbing submonoids of $\mathbb{N}^2$, and prove that every ideal extension of $\mathbb{N}^2$ is gap absorbing. Then we particularize the results of the preceding sections to dimension two.

Our first goal is to prove that for any ideal extension $S$ of $\mathbb{N}^2$, the set $2\mathcal{A}(S)$ is closed by intervals, and by Proposition~\ref{prop:2A-interval-eq}, $S$ is gap absorbing. To this end, we start with a technical lemma that will help us describe $2\mathcal{A}(S)$.

In our particular setting, the set $A$ defined in Proposition~\ref{prop:gaps-atoms-pi} is \[A=\{v\in \mathbb{N}: \pi_{v}^1<\infty, \pi^2_{v}\neq 0\},\]
and for every $z\in A+A$, 
\[
z_A=\{v \in A : z-v\in A\}.
\]
In this case, denote
\begin{itemize}
\item $m_{{z}}=\min \{\pi^1_{{v}}+\pi^1_{{z-v}} : {v}\in z_A\}$,
\item $M_{{z}}=\max \{\pi^2_{{v}}+\pi^2_{{z-v}}-2 : {v}\in {z}_A\}$.
\end{itemize}


\begin{lemma}
Let ${z}\in A+A$. Then
\[
\bigcup_{v\in z_A} [\pi^1_{{v}}+\pi^1_{{z-v}},\pi^2_{{v}}+\pi^2_{{z-v}}-2]=[m_{z}, M_{z}].
\]
\end{lemma}
\begin{proof}

Let ${v}_1,{v}_2 \in {z}_A$, with ${v}_1\leq {v}_2$. By Proposition~\ref{prop:gaps-atoms-pi}~(2), we know that ${z}_A$ is closed under intervals and since in this case ${z}_A\subseteq \mathbb{N}$, we deduce that $z_A$ is itself an interval. So, we can consider the sequence ${u}_1,{u}_2,\ldots,{u}_{n-1},{u}_n \in {z}_A$ such that ${u}_1={v}_1$, ${u}_n={v}_2$ and ${u}_{i+1}={u}_i+1$ for $i\in \{1,\ldots,n-1\}$. We obtain our result by proving that for all $i\in \{1,\ldots,n\}$ the union $[\pi^1_{{u}_i}+\pi^1_{{z}-{u}_i},\pi^2_{{u}_i}+\pi^2_{{z}-{u}_i}-2]\cup [\pi^1_{{u}_{i+1}}+\pi^1_{{z}-{u}_{i+1}},\pi^2_{{u}_{i+1}}+\pi^2_{{z}-{u}_{i+1}}-2]$ is an interval.
For $i\in \{1,\ldots,n\}$, let ${w}_i={z}-{u}_i$ and ${w}_{i+1}={z}-{u}_{i+1}={w}_i-1$. As ${u}_i\leq {u}_{i+1}$, by Lemma~\ref{lem:pi-decreasing}, we have $\pi^1_{{u}_{i+1}}\leq \pi^1_{{u}_i}$ and, since ${w}_i={w}_{i+1}+1\in A$, by Proposition~\ref{prop:gaps-atoms-pi}~(3), we have $\pi^1_{{w}_{i+1}}\leq \pi^2_{{w}_i}$. Also, Lemma~\ref{lem:pi1-less-pi2} ensures that $\pi_{u_i}^1\le \pi_{u_i}^2-1$. Putting all these inequalities together, we get $\pi^1_{{u}_{i+1}}+\pi^1_{{w}_{i+1}}\leq \pi^1_{{u}_i}+\pi^1_{{w}_{i+1}}\leq \pi^2_{{u}_i}-1+\pi^2_{{w}_i}$.
Furthermore, by ${w}_{i+1}\leq {w}_i$ we have $\pi^2_{{w}_{i+1}}\geq \pi^2_{{w}_i}$ (Lemma~\ref{lem:pi-decreasing}), and since ${u}_{i+1}={u}_i+1 \in A$, we have $\pi^2_{{u}_{i+1}}\geq \pi^1_{{u}_i}$ by Proposition~\ref{prop:gaps-atoms-pi}~(3). Therefore, using once more Lemma~\ref{lem:pi1-less-pi2}, we obtain $\pi^2_{{u}_{i+1}}+\pi^2_{{w}_{i+1}}-2\geq \pi^2_{{u}_{i+1}}+\pi^2_{{w}_i}-2\geq \pi^2_{{u}_{i+1}}+\pi^1_{{w}_i}-1\geq \pi^1_{{u}_i}+\pi^1_{{w}_i}-1$. 
This, in particular, proves that $[\pi^1_{{u}_i}+\pi^1_{{z}-{u}_i},\pi^2_{{u}_i}+\pi^2_{{z}-{u}_i}-2]\cup [\pi^1_{{u}_{i+1}}+\pi^1_{{z}-{u}_{i+1}},\pi^2_{{u}_{i+1}}+\pi^2_{{z}-{u}_{i+1}}-2]$ is an interval.
\end{proof}

According to this last result and Proposition~\ref{prop:gaps-atoms-pi}~(5),
\begin{equation}\label{eq:exp-AS+AS}
	2\mathcal{A}(S)=\{({z},x)\in \mathbb{N}^2 : {z}\in A+A, x\in [m_{z}, M_{z}]\}.
\end{equation}

\begin{lemma} Let ${z}_1,{z}_2 \in A+A$ such that ${z}_1\leq {z}_2$. Then, $m_{{z}_1}\geq m_{{z}_2}$ and $M_{{z}_1}\geq M_{{z}_2}$.
\label{m1<m2}
\end{lemma}
\begin{proof}
We first prove that $m_{{z}_1}\geq m_{{z}_2}$. Let ${z}_1={a}_1+{a}_2$ be such that $m_{{z}_1}=\pi^1_{{a}_1}+\pi^1_{{a}_2}$, with $a_1,a_2\in A$. We can assume ${a}_1\leq {a}_2$. Since ${a}_1\leq {z}_1\leq {z}_2$, there exists ${y}\in \mathbb{N}$ such that ${z}_2={a}_1+{a}_2+{y}$. Set ${x}={a}_2+{y}$, in particular $a_2\leq x$. We distinguish two cases depending on $x\in A$ or $x\not\in A$.
\begin{itemize}
    \item If ${x}\in A $, then by the definition of $m_{z_2}$ and by Lemma~\ref{lem:pi-decreasing} we have $m_{{z}_2}\leq \pi^1_{{a}_1}+\pi^1_{{x}}\leq \pi^1_{{a}_1}+\pi^1_{{a}_2}=m_{{z}_1}$, obtaining the desired result.
    \item Now, suppose ${x}\notin A$. Since $a_2\in A$, $A$ is an interval (Proposition~\ref{prop:gaps-atoms-pi}~(1)) and $a_2\le x\not\in A$, we deduce that there exists ${m}=\max A$ and then $a_2\le m<x$. Set $y_1=m-a_2$ and $y_2=y-y_1$. Hence, ${z}_2=a_1+a_2+y={a}_1+{y}_2+{m}$. Set ${w}={a}_1+{y}_2$. Then ${w}\in A$, since, otherwise, ${w}>{m}$ and ${z}_2>2{m}$, which is in contradiction with ${z}_2\in A+A$. So $z_2=w+m$, with $w,m\in A$. Therefore by the definition of $m_{z_2}$ and by Lemma~\ref{lem:pi-decreasing} we deduce that $m_{{z}_2}\leq \pi^1_{{w}}+\pi^1_{{m}}\leq \pi^1_{{a}_1}+\pi^1_{{a}_2}=m_{{z}_1}$, as desired.
 
\end{itemize} 
Next, we prove that $M_{{z}_1}\geq M_{{z}_2}$. Let ${z}_2={b}_1+{b}_2$ be such that $M_{{z}_2}=\pi^2_{{b}_1}+\pi^2_{{b}_2}-2$ and $b_1,b_2\in A$. We first show that we can find ${c}_1,{c}_2\in A$ with ${c}_1\leq {b}_1$ and ${c}_2\leq {b}_2$ such that ${z}_1={c}_1+{c}_2$. Since ${z}_1\leq {b}_1+{b}_2$, there exists ${x}\in\mathbb{N}$ such that ${z}_1={b}_1+{b}_2-{x}$. If $x=0$ there is nothing to show, so suppose $x\neq 0$.
Let ${n}=\min A$, so by $z_1\in A+A$ we have $2n\le z_1=b_1+b_2-x< b_1+b_2$. Moreover, since $b_1,b_2\in A$, we have $b_1\ge n$ and $b_2\ge n$. As consequence we have that ${b}_1>{n}$ or ${b}_2>{n}$. Without loss of generality we can assume ${b}_1>{n}$. In particular ${b}_1-1\in A$ and ${z}_1=({b}_1-1)+{b}_2-({x}-1)$, with ${b}_1-1,{b}_2\in A$. If ${x}-1=0$ we conclude. Otherwise, since $({b}_1-1)+{b}_2\in A+A$, $b_1-1,b_2\in A$ and $2n\leq {z}_1< ({b}_1-1)+{b}_2 $, the only possibility is  ${b}_1-1>{n}$ or ${b}_2>{n}$, and so we can continue with the same argument. Therefore, we can eventually find ${c}_1,{c}_2\in A$ with ${c}_1\leq {b}_1$ and ${c}_2\leq {b}_2$ such that ${z}_1={c}_1+{c}_2$.
Hence, by the definition of $M_{z_1}$ and by Lemma~\ref{lem:pi-decreasing}, $M_{{z}_2}=\pi^2_{{b}_1}+\pi^2_{{b}_2}-2\leq \pi^2_{{c}_1}+\pi^2_{{c}_2}-2\leq M_{{z}_1}$, which concludes the proof.
\end{proof}

We can use the above results to show that $2\mathcal{A}(S)$ is closed under intervals.

\begin{proposition}\label{prop:2A-interval}
Let $S$ be an ideal extension of $\mathbb{N}^2$. 
Then, for every $\mathbf{a},\mathbf{b}\in 2\mathcal{A}(S)$ with $\mathbf{a}\leq \mathbf{b}$, we have that $ \llbracket \mathbf{a},\mathbf{b}\rrbracket\subseteq 2\mathcal{A}(S)$.
\end{proposition}
\begin{proof} 
    First observe that, since $A$ is an interval (Proposition~\ref{prop:gaps-atoms-pi}~(1)), the set $A+A$ is also an interval. In particular, if $v, w\in A+A$ and $z$ is a positive integer such that $v\leq z\leq w$, then $z\in A+A$.

	In light of \eqref{eq:exp-AS+AS}, we can Write $\mathbf{a}=(v_a,x_a)$, with $v_a\in A+A$ and $x_a\in [m_{v_a},M_{v_a}]$. Similarly, $\mathbf{b}=(v_b,x_b)$ with $v_b\in A+A$ and $x_b\in [m_{v_b},M_{v_b}]$. Let $\mathbf{c}\in  \llbracket \mathbf{a},\mathbf{b} \rrbracket$, we can denote $\mathbf{c}=(v_c,x_c)$ with $v_c, x_c\in \mathbb{N}$. We have $v_a\leq v_c\leq v_b$ and $x_a\leq x_c \leq x_b$. Since $v_a,v_b\in A+A$, we deduce that $v_c\in A+A$. Moreover, by Lemma~\ref{m1<m2}, we obtain that $m_{v_c}\leq m_{v_a}\leq x_a \leq x_c \leq x_b \leq M_{v_b}\leq M_{v_c}$, that is, $x_c \in [m_{v_c}, M_{v_c}]$, and so $\mathbf{c} \in 2\mathcal{A}(S)$ by \eqref{eq:exp-AS+AS}.
\end{proof}

%

\begin{corollary}
    Let $S$ be a submonoid of $\mathbb{N}^2$. Then, $S$ is gap absorbing if and only if $S$ is an ideal extension of $\mathbb{N}^2$.
\end{corollary}
\begin{proof}
    We already know that if $S$ is gap absorbing, then $S=\{\mathbf{0}\}\cup(\mathcal{M}(S)+\mathbb{N}^2)$ (Proposition~\ref{prop:gap-absorbing-implies-ideal}). Sufficiency follows from Propositions~\ref{prop:2A-interval} and \ref{prop:2A-interval-eq}.
\end{proof}

In the following, $\preceq$ denotes the lexicographical order on $\mathbb{N}^2$. For $\mathbf{a},\mathbf{b}\in X\subseteq \mathbb{N}^2$, we write $\mathbf{a}\prec \mathbf{b}$ if $\mathbf{a}\preceq \mathbf{b}$ and $\mathbf{a}\neq \mathbf{b}$. We say $\mathbf{a}$ \emph{covers} $\mathbf{b}$ in $X$ (with respect to $\preceq$) if $\mathbf{a}\prec \mathbf{b}$ and there is no $\mathbf{c}\in X$ such that $\mathbf{a}\prec \mathbf{c}\prec \mathbf{b}$. 

\begin{lemma}\label{lem:sup-is-atom}
 Let $S\subseteq \mathbb{N}^2$ be a gap absorbing monoid. Let $\mathbf{m},\mathbf{n}\in \mathcal{M}(S)$ be such that $\mathbf{n}$ covers $\mathbf{m}$ in $\mathcal{M}(S)$ with respect to the lexicographical ordering. If $\operatorname{Supp}(\mathbf{m})\cap\operatorname{Supp}(\mathbf{n})\neq \emptyset$, then $\mathbf{m}\vee \mathbf{n}$ is an atom.
\end{lemma}
\begin{proof}
     Write $\mathbf{m}=(m_1,m_2)$ and $\mathbf{n}=(n_1,n_2)$. As $\mathbf{m}\preceq \mathbf{n}$, we have that $m_1\neq n_1$, since otherwise $\mathbf{m}\le \mathbf{n}$ and this is impossible as both $\mathbf{m}$ and $\mathbf{n}$ are in $\mathcal{M}(S)$. Thus, $m_1<n_1$.

    Notice that $\mathbf{m}\vee \mathbf{n}\le \mathbf{m}+\mathbf{n}\in \mathcal{A}(S)+ \mathcal{A}(S)$. If $\mathbf{m}\vee\mathbf{n}$ is not an atom, then $\mathbf{m}\vee\mathbf{n} = \mathbf{c}+\mathbf{d}$ for some $\mathbf{c},\mathbf{d}\in S^*$. Take $\mathbf{a},\mathbf{b}\in \mathcal{M}(S)$ such that $\mathbf{a}\le \mathbf{c}$ and $\mathbf{b}\le\mathbf{d}$. We may suppose without loss of generality that $\mathbf{a}\preceq \mathbf{b}$. Write $\mathbf{a}=(a_1,a_2)$. Observe that $\mathbf{a}+\mathbf{b}\le \mathbf{m}\vee\mathbf{n}\le \mathbf{m}+\mathbf{n}$, and in particular $\mathbf{a}+\mathbf{b}\preceq \mathbf{m}+\mathbf{n}$. If $\mathbf{n}\preceq \mathbf{a}$, then $\mathbf{n}+\mathbf{b}\preceq \mathbf{a}+\mathbf{b}$, and as $\mathbf{m}\prec \mathbf{b}$, we obtain $\mathbf{m}+\mathbf{n}\prec \mathbf{n}+\mathbf{b}$, which by transitivity leads to $\mathbf{m}+\mathbf{n}\prec \mathbf{a}+\mathbf{b}$, which contradicts $\mathbf{a}+\mathbf{b}\preceq \mathbf{m}+\mathbf{n}$. Thus $\mathbf{a}\prec \mathbf{n}$, and as $\mathbf{n}$ covers $\mathbf{m}$, we have that $\mathbf{a}\preceq \mathbf{m}$; in particular, $a_1\le m_1$. 

    If $\mathbf{a}\neq \mathbf{m}$, then $\mathbf{a}\in \operatorname{B}(\mathbf{m}\vee\mathbf{n})\setminus (\operatorname{B}(\mathbf{m})\cup \operatorname{B}(\mathbf{n})) = \llbracket \mathbf{m}\wedge \mathbf{n},\mathbf{m}\vee\mathbf{n}\rrbracket$. 
    Hence, $a_1\le m_1\le a_1\le n_1$, which forces $a_1=m_1$. From $\mathbf{a}\preceq \mathbf{m}$ we deduce that $a_2\le m_2$, and then $\mathbf{a}< \mathbf{m}$, which is impossible since $\mathbf{a}, \mathbf{m}\in \mathcal{M}(S)$.
    
    Thus, $\mathbf{m}=\mathbf{a}$. Then, either $\mathbf{a}=\mathbf{m}\prec \mathbf{n}\preceq \mathbf{b}$ or $\mathbf{a}=\mathbf{b}= \mathbf{m}\prec \mathbf{n}$. If $\mathbf{a}=\mathbf{b}= \mathbf{m}\prec \mathbf{n}$, then $2\mathbf{m}=\mathbf{a}+\mathbf{b}\le \mathbf{m}\vee\mathbf{n}\le \mathbf{m}+\mathbf{n}$, and consequently $2\mathbf{m}\le \mathbf{m}+\mathbf{n}$, which forces $\mathbf{m}\le \mathbf{n}$, a contradiction. Thus, $\mathbf{a}=\mathbf{m}\prec \mathbf{n}\preceq \mathbf{b}$. Hence, $\mathbf{a}+\mathbf{b}\le \mathbf{m}\vee\mathbf{n}\le \mathbf{m}+\mathbf{n}\preceq \mathbf{m}+\mathbf{b}=\mathbf{a}+\mathbf{b}$, which leads to $\mathbf{m}+\mathbf{n}=\mathbf{m}\vee\mathbf{n}$, that is, $\mathbf{m}$ and $\mathbf{n}$ have disjoint support. 
\end{proof}

\begin{proposition}
    Let $S\subseteq \mathbb{N}^2$ be a gap absorbing monoid. Then $\mathbf{s}\in \mathrm{Betti}(S)$ if and only if $\mathbf{s}\in 2\mathcal{A}(S)$ and $|\mathsf{Z}(\mathbf{s})|\ge 2$.
\end{proposition}
\begin{proof}
\emph{Sufficiency}. By Proposition~\ref{prop:l2-denumerant2} we know that if $\mathbf{s}\in 2\mathcal{A}(S)$ and has more than one factorization, then $\mathbf{s}\in \mathrm{Betti}(S)$.

\emph{Necessity}. In light of Theorem~\ref{th:bound-catenary}, we know that $\ell(\mathbf{s})\le 3$. Suppose $\ell(\mathbf{s})=3$. By Coro\-llary~\ref{cor:antichain-of-supatoms} it suffices to prove that for every $\mathbf{m},\mathbf{n}\in \mathcal{M}(S)\cap \operatorname{B}(\mathbf{s})$, there exists $\mathbf{m}_1,\ldots,\mathbf{m}_{h} \in \mathcal{M}(S)\cap \operatorname{B}(\mathbf{s})$ such that $\mathbf{m}_1=\mathbf{m}, \mathbf{m}_h=\mathbf{n}$ and $\mathbf{m}_i \vee \mathbf{m}_{i+1}\in \mathcal{A}(S)$ for all $i\in \{1,\ldots,h-1\}$.
    Consider the elements $\mathbf{m}_1=\mathbf{m},\mathbf{m}_2, \ldots, \mathbf{m}_{r-1},\mathbf{m}_{r}=\mathbf{n}\in \mathcal{M}$ such that $\mathbf{m}_i \preceq \mathbf{m}_{i+1}$ and $\mathbf{m}_i$ covers $\mathbf{m}_{i+1}$ in $\mathcal{M}(S)$ with respect to $\preceq$, for all $i\in \{1,\ldots,r-1\}$. Observe that, for all $i\in \{1,\ldots,r-1\}$, the second coordinate of $\mathbf{m}_i$ is smaller than the second coordinate of $\mathbf{m}$ and the first coordinate of $\mathbf{m}_i$ is smaller than the first coordinate of $\mathbf{n}$, so $\mathbf{m}_i\in \operatorname{B}(\mathbf{s})$. 
    We can also assume that $\mathcal{M}(S)\neq \mathcal{E}(S)$, otherwise we obtain our claim by Remark~\ref{rem:betti-base-case}. This implies that for all $i\in \{1,\ldots,r-1\}$, we have $\mathbf{m}_i \cdot \mathbf{m}_{i+1}\neq 0$, and by Lemma~\ref{lem:sup-is-atom},  $\mathbf{m}_i \vee \mathbf{m}_{i+1}$ is an atom. Therefore, by Corollary~\ref{cor:antichain-of-supatoms} we can conclude that $\mathbf{s}$ is not a Betti element, forcing $\ell(\mathbf{s})=2$. As $\mathbf{s}$ is a Betti element, we also have that $|\mathsf{Z}(\mathbf{s})|\ge 2$.
\end{proof}

As a consequence of Corollary~\ref{cor:catenary-three}, we deduce that if $S$ is a gap absorbing submonoid of $\mathbb{N}^2$, then the catenary degree of $S$ is at most three.

\section{Further research}

We briefly summarize some of the open questions that arose in the preceding sections. Some were already stated in the form of conjectures.

\begin{enumerate}[({Q}1)]
    \item Is every ideal extension gap absorbing? (see Conjectures~\ref{conj:ideal-implies-gap-absorbing} and \ref{conj:ideal-2A-closed-intervals}). 
    \item For every ideal extension, is the minimal length of a Betti degree at most two? We prove in Theorem~\ref{thm:lge4-not-betti} that the minimal length of a Betti degree in a gap absorbing monoid is at most three. We have the impression that this also holds for ideal extensions of free monoids. Furthermore, we did not find any example where a Betti degree has minimal length equal to three.
    \item Is the catenary degree of an ideal extension at most three? Theorem~\ref{th:bound-catenary} states that the catenary degree of a gap absorbing monoids is at most four. We have not found any ideal extension of a free numerical semigroup with catenary degree equal to four. The fact that every ideal extension of a free commutative monoid has at most catenary degree three would also prove \cite[Conjecture~4.16]{Baeth}: for every element $\mathbf{s}$ on a finite-complement ideal $S$ of a free-monoid, $\mathsf{L}(\mathbf{s})$ is an interval. We believe that Theorem~\ref{th:Ls-interval} also holds for ideal extensions of free commutative monoids. Also, notice that in light of Corollary~\ref{cor:catenary-three}, (Q2) implies (Q3) for gap absorbing monoids.

    \item For the $\omega$-primality, if $S$ is an ideal extension we know by Theorem~\ref{thm:upper-bound-omega-norm1} that $\omega(S)$ is upper bounded by the supremum of 1-norms of its atoms plus one. We have not found any example where this upper bound is attained. 

    \item In Section~\ref{sec:back-slash}, given a sequence $\lambda=(\lambda_i)_{i\in I}$ of non-negative integers and a numerical semigroup $T$, we introduce the class of monoids $S_\lambda^I(T)$, focusing on the case $\lambda_j=1$ for all $j\in J$ and $\lambda_i=0$ for all $i\in I\setminus J$, with $J$ a subset of $I$. In this particular case, by Proposition~\ref{prop:lengths-S_J^I(T)} we know that $S_\lambda^I(T)$ inherits from $T$ the values of elasticity and lenght-density. In the case $T$ is an ordinary numerical semigroup, we also study catenary degree and $\omega$-primality. We ask if it is possible to find the values of these two invariants in the case $T$ is not ordinary. In general, it could be interesting to study monoids of kind $S_\lambda^I(T)$ also for different occurrences of the sequence $\lambda$, asking if some properties and invariants of $S_\lambda^I(T)$ can be expressed in terms of properties and invariants of $T$.
\end{enumerate}

\section*{Acknowledgements}

The first author acknowledges support from the Institute of Mathematics of the University of Granada (IMAG) through the program of Visits of Young Talented Researchers and from Istituto Nazionale di Alta Matematica (INDAM) through the program Concorso a n. 30 mensilità di borse di studio per l’estero per l’a.a. 2022-2023.

The second and third authors are partially supported by the grant number ProyExcel\_00868 (Proyecto de Excelencia de la Junta de Andalucía) and by the Junta de Andaluc\'ia Grant Number FQM--343. The second author acknowledges financial support from the grant PID2022-138906NB-C21 funded by MCIN/ AEI/10.13039/501100011033 and by ERDF ``A way of making Europe'', and from the Spanish Ministry of Science and Innovation (MICINN), through the ``Severo Ochoa and María de Maeztu Programme for Centres and Unities of Excellence'' (CEX2020-001105-M).

The authors would like to thank S. T. Chapman for his comments and suggestions, and A. Geroldinger for drawing our attention to \cite[Theorem~7.6.9]{g-hk}.


\begin{thebibliography}{99}
%
%

\bibitem{a-c-al} J. Amos, S. T. Chapman, N. Hine, J. Paixão, Sets of lengths do not characterize numerical monoids, Integers \textbf{7} (2007), A50, 8 pp.

\bibitem{algorithm_omega}  D. F. Anderson, S. T. Chapman, N. Kaplan, and D. Torkornoo, An algorithm to compute $\omega$-primality in a numerical monoid, Semigroup Forum 82 (2011) 96–108.

\bibitem{Num-sem-book} A. Assi, M. D'Anna, P. A. Garc{\'\i}a-S\'anchez, Numerical Semigroups and Applications, 2nd ed., RSME Springer Series 3, Springer:
Cham, Switzerland, (2020).

\bibitem{Baeth} N. Baeth, Complement-Finite Ideals. In: Chabert, JL., Fontana, M., Frisch, S., Glaz, S., Johnson, K. (eds) Algebraic, Number Theoretic, and Topological Aspects of Ring Theory. Springer, Cham, (2023).

\bibitem{bgsg} V. Blanco, P. A. García-Sánchez, A. Geroldinger, Semigroup-theoretical characterizations of arithmetical invariants with applications to numerical monoids and Krull monoids, Illinois J. Math. \textbf{55} (2011), 1385--1414.

\bibitem{bckm} C. Brower, S. Chapman, T. Kulhanek, J. McDonough, C. O’Neill, V. Pavlyuk, V. Ponomarenko, Length Density and Numerical Semigroups. In: Nathanson, M.B. (eds) Combinatorial and Additive Number Theory V. CANT 2021. Springer Proceedings in Mathematics \& Statistics, vol 395 (2023). Springer, Cham. 

\bibitem{bgs} M. Bullejos, P. A. García-Sánchez, Minimal presentations for monoids with the ascending chain condition on principal ideals, Semigroup Forum \textbf{85} (2012), 185--190.
        
\bibitem{maxdelta} S. T. Chapman, P. A. García-Sánchez, D. Llena, A. Malyshev, D. Steinberg, On the Delta set and the Betti elements of a BF-monoid, Arab J Math \textbf{1} (2012), 53--61.

\bibitem{cgspr} S. T. Chapman, P. A. García-Sánchez, D. Llena, V. Ponomarenko, and J. C. Rosales, The catenary and tame degree in finitely generated commutative cancellative monoids, Manuscripta Math. \textbf{120} (2006), 253--264. 

\bibitem{GNS} C. Cisto, G. Failla, R. Utano, On the generators of a generalized numerical semigroup. Analele Univ. ``Ovidius'', \textbf{27}(1) (2019),49--59.

\bibitem{classesGNS} C. Cisto, F. Navarra, On some classes of generalized numerical semigroups, arxiv: 2212.12467, (2022).

\bibitem{numericalsgps}  M. Delgado,  P. A. Garcia-Sanchez, J.  Morais, NumericalSgps, A package for numerical semigroups, Version 1.3.1 dev (2023), Refereed GAP package, \url{https://gap-packages.github.io/numericalsgps}.

\bibitem{gap} The GAP group,
GAP -- Groups, Algorithms, and Programming, Version 4.12.2, (2022), \url{https://www.gap-system.org}.

\bibitem{overview} P. A. García-Sánchez, An overview of the computational aspects of nonunique factorization invariants, in: Multiplicative Ideal Theory and Factorization Theory - Commutative and Non-Commutative Perspectives, Springer Proceedings in Mathematics \& Statistics 170 (2016), Springer, pp. 159--181.

\bibitem{gsonw} P. A. García-Sánchez, C. O'Neill, G. Webb, On the computation of factorization invariants for affine semigroups, J. Algebra Appl. 18 (2019) 1950019.

\bibitem{g-hk} A. Geroldinger, F. Halter-Koch, Non--unique factorizations. Algebraic, combinatorial and analytic theory. Pure and Applied Mathematics (Boca Raton), \textbf{278}. Chapman \& Hall/CRC, Boca Raton, FL, (2006).

\bibitem{gk} A. Geroldinger and F. Kainrath, On the arithmetic of tame monoids with applications
to Krull monoids and Mori domains, J. Pure Appl. Algebra \textbf{214} (2010), 2199--2218.

\bibitem{kl} F. Kainrath, G. Lettl, Geometric Notes on monoids. Semigroup Forum \textbf{61} (2000), 298--302.
  
\bibitem{Li} S. Li, On the number of generalized numerical semigroups, arXiv:2212.13740 (2023)

\bibitem{ph} A. Philipp,
A characterization of arithmetical invariants by the monoid of relation,
Semigroup Forum \textbf{81} (2010), 424--434.
%
%

\bibitem{fg} J. C. Rosales y P. A. García-Sánchez, Finitely generated commutative monoids, Nova Science Publishers, Inc., New York, 1999.
\end{thebibliography}
\end{document}